\documentclass{amsart}
\usepackage{amsmath,amsthm,amssymb,amsfonts,latexsym,graphics,enumerate}

\newcommand{\vg}[1]{}
\newcommand{\tu}[1]{\textup{#1}}
\newcommand{\ie}{{\rm i.e.,} }
\newcommand{\cf}{{\rm cf.~}}

\newcommand{\baire}{\ensuremath{\ca{N}}}
\newcommand{\cantor}{\ensuremath{2^\omega}}

\newcommand{\pr}{\textrm{pr}}

\renewcommand{\qedsymbol}{$\dashv$}

\newcommand{\om}{\ensuremath{\omega}}

\newcommand{\ep}{\ensuremath{\varepsilon}}
\newcommand{\ck}{\ensuremath{\om_1^{\rm CK}}}

\newcommand\tboldsymbol[1]{%
\protect\raisebox{0pt}[0pt][0pt]{%
$\underset{\widetilde{}}{\boldsymbol{#1}}$}\mbox{\hskip 1pt}}

\newcommand{\bolds}{\tboldsymbol{\Sigma}}
\newcommand{\boldp}{\tboldsymbol{\Pi}}
\newcommand{\boldd}{\tboldsymbol{\Delta}}

\newcommand{\del}{\ensuremath{\Delta^1_1}}
\newcommand{\sig}{\ensuremath{\Sigma^1_1}}
\newcommand{\pii}{\ensuremath{\Pi^1_1}}

\newcommand{\ca}[1]{\ensuremath{\mathcal{#1}}}
\newcommand{\set}[2]{\ensuremath{\{#1 \hspace{0.3mm} : \hspace{0.3mm} #2\}}}

\newcommand{\bij}{\mbox{\,$\rightarrowtail\kern -.8em \rightarrow$\,}}
\newcommand\surj{\twoheadrightarrow}

\newcommand{\teq}{\ensuremath{\equiv_{\rm T}}}
\newcommand{\tleq}{\ensuremath{\leq_{\rm T}}}
\newcommand{\tgeq}{\ensuremath{\geq_{\rm T}}}

\newcommand{\rfn}[1]{\ensuremath{\{#1\}}}
\newcommand{\tpower}[1]{\ensuremath{^{(#1)}}}

\newcounter{quest}
\newcommand{\cn}[2]{\ensuremath{#1 \ \hat{} \ #2}}

\newcommand{\univ}{{\rm G}}
\newcommand{\dec}{{\rm dec}}
\newcommand{\eg}{{\rm e.g.} }

\newcommand{\BC}{\mathrm{BC}}

\newcommand{\bcode}{\mathrm{B}}
\newcommand{\ckr}[1]{\ensuremath{\om_1^{#1}}}

\newcommand{\omseq}{\ensuremath{\om^{<\om}}}
\newcommand{\Dom}{{\rm Dom}}

\renewcommand{\ie}{~\text{i.e.,} }

\newcommand{\goodt}{S}

\newcommand{\xx}{\mathcal{X}}
\newcommand{\yy}{\mathcal{Y}}

\newcommand{\fr}{\mbox{}^\smallfrown}

\newcommand{\rs}[1]{\upharpoonright{#1}}
\newcommand{\hyp}{{\sf hyp}}
\newcommand{\nbase}[1]{{\rm Nbase}({#1})}
\newcommand{\partialf}{\rightharpoonup}
\newcommand{\genname}{\dot{\mathbf{r}}_{gen}}

\newcommand{\effuniv}{{\rm H}}

\renewcommand{\cn}[2]{#1 \fr #2}

\newtheorem*{claim}{Claim}

\title[Degrees and Decomposability]{Turing degrees in Polish spaces and decomposability of Borel functions}

\author{Vassilios Gregoriades}
\address[Vassilios Gregoriades]{Department of Mathematics ``"Giuseppe Peano", University of Turin, Italy}
\email{vassilios.gregoriades@unito.it}

\author{Takayuki Kihara}
\address[Takayuki Kihara]{Department of Mathematics, The University of California, Berkeley, United States}
\email{kihara@math.berkeley.edu}

\author{Keng Meng Ng}
\address[Keng Meng Ng]{Division of Mathematical Sciences, School of Physical and Mathematical Sciences, Nanyang Technological University, Singapore}
\email{kmng@ntu.edu.sg}

\date{\today}



\theoremstyle{plain}
\newtheorem{theorem}{Theorem}[section]
\newtheorem{lemma}[theorem]{Lemma}
\newtheorem{proposition}[theorem]{Proposition}
\newtheorem{corollary}[theorem]{Corollary}
\newtheorem{fact}[theorem]{Fact}
\newtheorem*{decconjecture}{The Decomposability Conjecture}
\newtheorem{conjecture}[theorem]{Conjecture}
\newtheorem{myquestion}[theorem]{Question}

\theoremstyle{definition}
\newtheorem{definition}[theorem]{Definition}
\newtheorem{remark}[theorem]{Remark}
\newtheorem{example}[theorem]{Example}

\begin{document}

\maketitle

\begin{abstract}
We give a partial answer to an important open problem in descriptive set theory, the Decomposability Conjecture for Borel functions on an analytic subset of a Polish space to a separable metrizable space. Our techniques employ deep results from effective descriptive set theory and recursion theory. In fact it is essential to extend several prominent results in recursion theory (\eg the Shore-Slaman Join Theorem) to the setting of Polish spaces. As a by-product we give both positive and negative results on the Martin Conjecture on the degree preserving Borel functions between Polish spaces.
Additionally we prove results about the transfinite version as well as the computable version of the Decomposability Conjecture, and we explore the idea of applying the technique of turning Borel-measurable functions into continuous ones.
\end{abstract}

\section{Introduction}

\label{section introduction}

The study of decomposability of Borel functions originated from Luzin's famous old problem:
Is every Borel function decomposable into countably many continuous functions?
This problem has been solved in negative.
Indeed, even a Baire class one (i.e., $F_\sigma$-measurable) function is not necessarily decomposable.
This result directs our attention to a finer hierarchy of Borel functions than the Baire hierarchy.
A function $f:\ca{X} \to \ca{Y}$ between topological spaces is a {\em Borel function at level $(\eta,\xi)$} (denoted by $f^{-1} \bolds^0_{1+\eta}\subseteq \bolds^0_{1+\xi}$) if for all $\bolds^0_{1+\eta}$ sets $A \subseteq\ca{Y}$ the preimage $f^{-1}[A]$ is a $\bolds^0_{1+\xi}$ subset of $\ca{X}$. Essentially the same notion was introduced by Jayne \cite{Jayne74} to show Baire class variants of the Banach-Stone Theorem and the Gel'fand-Kolmogorov Theorem in functional analysis.
Later, Jayne and Rogers \cite{jayne_rogers_first_level_Borel_functions_and_isomorphisms} discovered a deep connection between the first level of this fine hierarchy and decomposability, now known as the \emph{Jayne-Rogers Theorem}:
A function $f:\ca{X}\to\ca{Y}$ from an analytic subset $\ca{X}$ of a Polish space into a separable metrizable space $\ca{Y}$ is a first-level Borel function (i.e., $f^{-1} \bolds^0_{2}\subseteq \bolds^0_{2}$) if and only if it is decomposable into countably many continuous functions with closed domains.

In recent years, researchers have made remarkable progress on extending the Jayne-Rogers theorem (\cf \cite{MilCar15a,MilCar15b,Debs14,KaMoSe,kihara_decomposing_Borel_functions_using_the_Shore_Slaman_join_theorem,motto_ros_on_the_structure_of_finite_level_and_omega_decomposable_Borel_functions,pauly_debrecht_nondeterministic_computation_and_the_jayne_rogers_theorem,pawlikowski_sabok_decomposing_Borel_functions,semmes_phd_thesis,Solecki98}). One prominent result among them is by Semmes \cite{semmes_phd_thesis}, who showed that a function $f:\baire\to\baire$ on the Baire space $\baire:=\omega^\omega$ is a second-level Borel function (i.e., $f^{-1} \bolds^0_{3}\subseteq \bolds^0_{3}$) if and only if it is decomposable into countably many continuous functions with $G_\delta$ domains, and that a function $f:\baire\to\baire$ is a Borel function at level $(1,2)$ (i.e., $f^{-1} \bolds^0_{2}\subseteq \bolds^0_{3}$) if and only if it is decomposable into countably many $F_\sigma$-measurable functions with $G_\delta$ domains.
These results lead researchers to the conjecture that the Jayne-Rogers Theorem can be extended to all finite levels of the Borel hierarchy.
This remains an open problem, which is receiving increasing attention by researchers in the area.

To be more precise, given a function $f:\ca{X} \to \ca{Y}$ between topological spaces we write $f \in \dec(\bolds^0_m)$ if there exists a partition $(\ca{X}_i)_{i \in \om}$ of $\ca{X}$ such that the restriction $f \upharpoonright\ca{X}_i$ is $\bolds^0_m$-measurable for all $i \in \om$.
We also write $f \in \dec(\bolds^0_m,\boldd^0_n)$ if such a partition $(\ca{X}_i)_{i\in\omega}$ can be $\boldd^0_n$ subsets of $\ca{X}$.
It is clear that $f \in \dec(\bolds^0_m,\boldd^0_n)$ exactly when there exists a cover $(\ca{Z}_i)_{i \in \om}$ of $\ca{X}$ consisting of (not necessarily disjoint) $\boldp^0_{n-1}$ sets such that $f \upharpoonright \ca{Z}_i$ is $\bolds^0_m$-measurable for all $i \in \om$.

It is not hard to verify that if $f \in \dec(\bolds^0_{n-m+1},\boldd^0_n)$ then condition $f^{-1} \bolds^0_m \subseteq \bolds^0_n$ is also satisfied as well.
The problem is whether the converse is also true.

\begin{decconjecture}[\cf \cite{andretta_the_slo_principle_and_the_Wadge_hierarchy,motto_ros_on_the_structure_of_finite_level_and_omega_decomposable_Borel_functions,pawlikowski_sabok_decomposing_Borel_functions}]
\label{conjecture the decomposability conjecture}
Suppose that $\ca{X}$ is an analytic subset of a Polish space and that $\ca{Y}$ is separable metrizable.\footnote{Since one can always replace \ca{Y} with its completion, we may assume without loss of generality that \ca{Y} is Polish.} For every function $f:\ca{X}\to\ca{Y}$ and every $n \geq 1$ it holds
\begin{align*}
\hspace*{-20mm}{\rm(flat \ case)}\hspace*{5mm}f^{-1} \bolds^0_n \subseteq \bolds^0_n \ \iff& \ f \in \dec(\bolds^0_1,\boldd^0_n).
\end{align*}
More generally for all $1 \leq m \leq n$ we have
\begin{align*}
\hspace*{12mm}f^{-1} \bolds^0_m \subseteq \bolds^0_n \ \iff& \ f \in \dec(\bolds^0_{n-m+1},\boldd^0_n).
\end{align*}
\end{decconjecture}

It is clear that the case $n=m=2$ in the Decomposability Conjecture is the Jayne-Rogers Theorem. The case $m\leq n\leq 3$ has been answered positively by Semmes for functions $f: \baire \to \baire$ \cite{semmes_phd_thesis} as mentioned above.
In this article we give the following partial answer.\footnote{At the end of this article in Section \ref{sec:last} we pose a question, which if answered affirmatively, solves the Decomposability Conjecture. An interesting aspect of this question is that it is a purely recursion-theoretic statement, bringing thus the problem closer to many more researchers.}

\begin{theorem}
\label{theorem main}

Suppose that $n \geq m \geq 1$, $\ca{X}$, $\ca{Y}$ are Polish spaces, $\ca{A}$ is an analytic subset of $\ca{X}$, and that $f: \ca{A} \to \ca{Y}$ is such that $f^{-1} \bolds^0_m \subseteq \bolds^0_n$. Then it holds $f \in  \dec(\bolds^0_{n-m+1},\boldd^0_{n+1})$.

If moreover $m \geq 3$ and $f$ is $\bolds^0_{n-1}$-measurable, then $f \in  \dec(\bolds^0_{n-m+1},\boldd^0_n)$, \ie the Decomposability Conjecture is true for functions, whose level of measurabilility is one step below than the one of the assumption. The same conclusion holds if the graph of $f$ is $\bolds^0_m$ with $m \geq 3$.
\end{theorem}

Besides the preceding theorem, we establish a collection of results connecting the Turing degree theory with Polish spaces. A notable work on this topic is done by Miller with his introduction of \emph{continuous degrees} \cf \cite{miller2}, which extends the notion of the Turing degree. In this article we extend the Turing-jump to all Polish spaces, and we show that this jump operation is well-defined on the degree structure of the continuous degrees.
Each statement in the language of Turing degrees (with the Turing jump operation) has a natural ``translation" to the language of continuous degrees with our jump operation. In other words this yields an extension of the former theory to the latter.

As an example of the preceding extension we prove the \emph{Shore-Slaman Join Theorem} for continuous degrees, which is essential for the proof of Theorem \ref{theorem main}. This is probably the first application of continuous degrees in an area which is not directly related to recursion theory. The Shore-Slaman Join Theorem is also related to a well-known open problem in recursion theory, the \emph{Martin Conjecture}. In the same fashion we state the Martin Conjecture for continuous degrees. We explore the latter in this extended setting and we show that the continuous-degree version of the Shore-Slaman Join Theorem has the analogous connections.

We think that the preceding work has the potential to form a basis for a new branch of recursion theory in Polish spaces.

Going back to Theorem \ref{theorem main}, we note that it generalizes (and gives a new proof of) previous results of Motto Ros \cite{motto_ros_on_the_structure_of_finite_level_and_omega_decomposable_Borel_functions} and Pawlikowski-Sabok \cite{pawlikowski_sabok_decomposing_Borel_functions}. The latter results are the assertions of Theorem \ref{theorem main} in the flat case ($n=m$).\footnote{Note that in the flat case, the hypothesis that the graph of $f$ is $\bolds^0_n$ is weaker than that of $f$ being $\bolds^0_{n-1}$-measurable. The result of Motto Ros in the flat case is actually under this weaker hypothesis. We nevertheless mention that Pawlikowski-Sabok have a separate result of this type, in particular they prove that the flat case of the Decomposability Conjecture holds for the functions, which are open and injective.}
In fact our techniques are substantially different from the ones of \cite{motto_ros_on_the_structure_of_finite_level_and_omega_decomposable_Borel_functions,pawlikowski_sabok_decomposing_Borel_functions}. More specifically we apply tools from {\em effective descriptive set theory} and (as suggested above) {\em recursion theory}.\footnote{It is not unusual for the areas of recursion theory and effective descriptive set theory to have applications in classical theory, \ie they can give results whose statement does not involve any recursion-theoretic notions. The applications of the latter areas are though largely independent. It is perhaps worth noting that --to our best knowledge-- this is the only application, where both of these areas are put together in order to make progress on a problem in classical theory.}

Our motivation is a previous work of Kihara on the Decomposability Conjecture \cite{kihara_decomposing_Borel_functions_using_the_Shore_Slaman_join_theorem}. We nevertheless make significant improvements to the main theorem of the latter in removing its {\em uniformity} and {\em dimension-theoretic} restrictions. Actually our result is the only one in the non-flat case ($n>m$) of the Decomposability Conjecture without any dimension-theoretic restrictions on the given spaces.

The preceding constraints are indeed undesirable from the viewpoint of Jayne's work \cite{Jayne74} in functional analysis. Given a separable metrizable space $\ca{X}$ and a countable ordinal $\xi$ we denote by $\mathcal{B}^*_\xi(\mathcal{X})$ the Banach algebra of bounded real-valued Baire-class $\xi$ functions on $\mathcal{X}$ equipped with the supremum norm and the pointwise operation. Jayne showed that two given separable metrizable spaces $\mathcal{X}$ and $\mathcal{Y}$ are Borel isomorphic\footnote{In fact Jayne showed the result for realcompact spaces $\mathcal{X}$ and $\mathcal{Y}$ and with Baire isomorphisms instead of Borel.} at level $(\eta,\xi)$
if and only if $\mathcal{B}^*_\xi(\mathcal{X})$ and $\mathcal{B}^*_\eta(\mathcal{Y})$ are linearly isometric (ring isomorphic, etc.), so in particular, the classification problem on Banach algebras of the forms $\mathcal{B}^*_\xi(\mathcal{X})$ is equivalent to the {\em $\xi$-th level Borel isomorphism problem} (see \cite{KihPau}).

On the other hand, the above mentioned result of Motto Ros and Pawlikowski-Sabok is essential to show that every {\em $n$-th level Borel isomorphism \ie $\boldd^0_{n+1}$-isomorphism, is covered by countably many partial homeomorphic maps}. Using this, Kihara-Pauly \cite{KihPau} solved the $n$-th level Borel isomorphism problem, and hence, a problem on the linear-isometric (ring-theoretic, etc.)~classification of Banach algebras of the forms $\mathcal{B}^*_n(\mathcal{X})$. We note that it is essential here to consider Polish spaces of not necessarily transfinite inductive dimension, since the Banach algebra $\mathcal{B}^*_\xi(\mathcal{X})$ for $\xi\geq 2$ becomes trivial under linear-isometric (ring-theoretic, etc.)~classification whenever $\mathcal{X}$ is a Polish space of transfinite inductive dimension (see \cite{JayRog79a}).

Since the preceding $n$-th level Borel-isomorphism problem is solved using the decomposability result of Motto Ros and Pawlikowski-Sabok, and since the latter is covered by Theorem \ref{theorem main}, we obtain a \emph{completely computability-theoretic} solution to the problem on the $n$-th level Borel-isomorphism.

Finally we remark that Theorem \ref{theorem main} reduces the Decomposability Conjecture to a simpler one. (The similar remark is made by Motto Ros \cite[Corollary 5.11]{motto_ros_on_the_structure_of_finite_level_and_omega_decomposable_Borel_functions} in the flat case using the analogous result.)

\begin{corollary}
\label{corollary DC reduces to m=2}
Suppose that $\ca{X}$, $\ca{Y}$ are Polish spaces, $n \geq 2$, and that $\ca{A}$ is an analytic subset of $\ca{X}$. If for all functions $f: \ca{A} \to \ca{Y}$ the implication
\[
f^{-1} \bolds^0_2 \subseteq \bolds^0_n \ \Longrightarrow \ f \in \dec(\bolds^0_{n-1},\boldd^0_n)
\]
holds, then for all naturals $m$ with $2 \leq m \leq n$ and all functions $f: \ca{A} \to \ca{Y}$ the implication
\[
f^{-1} \bolds^0_m \subseteq \bolds^0_n \ \Longrightarrow \ f \in \dec(\bolds^0_{n-m+1},\boldd^0_n)
\]
holds as well.
\end{corollary}

This result suggests that the cases $m=2 \leq n$ in the Decomposability Conjecture form the ``correct'' extension of the Jayne-Rogers Theorem.

\begin{proof}
Let $\emptyset \neq \ca{A} \subseteq \ca{X}$ be analytic and suppose that $f: \ca{A} \to \ca{Y}$ satisfies $f^{-1} \bolds^0_m \subseteq \bolds^0_n$ for some $3 \leq m \leq n$ (the case $m=2$ is covered by our hypothesis). In particular $f$ satisfies the condition $f^{-1} \bolds^0_2 \subseteq \bolds^0_n$. So from our hypothesis there exists a sequence $(\ca{B}_i)_{i \in \om}$ of $\boldp^0_{n-1}$ subsets of $\ca{X}$ such that the restriction $f \upharpoonright \ca{B}_i \cap \ca{A}$ is $\bolds^0_{n-1}$-measurable for all $i \in \om$. (Here we use that the $\boldp^0_\xi$ subsets of $\ca{A}$ are exactly the sets of the form $\ca{B} \cap \ca{A}$ for some $\boldp^0_\xi$ subset $\ca{B}$ of $\ca{X}$.)

Put $g_i = f \upharpoonright \ca{B}_i \cap \ca{A}$ for all $i$. It is clear that every function $g_i$ satisfies the condition $g_i^{-1} \bolds^0_m \subseteq \bolds^0_n$ and that the sets $\ca{B}_i \cap \ca{A}$ are analytic. Moreover every function $g_i$ is $\bolds^0_{n-1}$-measurable. Using that $m \geq 3$ it follows from the second assertion of Theorem \ref{theorem main} that for all $i$ there exists a sequence $(\ca{C}^i_j)_{j \in \om}$ of $\boldp^0_{n-1}$ subsets of $\ca{X}$ such that $g_i \upharpoonright \Dom(g_i) \cap \ca{C}^i_j = f \upharpoonright \ca{A} \cap \ca{B}_i \cap \ca{C}^i_j$ is $\bolds^0_{n-m+1}$-measurable. Clearly every $\ca{A} \cap \ca{B}_i \cap \ca{C}^i_j$ is a $\boldp^0_{n-1}$ subset of $\ca{A}$ and so $f \in \dec(\bolds^0_{n-m+1},\boldd^0_n)$.
\end{proof}

We proceed with the notation and some definitions.

\subsection*{Notation}

We review some basic notations (see also Moschovakis \cite{yiannis_dst} and Cooper \cite{CooperBook}).
We use the symbol $\baire$ to denote Baire space $\omega^\omega$ endowed with the usual product topology.
We will often identify sets with relations, that is, $P(x)$ means $x \in P$ for $P \subseteq X$. For every $P \subseteq X \times Y$ and every $x \in X$ the $x$-section of $P$
 is denoted by $P_x = \set{y \in Y}{P(x,y)}$.
By \omseq \ we mean the set of all finite sequences in \om.
For $i,n \in \om$ and $\alpha \in \baire$ we put $(\alpha)_i(n)=\alpha(\langle i,n \rangle)$, where $\langle \cdot \rangle$ is a recursive injection from $\omseq$ to $\om$.
The \emph{concatenation} of two sequences $u$ and $v$ (where $u$ is finite and $v$ is perhaps infinite) is denoted by $\cn{u}{v}$. In case where $u$ is a sequence of length $1$, say $u=\langle u_0\rangle$, we simply write $\cn{u_0}{v}$.
Given two sequences $u,v$, we write $u\preceq v$ if $u$ is an initial segment of $v$.

For $x,y\in\baire$, the join $x\oplus y$ is defined by $(x\oplus y)(2n)=x(n)$ and $(x\oplus y)(2n+1)=y(n)$ for each $n\in\om$.
We will often identify subsets of the natural numbers with members of the Cantor space.
Under this identification, for instance,  the join $A \oplus B$ for $A, B \subseteq \om$ representes the set $\{2n:n\in A\}\cup\{2n+1:n\in B\}$.
Partial functions will be denoted by $f: X \rightharpoonup Y$. By $f(x) \downarrow$ we mean that $f$ is defined on $x$.
The domain of $f$ is denoted by $\Dom(f)$.

The $e$-th partial recursive function on the natural numbers $\omega$ is denoted by $\Phi_e$.
The $\Phi_e^\alpha(n)$ (also written as $\Phi_e(\alpha)(n)$) is the result of the $\Phi_e$-computation with an oracle $\alpha \in\omega^{<\omega}\cup\baire$ and an input $n\in\omega$.
As usual, $\Phi_e$ can be thought of as a partial function on Baire space $\baire$.
Given $A, B \subseteq \om$ we write $A \tleq B$ to denote that $A$ is $B$-recursive, that is, $A=\Phi_e^B$ for some $e\in\omega$.
The Turing jump of $A$ is defined by $A'=\{e\in\omega:\Phi_e^A(e)\downarrow\}$.
By $\ckr{A}$ we mean the least non $A$-recursive ordinal and by $A\tpower{\xi}$ the $\xi$-th Turing jump of $A$ for $\xi < \ckr{A}$.

We will also employ the {recursively presented metric spaces} (also with respect to some parameter $\ep$) as well as the lightface pointclasses, \eg $\Sigma^0_n(\ep)$ from effective descriptive set theory \cf \cite{yiannis_dst}.
A \emph{recursively presented metric space} is a triple $\xx\equiv(\ca{X},d,\bar{a})$ consisting of  a Polish space $\ca{X}$, a metric $d$ on $\ca{X}$ and a recursive presentation $\bar{a}=(a_n)_{n\in\omega}$ of $(\xx,d)$ (i.e., a countable dense subset of $\ca{X}$ satisfying certain effectivity conditions; see \cite{yiannis_dst}).
A recursive presentation involves a canonical (countable) basis $(B^\xx_{n,r})_{n,r \in \om}$ of balls $B^\xx_{n,r}=\{x\in \ca{X}:d(x,a_n)<q_r\}$, where $(q_r)_{r \in \om}$ is a fixed recursive enumeration of $\mathbb{Q}^+$. We may assume that the balls of a finite product of spaces is a product of balls.

\subsection*{Universal sets and uniformity}

Suppose that \ca{Y} and \ca{Z} are separable metric spaces and that $\Gamma$ is a pointclass.  We denote by $\Gamma \upharpoonright \ca{Y}$ the family of all subsets of \ca{Y} which are in $\Gamma$.
A set $G \subseteq \ca{Z} \times \ca{Y}$ \emph{parametrizes} $\Gamma \upharpoonright \ca{Y}$ if for all $P \subseteq \ca{Y}$ we have that
\[
P \in \Gamma \iff \textrm{exists $z \in \ca{Z}$ such that} \ P = G_z.
\]
We think of $z$ as a \emph{$\Gamma$-code} for $P$.
The set $G \subseteq \ca{Z} \times \ca{Y}$ is \emph{universal} for $\Gamma \upharpoonright \ca{Y}$ if $G$ is in $\Gamma$ and parametrizes $\Gamma \upharpoonright \ca{Y}$.
A \emph{universal system} for $\Gamma$ is an assignment  $\ca{Y} \mapsto G^\ca{Y} \subseteq \baire \times \ca{Y}$ such that the set $G^\ca{Y}$ is universal for $\Gamma \upharpoonright \ca{Y}$ for all Polish spaces $\ca{Y}$.
Similarly we define the notion of a \emph{parametrization system} for $\Gamma$, the only difference being that the sets $G^\ca{Y}$ need not be in $\Gamma$.
See also Moschovakis \cite{yiannis_dst}.

In this article, we will use the {\em canonical} universal system of (lightface) Borel pointclasses in recursively presented metric spaces.
For a rank $n\in\om$, an oracle $\alpha\in\baire$ and an index $e\in\om$ we use the symbol $\effuniv^\om_{1,\alpha,e}$ to denote the $e$-th recursively enumerable subset of $\om$ relative to $\alpha$.
Suppose that a Polish space $\ca{Y}$ admits an $\ep$-recursive presentation $\bar{a}$, and hence, has a canonical basis $(B^\ca{Y}_{n,r})_{n,r\in\om}$.
For each $n\geq 1$ we inductively define $(\effuniv_n^{(\ca{Y},\ep)})_{n\in\om}$ as follows:
\begin{align*}
\effuniv_1^{(\ca{Y},\ep)}&=\{(e,\alpha,y)\in\om\times\baire\times\yy:(\exists\langle m,r\rangle\in \effuniv^\om_{1,\alpha,e})\;y\in B^\yy_{m,r}\},\\
\effuniv_{n+1}^{(\ca{Y},\ep)}&=\{(e,\alpha,y)\in\om\times\baire\times\yy:(\exists i\in\om)\;\neg \effuniv^{(\om\times\ca{Y},\ep)}_n(e,\alpha,i,y)\}.
\end{align*}
It is clear that the $\alpha$-th section $\effuniv^{(\ca{Y},\ep)}_{n,\alpha}$ parametrizes $\Sigma^0_n(\ep,\alpha)\upharpoonright \ca{Y}$.
The $e$-th section $\effuniv^{(\ca{Y},\ep)}_{n,\alpha,e}\subseteq\ca{Y}$ is usually called the {\em $e$-th $\Sigma^0_n(\alpha)$ subset of $\yy$}.
It follows that the set $\univ^{(\ca{Y},\ep)}_n \subseteq \baire \times \ca{Y}$ defined by
\[
\univ^{(\ca{Y},\ep)}_n(e\fr \alpha,y) \iff \effuniv^{(\ca{Y},\ep)}_n(e,\alpha,y)
\]
is $\Sigma^0_n(\ep)$ and universal for $\bolds^0_n \upharpoonright \ca{Y}$.
For every $n \geq 1$ \emph{we fix once and for all systems} $(\effuniv^{(\ca{Y},\ep)}_n)_{(\ca{Y},\ep)}$ as $(\univ^{(\ca{Y},\ep)}_n)_{(\ca{Y},\ep)}$ as above. When $\ep$ is understood from the context we simply write $\univ^{\ca{Y}}_n$ instead of $\univ^{(\ca{Y},\ep)}_n$ and similarly for $\effuniv_n$.

\subsection*{A brief summary of the remaining} We conclude this introductory section with a discussion which summarizes the remaining of this article and with some further results. Let us give first the following basic notion.
Suppose that $\Lambda,\Gamma_0,\Gamma_1$ are given pointclasses, $({\rm E}_i^\mathcal{Z})_{\mathcal{Z}}$ is a parametrization system for $\Gamma_i$ for each $i<2$.
Recall that for a function $f:\xx\to\yy$ between separable metrizable spaces, we write $f^{-1} \Gamma_0 \subseteq \Gamma_1$ to denote that $f^{-1}[A] \in \Gamma_1$ for all $A \in \Gamma_0$.

\begin{definition}\normalfont
\label{definition uniformly in the codes}
For a function $f:\xx\to\yy$, we say that the condition {\em $f^{-1} \Gamma_0 \subseteq \Gamma_1$ holds $\Lambda$-uniformly \tu{(}in the codes with respect to ${\rm E}^\ca{Y}_0, {\rm E}^\ca{X}_1$}\tu{)} if $f^{-1} \Gamma_0 \subseteq \Gamma_1$ holds, and moreover, there exists a $\Lambda$-measurable function $u: \baire \to \baire$ such that given an ${\rm E}^\ca{Y}_0$-code $\alpha$ of a $\Gamma_0$ set $A\subseteq\yy$, $u(\alpha)$ returns an ${\rm E}^\xx_1$-code of the $\Gamma_1$ set $f^{-1}[A]\subseteq\xx$.
More precisely, for all $\alpha$ and all $x \in\ca{X}$ we have that
\[
{\rm E}^{\ca{Y}}_0(\alpha,f(x)) \iff {\rm E}_1^{\ca{X}}(u(\alpha),x).
\]
If $\Lambda = \bolds^0_1$ we say ``continuous-uniformly'', and if $\Lambda$ is the class of all Borel sets we say ``Borel-uniformly''. Moreover we say ``recursive-uniformly" if $u$ can be chosen to be a recursive function.
\end{definition}

In the sequel we will focus on the cases $\Gamma_0 = \bolds^0_m$ and $\Gamma_1 = \bolds^0_n$, where $n,m \in \om$. Kihara characterized the family of functions on Polish spaces of transfinite inductive dimension, which satisfy the Decomposability Conjecture at some instances of $(n,m)$.

\begin{theorem}[Kihara \cite{kihara_decomposing_Borel_functions_using_the_Shore_Slaman_join_theorem}]
\label{theorem Kihara continuous transition}
Let $n\geq m\geq 1$.
For any function $f: \baire \to \baire$, we have the following:
\begin{enumerate}
\item[\tu{(}1\tu{)}] If $f \in \dec(\bolds^0_{n-m+1},\boldd^0_n)$ then $f^{-1} \bolds^0_m \subseteq \bolds^0_n$ holds continuous-uniformly.
\item[\tu{(}2\tu{)}] If $f^{-1} \bolds^0_m \subseteq \bolds^0_n$ holds continuous-uniformly, then $f \in \dec(\bolds^0_{n-m+1},\boldd^0_{n+1})$.
\end{enumerate}

Moreover, if $n < 2m-1$, then the condition $f^{-1} \bolds^0_m \subseteq \bolds^0_n$ holds continuous-uniformly if and only if  $f \in \dec(\bolds^0_{n-m+1},\boldd^0_n)$.
The similar assertion holds for spaces of transfinite inductive dimension.
\end{theorem}

Our strategy to obtain the decomposition of the function $f$ in Theorem \ref{theorem main} is similar to the one of \cite{kihara_decomposing_Borel_functions_using_the_Shore_Slaman_join_theorem} for Theorem \ref{theorem Kihara continuous transition}, but on the other hand there are three important points which require a more elaborated approach.

The first and perhaps most obvious point lies in the hypothesis: whereas the continuous-uniform transition in the codes is a natural phenomenon (see also \cite{brattka_effective_Borel_measurability_and_reducibility_of_function} and \cite{pauly_debrecht_nondeterministic_computation_and_the_jayne_rogers_theorem}), it is not unusual to have this condition as the result of some ``constructive'' argument rather as the hypothesis of a theorem, \cf the Suslin-Kleene Theorem. Therefore (given also our preceding comments on Jayne's work) it is somewhat restrictive to make a priori assumptions of this type.
It is a central aspect of the effective descriptive set theory to establish uniformity functions in a Borel way. By employing the \emph{Louveau separation} we prove

\begin{theorem}
\label{theorem Borel transition in the codes}
Suppose that $\ca{X}$, \ca{Y} are Polish spaces, $\ca{A} \subseteq \ca{X}$ is analytic and that $f:\ca{A} \to \ca{Y}$ satisfies $f^{-1} \bolds^0_m \subseteq \bolds^0_n$ for some $m, n \geq 1$. Then condition $f^{-1} \bolds^0_m \subseteq \bolds^0_n$ holds Borel-uniformly \tu{(}in the codes with respect to $\univ^{\ca{Y}}_m, \univ^{ \ca{X}}_n$\tu{)}.
\end{theorem}


The second point is essentially a consequence of the preceding. Having established a Borel rather than a continuous transition in the codes, it is evident that we need new arguments to obtain the decomposition. We do this by extending a result on canceling out Turing jumps, which is implicitly used in \cite{kihara_decomposing_Borel_functions_using_the_Shore_Slaman_join_theorem} and utilizes the Shore-Slaman Join Theorem. We call this extended result the \emph{Cancellation Lemma}, \cf Lemma \ref{lem:deg-Shore-Slaman}.
In order to prove this we employ an additional tool from recursion theory, namely the Friedberg Jump Inversion Theorem.

This method has however its limitations. The issue is that the arguments, which utilize {\em Turing degrees}, are only applicable in Polish spaces, which have transfinite inductive dimension. Kihara-Pauly \cite{KihPau} clarified the reason for such a dimension-theoretic restriction (e.g., in Kihara's argument \cite{kihara_decomposing_Borel_functions_using_the_Shore_Slaman_join_theorem}) by showing that a Polish space has transfinite inductive dimension if and only if its degree structure -as a substructure of the continuous degrees- can be identified with the Turing degrees up to an oracle.

This brings us to the next point. As we mentioned we utilize the theory of continuous degrees introduced by Miller \cite{miller2}, and more specifically we show that (a weak form of) the Shore-Slaman Join Theorem does hold for this degree structure in Polish spaces. On the other hand the Friedberg Jump Inversion Theorem does not go through in this setting. We will see though that there is a way of going around this obstacle by using the non-uniformity feature of the \emph{enumeration reducibility}, the \emph{almost totality} of continuous degrees \cite{miller2}, and some other tools from recursion theory.

We are also concerned with the transfinite version of the conjecture, and we will give the analogous version of Theorem \ref{theorem main} at levels $(\eta,\xi)$ with $\eta \leq \xi < \om_1$, this is Theorem \ref{theorem main transfinite}. It is sufficient to prove the latter result in zero-dimesional Polish spaces, since as proved by Kuratowski, every uncountable Polish space is $\boldd^0_\om$-isomorphic to the Baire space. This relieves us from the theory of continuous degrees, but on the other hand there are other subtle points: In order to apply the Louveau separation in the transfinite case, we need to consider a different parametrization system, which is induced by the notion of a \emph{Borel code} (Louveau-Moschovakis). We prove that the latter parametrization system is a good one, \cf The Good Parametrization Lemma for $\bcode^{\ca{X}}_\xi$ (Lemma \ref{lemma good parametrization for transfinite}) - this is the key element for the decomposition in the transfinite case.

Since the problem of decomposing a given function is also important in the area of computability theory we devote a brief section, where we extract the effective content from our proofs (Section \ref{section computable and continuous cases}).

In the subsequent section we deal with the continuous-degree version of Martin's Conjecture that we mentioned above. We will see that the picture in this setting is richer than the original one, but still there is some strong analogy between these two theories. In particular we extend a result of Slaman (Theorem \ref{thm:positive-Martin}), and on the other hand we prove that some parts of the extended conjecture fail (Theorem \ref{thm:failure-Martin}).

Finally we deal with the problem of the continuous-uniform transition in the codes and we give some results, which have a descriptive-set-theoretic interest in their own right. More specifically we assign to every function $f$ with $f^{-1}\bolds^0_m \subseteq \bolds^0_n$ a universal system for $\bolds^0_m$, under which the latter condition holds continuous-uniformly (Corollary \ref{corollary from continuous transition}). This answers a question of Kihara \cite{kihara_decomposing_Borel_functions_using_the_Shore_Slaman_join_theorem}, albeit partially, since this new universal system does not seem to be a good one. (This poses a serious restriction for obtaining the best decomposition even in the cases $m \leq n \leq 2m-1$.) As it is witnessed by Theorem \ref{theorem Borel transition in the codes} it couldn't be otherwise, unless we actually have a continuous transition in the original codes. We think that these observations underline the difficulty of the problem and that,  alongside our proposed plan at the end of this article, new approaches to this problem should be explored.

\section{Generalized Turing degree theory}

\label{section generalized Turing degree theory}

As mentioned in the introduction the difficulty for proving Theorem \ref{theorem main} in Polish spaces of not necessarily transfinite dimension is the lack of a sufficiently rich theory of a Turing-degree-like structure, which would allow us to carry out the proofs in this general setting. The aim of this section is to fill in this gap.

\subsection{Generalized Turing Reducibility}

Miller \cite{miller2} developed a nontrivial degree theory on computable metric spaces, called {\em continuous degrees}. We review the notion of continuous degrees in the framework of recursively presented metric spaces.

A {\em representation} of a set $\ca{X}$ is a surjection $\rho_\ca{X}$ from a subset of the Baire space $\baire$ onto the set $\ca{X}$.
Every $p\in\Dom(\rho_\ca{X})$ is called a {\em $\rho_\ca{X}$-name} (or simply, {\em name}) of the point $\rho_\ca{X}(p)\in\xx$.
This notion allows us to define the notion of computability in any represented set since we already have the notion of computability on the Baire space $\baire$.
More precisely, a partial function $f:\xx\partialf\yy$ is {\em recursive} if there is a partial recursive function $F:\baire\partialf\baire$ sending each name of $x\in{\rm Dom}(f)$ to a name of $f(x)$.

\begin{example}[Representation of a Second-Countable Space]
A recursive presentation $\overline{a}=(a_n)_{n\in\omega}$ of a separable metric space $(\mathcal{X},d)$ yields a representation by identifying each point $x\in\xx$ with enumerations of the \emph{encoded neighborhood basis} $\nbase{x}$ of $x\in\xx$:
\[\nbase{x}=\{\langle m,s \rangle \in\omega:x\in B_{m,s}^\xx\}.\]
That is, every point $x\in\mathcal{X}$ is coded by $p\in\baire$ such that ${\rm Rng}(p)=\nbase{x}$, and then we write $\rho_\xx(p)=x$.
\end{example}

\begin{example}[Representation of a Compact Metric Space]\label{exa:rep-Hilbert-cube}
It is known that a metrizable compactum $\yy$ admits a representation with a compact domain.
We give a sketch of an explicit construction of such a representation of Hilbert cube $\yy=[0,1]^\omega$.
For a positive rational $\ep>0$, an {\em $\ep$-cover} of a space $\yy$ is an open cover whose mesh is less than $\ep$ (that is, the diameter of each element in the cover is less than $\ep$).
By compactness, we have a sequence $(\mathcal{C}^n)_{n\in\omega}$ of $2^{-n-1}$-covers of $\yy$ with some bound $h\in\omega^\omega$ such that $\mathcal{C}^n=(B^n_m)_{m<h(n)}$  consists of $h(n)$-many rational open balls, and each $\mathcal{C}^{n+1}$ refines $\mathcal{C}^n$.
We can moreover assume that $h$ is computable and that $\{(l,k,n,m):\hat{B}^l_k\subseteq B^n_m\}$ is also computable where $\hat{B}^l_k$ is the closed ball having the same center and the radius with the open ball $B^l_k$.
Without loss of generality, we can also assume that for every $n,k<h(n)$, there is $l<h(n+1)$ such that $\hat{B}^{n+1}_l\subseteq B^{n}_k$.

We now define {\em the tree $\mathbf{H}$ of names \tu{(}of points in $\yy$\tu{)}} by the following way:
\begin{enumerate}
\item First put the empty string $\langle\rangle$ and all strings $\langle n\rangle$ of length $1$ (with $n<h(0)$) into $\mathbf{H}$.
\item For each $\sigma$ of length $l$ such that $\sigma\fr k\in\mathbf{H}$ we enumerate $\sigma\fr k\fr n$ into $\mathbf{H}$ if $n<h(l+1)$ and $\widehat{B}^{l+1}_n\subseteq B^{l}_k$.
\end{enumerate}

By our assumption, $\mathbf{H}$ is a recursively bounded recursive tree with no terminal node.
For each string $\sigma\in\mathbf{H}$, we write $B^\ast_\sigma$ for $B^{|\sigma|-1}_n$ where $n$ is the last value of $\sigma$.
Since $\mathcal{C}^{|\sigma|-1}$ is $2^{-|\sigma|}$-cover, the diameter of $B^\ast_\sigma$ is less than $2^{-|\sigma|}$.
Clearly, every infinite path $\beta\in[\mathbf{H}]$ gives a unique point $\rho_\yy(\beta):=\bigcap_{n\in\omega}B^\ast_{\beta\rs n}\in\yy$.
Consequently, the map $\rho_\yy:[\mathbf{H}]\to\yy$ is a desired representation of Hilbert cube.
\end{example}

If $\ca{X}$ and $\ca{Y}$ are recursively presented metric spaces, equipped with the Nbase-representation, this notion of computability can be characterized in a convenient way using the enumeration reducibility.
Recall that a function $\Psi:\mathcal{P}(\om)\to\mathcal{P}(\om)$ is called an {\em enumeration operator} (see \cite[Section 11]{CooperBook}) if there is a $\Sigma^0_1$ set $P \subseteq \omega \times \omseq$ such that for every $A\in 2^\omega$,
\[
k\in\Psi(A) \iff (\exists u = (u_0,\dots,u_{n-1}) \in \omseq)[P(k, u) \ \& \ (\forall i < n)[u(i) \in A]].
\]

It is not hard to see that a partial function $f:\xx\partialf\yy$ is recursive if and only if there is an enumeration operator $\Psi$ such that for all $x\in\Dom(f)$, we have $\Psi(\nbase{x})=\nbase{y}$.
Indeed, if underlying spaces are metrizable, we can always assume that the length of the witness $u$ is $1$, that is, a partial function $\Phi:\ca{X}\rightharpoonup \ca{Y}$ is recursive if and only if there is a  $\Sigma^0_1$ set $\tilde{\Phi}\subseteq\omega^4$ such that for all $x \in \Dom(\Phi)$ it holds:
\[
\Phi(x) \in B^\ca{Y}_{m,s} \iff (\exists n,r)[x \in B^\ca{X}_{n,r} \ \& \ \tilde{\Phi}(n,r,m,s)].
\]

We use the notation $\Phi_e^{\xx,\yy}$ to denote the $e^{\rm th}$ partial recursive function from $\xx$ into $\yy$ in the obvious sense, that is, it is induced by the $e^{\rm th}$ partial recursive function $F:\baire\rightharpoonup\baire$ via underlying representations, or equivalently, by the $e^{\rm th}$ enumeration operator $\Psi:\mathcal{P}(\om)\to\mathcal{P}(\om)$ under the Nbase-representation.
More explicitly, $\Phi_e^{\xx,\yy}$ is  \emph{largest} function $\Phi:\xx\partialf\yy$ induced by the $e$-th $\Sigma^0_1$ set $\tilde{\Phi}$ as above, \ie $\Phi_e^{\ca{X},\ca{Y}}(x)$ is defined exactly when there is a unique $y \in \ca{Y}$ such that $y \in B^\ca{Y}_{m,s} \iff (\exists n,r)[x \in B^\ca{X}_{n,r} \ \& \ \effuniv^{\om^4}_{1}(e,\emptyset,n,r,m,s)]$; in the latter case we let $\Phi_e^{\ca{X},\ca{Y}}(x)$ be the unique such $y$.
Hereafter we use $\tilde{\Phi}_e$ to denote the $e^{\rm th}$ $\Sigma^0_1$ subset of $\om^4$ (that is, $\tilde{\Phi}_e=\effuniv^{\om^4}_{1,e,\emptyset}$), which codes the $e$-th partial computable function as described above. It is not hard to see that the domain of $\Phi_e^{\ca{X},\ca{Y}}$ is a $\Pi^0_2$ subset of $\xx$. We often abbreviate $\Phi_e^{\xx,\yy}$ to $\Phi_e$ if the underlying spaces are clear from the context.

\begin{definition}[Miller {\cite[Definition 3.1]{miller2}}, see also Kihara-Pauly \cite{KihPau} and Moschovakis {\cite[Section 3D]{yiannis_dst}}]\normalfont Suppose that $\xx$ and $\yy$ are recursively represented metric spaces.
We say that {\em $y\in \yy$ is representation reducible to $x\in \xx$} (in the sense of Miller) if there is some $e \in \om$ such that $\Phi^{\ca{X},\ca{Y}}_e(x)=y$. In this case we write $y\leq_Mx$.\footnote{The relation $\leq_M$ can be characterized in terms of the \emph{Medvedev reducibility} between sets of names of $x$ and $y$, see \cite{KihPau}. This is the justification for our notation. This relation is also denoted by $\leq_r$ in Miller {\cite[Definition 3.1]{miller2}}, and (using an equivalent definition) by $\tleq$ in Moschovakis \cite[3D.16]{yiannis_dst}.}
The \emph{continuous degree} of $x$ is defined as usual to be its equivalence class under $\equiv_M : = \ \leq_M \cap \geq_M$. We say that a point $x\in\mathcal{X}$ is {\em total} if there is $z\in 2^\omega$ such that $x\equiv_Mz$.
\end{definition}
Given $A, B \subseteq \om$ recall that $A$ is \emph{enumeration reducible} to $B$ if there is an enumeration operator $\Psi$ such that $\Psi(B)=A$. In this case we write $A \leq_e B$.

The following are some useful facts which will be used throughout this section. For convenience we summarize it here.
\begin{fact}\label{usefulfactsaboutreducibility}
Let $\xx$ and $\yy$ be recursively represented metric spaces.
\begin{enumerate}[(i)]
\item Let $x\in \xx$ and $y\in \yy$. Then $y\leq_M x$ if and only if $\nbase{y}\leq_e \nbase{x}$.
\item If $z\in \baire$ then $\nbase{z}$ can be identified with $\left\{\sigma\in\om^{<\om}:\sigma\prec z\right\}$.
\item The restriction of $\leq_M$ on $\baire \times \baire$ is $\tleq$. That is, if $x,y\in\baire$ then $x\leq_M y$ if and only if $x\leq_T y$.
\item If $z\in\baire$ and $x\in\xx$ then $z\geq_M x$ iff there is a fixed name $p$ for $x$ such that $z\geq_T p$.
\item Let $x\in\xx$. Then $x$ is total if and only if $x$ has a canonical name, i.e. a name $p_x$ such that $p_x\equiv_M x$.
\end{enumerate}
\end{fact}

The following is a key property of the enumeration reducibility.

\begin{lemma}[Selman, Rozinas (see {\cite[Corollary 4.3]{miller2}})]\label{lem:unif-nonunif}
Let $A$ and $B$ be subsets of $\omega$.
Then, $A\leq_eB$ if and only if for every $C\subseteq\omega$, $B\leq_eC\oplus\overline{C}$ implies $A\leq_eC\oplus\overline{C}$.\smallskip

In particular, if $\mathcal{X}$ and $\mathcal{Y}$ are recursively presented metric spaces, then for any $x\in\mathcal{X}$ and $y\in\mathcal{Y}$,
\[x\leq_My\;\iff\;(\forall z\in 2^\omega)\;[y\leq_Mz\;\rightarrow\;x\leq_Mz].\qed\]
\end{lemma}

For any $\alpha,\beta\in\baire$ recall that $\alpha\oplus\beta$ is defined as the join of $\alpha$ and $\beta$.
Given representations $\rho_\xx$ and $\rho_\yy$ of spaces $\xx$ and $\yy$ automatically induces a representation $\rho_{\xx\times\yy}$ of the product space $\xx\times\yy$ defined by $\rho_{\xx\times\yy}(\alpha\oplus\beta)=(\rho_\xx(\alpha),\rho_\yy(\beta))$.
For any points $x\in\xx$ and $y\in\yy$ we often write $x\oplus y$ for the point $(x,y)$ in the product space $\xx\times\yy$.
Therefore, for any points $x,y,z$ in represented spaces, a formula such as $x\oplus y\leq_Mz$ makes sense.
By using this terminology, we state a result obtained from Miller's argument \cite[Proposition 5.3]{miller2} saying that every continuous degree is \emph{almost} total.

\begin{lemma}[Miller \cite{miller2}; see also {\cite[Lemma 9]{KihPau}}]\label{lem:nonsplitting}
Let $\mathcal{X}$ be a recursively presented metric space.
Then, for every $x\in\mathcal{X}$ and $z\in 2^\omega$, either $x>_Mz$ holds or $x\oplus z$ is total.\qed
\end{lemma}

Roughly speaking, the above Lemmata \ref{lem:unif-nonunif} and \ref{lem:nonsplitting} imply that certain arguments involving $\leq_M$ can be reduced to arguing about total degrees in $2^\omega$.
For instance, in the proof of Theorem \ref{theorem main}, we will use the above Lemmata \ref{lem:unif-nonunif} and \ref{lem:nonsplitting} to enable us to use the Friedberg Jump Inversion Theorem, which only holds for total degrees:

\begin{lemma}[The Friedberg Jump Inversion Theorem; see {\cite[Theorem 10.6.9]{CooperBook}}]\label{lem:jump-inversion}
Let $\xi$ be an ordinal less than the first non-computable ordinal $\omega_1^z$ relative to an oracle $z\in\baire$. For every $x\in\baire$, there exists $y\in\baire$ such that $y\tgeq z$ and $y\tpower{\xi}\teq x\oplus z\tpower{\xi}$.\qed
\end{lemma}

We will define the notion of the jump operation in the context of the generalized Turing degrees in Section \ref{section:turing-jump}.
However, we will see that the jump of a point must be a point in $\baire$, and therefore, there is no counterpart of the Friedberg Jump Inversion Theorem in the generalized Turing degrees.




\subsection{The Turing Jump Operations}\label{section:turing-jump}

The notion of  a ``jump" of points can be extended to all Polish spaces in a natural way.
Assume that \ca{X} is a recursively presented metric space and $n \geq 1$.
Recall that $\effuniv^\ca{X}_{n,\alpha,e}$ is the $e^{\text{th}}$ $\Sigma^0_n(\alpha)$ subset of $\ca{X}$.
We define the {\em $\Sigma^0_n(\alpha)$-jump of $x \in \xx$} as the set
\[
J^{(n),\alpha}_\xx(x)=\{e\in\omega:x\in \effuniv^\xx_{n,\alpha,e}\}.
\]
The function $J^{(n),\alpha}_\xx:\xx\to 2^\omega$ is called the {\em $\Sigma^0_n(\alpha)$-jump operator on $\xx$}. Similarly we define the $\Sigma^0_{n}(\alpha)$-jump in $\ep$-recursively presented metric spaces, where $\alpha \tgeq \ep$.
We use the abbreviations $J^{(n)}_\xx$ and $J_\xx$ to denote $J^{(n),\emptyset}_\xx$ and $J^{(1),\emptyset}_\xx$, respectively, where as usual the empty set is identified with the infinite sequence $(0,0,\dots)$. It is also standard to denote $W_e$ as the $e^{\text{th}}$ $\Sigma^0_1$ subset of $\omega$. For example, the unrelativized jump operator $J^{(1),\emptyset}_\xx$ is simply the function that takes $x\in\xx$ to the set
\[\left\{e\in\omega: x\in \bigcup_{(m,r)\in W_e} B^\xx_{m,r}\right\}.\]
Initially to distinguish between the different notions of a jump, if $p\in\baire$ then we write $TJ(p)$ for the usual Turing jump of $p$ (defined by the relativized halting problem).
Note that the Turing jump $TJ$ is equivalent to our $\Sigma^0_1$-jump $J_\baire$ on Baire space, i.e. if $\alpha\in\baire$ then $TJ(\alpha)\equiv_1 J_\baire(\alpha)$.

In order to check that our jump operator shares the same basic properties with the usual Turing jump we will use the following property of our canonical universal system:



\begin{itemize}
\item Our coding system $(\univ_n^\ca{Z})_\ca{Z}$ is \emph{good}, roughly speaking, in the sense that the map transforming $(x,P)\in \mathcal{X}\times\bolds^0_n$ into the $x$-section $P_x\in\bolds^0_n$ is continuous w.r.t.~codes in $(G_n^\mathcal{Z})_{\ca{Z}}$, where $\ca{X}$ is of the form $\om^k \times \baire^t$ for some $t,k \geq 0$.
\end{itemize}

More precisely, a system $(\univ^\yy)_\yy$ is \emph{good} if for every $\ca{X}$ of the form $\om^k \times \baire^t$ and every Polish space \ca{Y} there exists a continuous function $S^{\ca{X},\ca{Y}} \equiv S: \baire \times \ca{X} \to \baire$ such that
\begin{align*}\label{equation good transition}
\univ^{\ca{X} \times \ca{Y}}(\ep,x,y) \iff \univ^{\ca{Y}}(S(\ep,x),y).
\end{align*}

Our universal lightface system $(\effuniv_n^\yy)_\yy$ is moreover \emph{effectively good}, \ie it satisfies the similar condition
\[
\effuniv_n^{ \om^k \times \ca{Y}}(e,\alpha,\vec{e},y) \iff \effuniv^\ca{Y}_n(S(e,\vec{e}),\alpha,y)
\]
for some \emph{recursive} function $S: \om \times \om^k \to \om$. From this it is not hard to prove that the corresponding system $(\univ_n^\yy)_\yy$ is good.\smallskip

Recall that for $A,B\subseteq\mathbb{N}$, the $1$-reducibility $A\leq_1B$ is defined by the existence of a recursive injection $f:\omega\to\omega$ such that $n\in A$ if and only if $f(n)\in B$.

\begin{lemma}\label{prop:jump-equivalence}
Let $\mathcal{X}$ be a recursively presented metric space and $n, m \geq1$. For any $x\in\mathcal{X}$ and $\alpha\in\baire$, we have $J^{(n),\alpha}_\xx(x)\equiv_1J^{(n)}_{\baire \times \xx}(\alpha,x)$.
Moreover, we also have $J_\xx^{(n+m)}(x)\equiv_1TJ^{(m)}\circ J_\xx^{(n)}(x)$.
\end{lemma}

\begin{proof}
Both assertions are proved easily using the fact that our system $(\effuniv_{n}^\yy)$ is effectively good. To verify the second assertion for example, it is enough to show that $J_\xx^{(k+1)}(x)\equiv_1TJ\circ J_\xx^{(k)}(x)$ for any $k\geq 1$. We have that
\begin{align*}
e \in  J_\xx^{(k+1)}(x)
\iff& \ \effuniv^\ca{X}_{k+1}(e,\emptyset,x)
\iff \ \exists i \neg \effuniv^{\om \times \ca{X}}_{k}(e,\emptyset,i,x)\\
\iff& \ \exists i \neg \effuniv^\ca{X}_{k}(S(e,i),\emptyset,x)
\iff \ \exists i S(e,i) \in  J_\xx^{(k)}(x)
\end{align*}
for a suitably chosen recursive function $S$. The latter relation defines a $\Sigma^0_1(J_\xx^{(k)}(x))$ subset of $\om$, and hence it is $1$-reducible to $TJ (J_\xx^{(k)}(x))$. The other side of the inequality is proved similarly by remarking that the condition $e \in TJ (J_\xx^{(k)}(x))$ defines a $\Sigma^0_{k+1}$ relation on $(e,x)$.
\end{proof}

In the light of Lemma  \ref{prop:jump-equivalence} and the fact that the $\Sigma^0_1$ jump extends the usual Turing jump, we will henceforth use $x'$ to denote the first unrelativized jump $J^{(1),\emptyset}_{\ca{X}}(x)$, and $x^{(n)}$ to denote the $n^{\text{th}}$ unrelativized jump. Generally to relativize the jump operator, given a recursively presented metric space \ca{X}, a natural $n \geq 1$ and $\alpha \in \baire$ we define
\[
(x \oplus \alpha)\tpower{n} : = (\alpha \oplus x)\tpower{n} : = J^{(n)}_{\baire \times \ca{X}}(\alpha,x).
\]
Using Lemma \ref{prop:jump-equivalence} we can see that this notion of a jump is meaningful: the $n^{\text{th}}$ jump of $(x,\alpha)$ is up to $\teq$ the same as the $\Sigma^0_n(\alpha)$-jump of $x$. We will customarily write $x \oplus \alpha$ or $\alpha \oplus x$ to denote the pair $(x,\alpha)$, \eg by $x \oplus \alpha \leq_M y$ we mean $(x,\alpha) \leq_M y$.

The next step is to make sure that our jump operation is well-defined on the $M$-degree structure.

\begin{lemma}\label{fact:jumpproperties}
Suppose $x\in\xx$ and $y\in\yy$.
Then, $x\leq_M y$ if and only if $x'\leq_1 y'$.
\end{lemma}

\begin{proof}
Suppose $x\leq_My$.
Then, there is a recursive partial function $f:\yy\rightharpoonup\xx$ such that $f(y)=x$. Then $f^{-1}\bolds^0_1\subseteq\bolds^0_1$ holds recursive-uniformly.
This shows that there is a recursive function $h$ such that $e\in x'$ if and only if $h(e)\in y'$.
Conversely, suppose $x'\leq_1 y'$.
Then, there is a recursive function $g$ such that $e\in \nbase{x}$ if and only if $g(e)\in y'$.
The last condition is equivalent to $W_{g(e)}\cap\nbase{y}\not=\emptyset$ (recall that $W_k$ refers to the $k^{\text{th}}$ $\Sigma^0_1$ subset of $\omega$).
In particular, $\nbase{x}\leq_e\nbase{y}$.
Thus, we have $x\leq_My$.
\end{proof}

By Lemma \ref{fact:jumpproperties}, we have that $x'\leq_Tp'$ for every name $p$ of $x$ since $x\leq_Mp$.
The following says that the jump of a point can be simulated by the usual Turing jump of its generic name. Note that by Fact \ref{usefulfactsaboutreducibility}, $x\equiv_M p$ holds only if $x$ is total.

\begin{lemma}\label{prop:nameandjump}
Given any $x\in\xx$ there is a name $p\in\omega^\omega$ of $x$ such that $x'\equiv_T p'$.
\end{lemma}

\begin{proof}
Follow the usual jump inversion argument.
We construct $p=\lim_sp_s\in\baire$ such that ${\rm Rng}(p)=\nbase{x}$ by a $x'$-computable way.
Let $p_0$ be the empty string.
At stage $s$, suppose that $p_s\in\omega^{<\omega}$ is given.
For $\tau\in\omega^{<\omega}$, let $B_\tau^\xx$ denote $\bigcap_{i\in{\rm Rng}(\tau)}B_i^\xx$.
Then, check if $x$ is contained in the following $\Sigma^0_1$ set by using $x'$:
\[\bigcup\{B_{\tau}^\xx:\tau\succ p_s\mbox{ and }\Phi_s^\tau(s)\downarrow\}.\]

If yes with a witness $\tau$ such that $x\in B_{\tau}^\xx$, we define $p_s^*=\tau$, and otherwise $p_s^*=p_s$.
If $x\in B^\xx_s$, then put $p_{s+1}=p_s^*\fr s$, and otherwise $p_{s+1}=p_s^*$.
The construction is clearly $x'$-computable.
To see that ${\rm Rng}(p)=\nbase{x}$, we note that ${\rm Rng}(p_s)\subseteq\nbase{x}$ implies ${\rm Rng}(p_s^*)\subseteq\nbase{x}$ since the condition $x\in B_{\tau}^\xx$ is equivalent to that ${\rm Rng}(\tau)\subseteq\nbase{x}$.
Then, obviously, we have ${\rm Rng}(p_{s+1})\subseteq\nbase{x}$.
Conversely, if $s\in\nbase{x}$ then we have $s\in{\rm Rng}(p_{s+1})$ by our construction.
Finally, note that $p'(s)=1$ if and only if $\Phi_s^{p_{s+1}}(s)\downarrow$. This is because if $\Phi_s^{p_{s+1}}(s)\uparrow$ (meaning no witness $\tau$ was found at stage $s$), then as ${\rm Rng}(p)=\nbase{x}$ we must have $\Phi^{p}(s)\uparrow$.
Therefore, $p'\leq_Tx'$.
\end{proof}

Lemma \ref{prop:nameandjump} is also a key tool in our proof of Theorem \ref{theorem main}.
Indeed, this lemma has a similar role as Lemmata \ref{lem:unif-nonunif} and \ref{lem:nonsplitting}.

\subsection{The Cone Avoidance Lemma}

To prove the Shore-Slaman Join Theorem in an arbitrary Polish space, we need to show the cone avoidance lemma. To do this we utilize the idea of effective compactness.
For notational simplicity, if the underlying space $\xx$ is clear from the context, hereafter we write $\univ^\ep_e$ for the $e^{\text{th}}$ $\Sigma^0_1(\ep)$ set $\univ^\ca{X}_{1,e\fr\ep}$ so that $F^\ep_d = \ca{X} \setminus \univ^\ep_d$.
We say that a recursively presented metric space $\ca{X}$ is \emph{$\ep$-effectively compact} if it is compact and if there is an $\ep$-computable procedure deciding whether the union of a given finite set of basic open balls covers $\ca{X}$ or not.
If $\ep=\emptyset$ we simply say that $\ca{X}$ is effectively compact.
Clearly, every effectively compact space is also $\ep$-effectively compact for all oracles $\ep$.
It is not hard to see that $\ca{X}$ is $\ep$-effectively compact if and only if $\ca{X}$ is compact and the set $\{\langle d,e\rangle:F^\ep_d\subseteq \univ^\ep_e\}$ is a $\Sigma^0_1(\ep)$ subset of \om.
It is evident that every compact Polish space is $\ep$-effectively compact for some oracle $\ep$. It is also not difficult to verify that the \emph{Hilbert cube} $[0,1]^\om$ is effectively compact.

Going back to the functions $\Phi_e^{\ca{X},\ca{Y}} \equiv \Phi_e$ we recall that the condition $\tilde{\Phi}_e(n,r,m,s)$ implies that $\Phi_e[B^\ca{X}_{n,r}] \subseteq B^\ca{Y}_{m,s}$, where by $\Phi_e[A]$ we mean $\Phi_e[A \cap \Dom(\Phi_e)]$.
Since the domain of $\Phi_e$ is a $\Pi^0_2$ subset of $\xx$ (uniformly in $e$) we may write $\Dom(\Phi_e) = \bigcap_tD^e_t$, where the set $\set{(e,t,x)}{x \in D^e_t}$ is $\Sigma^0_1$.

\begin{lemma}[Cone avoidance]\label{thmconeavoidance}
Suppose that $\ep\in\baire$ is an oracle, and that $y$ is a point in a recursively presented metric space $\yy$ such that  $y\not\leq_M \ep$.
If $P$ is a nonempty $\Pi^0_1(\ep)$ subset of an effectively compact space $\ca{H}$, then there is some $z\in P$ such that $y\not\leq_M z$ and $z\leq_M y\oplus \ep'$.
\end{lemma}

\begin{proof}
We first prove the following variant of the Low Basis Theorem in continuous degrees. As usual a point $z$ is \emph{$\ep$-low} if $(z\oplus\ep)'\leq_T\ep'$.

\begin{claim}[Low Basis Theorem]
Assume $\ca{H}$ is effectively compact and $\varepsilon\in\baire$.
If $P$ is a nonempty $\Pi^0_1(\ep)$ subset of $\ca{H}$ then there is an $\ep$-low point $z\in\ca{H}$ such that $z\in P$.
\end{claim}

\noindent
{\it Proof.}
As in the usual proof, we construct an $\ep'$-computable decreasing sequence $(Q_e)_{e\in\omega}$ of $\Pi^0_1(\ep)$ sets in $\ca{H}$. More specifically, this means that the sequence of indices of $(Q_e)_{e\in\omega}$ is computable from $\varepsilon'$.
Define $Q_0=P$.
By ($\ep$-)effective compactness, we can decide $Q_e\subseteq \univ_e^\ep$ by using $\ep'$ uniformly in $e$.
Put $Q_{e+1}=Q_e$ if $Q_e\subseteq \univ_e^\ep$, and otherwise, put $Q_{e+1}=Q_e\setminus \univ_e^\ep$.
For any $z\in Q:=\bigcap_{e\in\omega}Q_e$, clearly, $J^\ep_{\ca{H}}(z)(e)=1$ if and only if $Q_{e}\subseteq \univ_{e}^\ep$.
In other words, $J^\ep_{\ca{H}}(z)\leq_T\ep'$.\qed(claim)

\medskip


Going back to the main proof, we have from Lemma \ref{lem:nonsplitting} that either $\ep'\leq_My$ or that $y\oplus\ep'$ is total. In the former case we obtain from the preceding claim an $\ep$-low point $z \in P$, which easily satisfies the assertion. In the remaining of this proof we assume that $y\oplus\ep'$ is total. By Fact \ref{usefulfactsaboutreducibility} we can fix the canonical name $p_y$ of $y$ and show that the construction is computable in $p_y\oplus \varepsilon'$.

Begin with $P_0=P$. Suppose that a nonempty $\Pi^0_1$ set $P_e\subseteq P$ has been already constructed. At substage $t$ of stage $e$, consider the $\Sigma^0_1$ set $D^e_t\subseteq\ca{H}$ introduced above.
Since $\ca{H}$ is $\ep$-effectively compact, by using $\ep'$, one can decide whether $P_e\subseteq D^e_t$ or not.
If $P_e\not\subseteq D^e_t$, define $Q_{e+1}=P_e\setminus D^e_t$.
This ensures that $Q_{e+1}\cap{\rm dom}(\Phi_e)=\emptyset$. We then go to the next stage $e+1$.

Otherwise, check whether $\tilde{\Phi}_{e,t}$ (note that this is a finite set) contains a computation $(n,r,m,s)$ such that
\begin{align}\label{formula1}
\hat{B}^{\ca{H}}_{n,r(t)}\cap P_e\not=\emptyset\mbox{ and }B^\yy_{i_0,i_1}\cap B^\yy_{m,s}=\emptyset,
\end{align}
where $\tilde{\Phi}_{e,t}$ is the stage $t$ approximation of $\tilde{\Phi}_e$, $\hat{B}$ is the closed ball which has the same center and radius with $B$, and $r(t)$ is such that $q_{r(t)} = q_r-2^{-t}$. 
Also $B^\yy_{i_0,i_1}$ is the $t^{\text{th}}$ ball enumerated by the name $p_y$ of $y$.

By $\ep$-effective compactness of $P_e$, $\ep'$ can decide whether $\hat{B}\cap P_e\not=\emptyset$ for a given basic closed ball $\hat{B}$.
Therefore, we can use $p_y\oplus\ep'$ to decide whether (\ref{formula1}) holds or not.
If yes, put $Q_{e+1}=P_e\cap\hat{B}^\ca{H}_{n,r(t)}$ and go to stage $e+1$.
This construction ensures that $\Phi_e[Q_{e+1}]\subseteq B^{\ca{H}}_{m,s}$ and hence, $y\not\in \Phi_e[Q_{e+1}]$.

If no, go to substage $t+1$, and repeat the above. At the end, if the substage $t$ goes to infinity, we have $P_e\subseteq \bigcap_tD^e_t$, that is, $\Phi_e(z)\in\yy$ for any $z\in P_e$.
We claim that $\Phi_e(z)=y$ for every $z\in P_e$.
If the claim is false, then we have $\Phi_e(z)\not=y$ for some $z\in P_e$, then $\tilde{\Phi}_e$ enumerates a computation $\Phi_e[B^{\ca{H}}_{n,r}]\subseteq B^\yy_{m,s}$ such that $z\in B_{n,r}^{\ca{H}}$ and $y\not\in \overline{B^\yy_{m,s}}$.
Then, for any sufficiently large $t$, we have $z\in \hat{B}^\ca{H}_{n,r(t)}\cap P_e$ and $B^\yy_{i_0,i_1}\cap B^\yy_{m,s}=\emptyset$ at substage $t$, and hence, the construction moves to stage $e+1$ after substage $t$ of stage $e$, which contradicts our assumption that $t$ goes to infinity.
Hence, $\Phi_e[P_e]=\{y\}$.

Now, by $\ep$-effective compactness of $P_e$, for any $i\in\omega$, one can $\ep$-computably find a finite set $V\subseteq \tilde{\Phi}_e\cap\{(n,r,m,s):q_s<2^{-i}\}$ such that
\[P_e\subseteq\bigcup_{(n,r,m,s)\in V}B^\xx_{n,r}\mbox{, and }\bigcap_{(n,r,m,s)\in V}B_{m,s}^\yy\not=\emptyset.\]

Note that the first clause is $\Sigma^0_1(\ep)$ and the second clause is $\Sigma^0_1$. Then, the diameter $E_i=\bigcap_{(n,r,m,s)\in V}B_{m,s}^\yy$ is smaller than $2^{-i}$, and it gives an $\ep$-computable decreasing sequence of open sets converging to $y$.
This contradicts our assumption that $y\not\leq_M\ep$, and thus eventually, the strategy must go to the next stage $e+1$.

Finally before starting stage $e+1$ we find the first closed ball $\hat{B}$ of radius $2^{-e-2}$ such that $\hat{B}\cap Q_{e+1}\neq \emptyset$ (we can do this by $\ep$-effective compactness). We set $P_{e+1}=\hat{B}\cap Q_{e+1}$. This ends the description of the construction.

Note that the construction is computable in $p_y\oplus\ep'\leq_M y\oplus \ep'$. This means that the oracle is not only able to decide a sequence of indices for $\{P_e\}$, but also a sequence $\{B_j\}$ of open balls such that the radius of $B_j$ is $2^{-j}$ and $P_j\subset B_j$, which yields a name for $z$. By compactness, as each $P_e\neq\emptyset$, we have a unique point $z\in\bigcap_e P_e$. Hence $z\leq_M p_y\oplus\ep'\leq_M y\oplus \ep'$. The above argument shows for every $e\in\omega$ then either $\Phi_e(z)$ is undefined or $\Phi_e(z)\not=y$.
Consequently, $y\not\leq_Mz\leq_My\oplus\ep'$ as desired.

We make a remark on the proof. The proof of Lemma \ref{thmconeavoidance} is non-uniform, and in the case where $y\oplus\ep'$ is total we can carry out an analogous version of the usual proof in the Turing degrees, by carrying out the construction relative to the fixed name $p_y$ of $y$ (together with $\ep'$). On the other hand if $y\oplus\ep'$ is not total then this approach will typically not work, because carrying out the usual proof relative to \emph{different} names for $y$ will produce \emph{different} objects at the end. Thus a different technique is usually needed in the case where the oracle is not total. This is a key obstacle in the proof of the Shore-Slaman Join Theorem (Theorem \ref{thm:shoreslamanpolish}).
\end{proof}

\subsection{The Shore-Slaman Join Theorem}

As in the proof of Kihara \cite{kihara_decomposing_Borel_functions_using_the_Shore_Slaman_join_theorem}, one of the key technical tools from recursion theory to show Theorem \ref{theorem main} is the so-called Shore-Slaman Join Theorem \cite{shore_slaman_defining_the_Turing_Jump}, but for the generalized Turing degrees.
The original Shore-Slaman Join Theorem asserts that the Turing degree structure $\mathcal{D}_T$ satisfies the following generalization of the Posner-Robinson Join Theorem:
\begin{align}\label{degree-formula2}
\mathcal{D}_T\models(\forall\mathbf{x})(\forall\mathbf{y}\not\leq\mathbf{x}^{(n)})(\exists\mathbf{g}\geq\mathbf{x})\;\mathbf{g}^{(n+1)}=\mathbf{y}\oplus\mathbf{g}=\mathbf{y}\oplus\mathbf{x}^{(n+1)}.
\end{align}

Of course, it is impossible that the generalized Turing degree structure satisfies the sentence (\ref{degree-formula2}) since $\mathbf{g}^{(n+1)}$ must be total while $\mathbf{y}\oplus \mathbf{x}^{(n+1)}$ can be non-total (for appropriate choices of $\mathbf{x}$ and $\mathbf{y}$).
Moreover, it seems that stronger separation axioms such as metrizability has an essential role to play here even in the equivalence $\mathbf{g}^{(n+1)}=\mathbf{y}\oplus\mathbf{g}$. Kalimullin \cite{Kali} has pointed out that the sentence (\ref{degree-formula2}) fails for the enumeration degrees (even for $n=1$):
\[\mathcal{D}_e\not\models(\forall\mathbf{y}\not\leq\mathbf{0}^\prime)(\exists\mathbf{g})\;\mathbf{g}^{\prime\prime}\leq\mathbf{y}\oplus\mathbf{g}.\]

Indeed, the witness $\mathbf{y}$ of the failure in the above formula can be chosen as a semi-recursive enumeration degree.
In particular, the Shore-Slaman join theorem fails for any space in which the reals $\mathbb{R}_<$ endowed with the lower topology (see \cite{KihPau}) is computably embedded.
The main result in this section is that the strongest possible form of the Shore-Slaman Join Theorem (which implies a generalization of the well-known Posner-Robinson Join Theorem) holds:
\[\mathcal{D}_{[0,1]^{\omega}}\models(\forall\mathbf{x})(\forall\mathbf{y}\not\leq\mathbf{x}^{(n)})(\exists\mathbf{g}\geq\mathbf{x})\;\mathbf{g}^{(n+1)}=\mathbf{y}\oplus\mathbf{g},\]
where $\mathcal{D}_{[0,1]^{\omega}}$ denotes the structure of the continuous degrees.
In our proof of the Shore-Slaman Join Theorem for the continuous degrees, we will see that the (compact) metrizability of the underlying space plays a key role.

\begin{theorem}[The Shore-Slaman Join Theorem for Polish spaces]\label{thm:shoreslamanpolish}
Let $x$ and $y$ be points in recursively presented metric spaces $\xx$ and $\yy$, respectively.
If $y\not\leq_M x^{(n)}$, then there is $G\in 2^\omega$ such that $G\geq_Mx$ and $G^{(n+1)}\equiv_MG\oplus y$.

Furthermore, if $y\oplus x^{(n+1)}$ is total then $G^{(n+1)}\equiv_MG\oplus y\equiv_M y\oplus x^{(n+1)}$, and if $y\oplus x^{(n+1)}$ is not total then $G^{(n+1)}\equiv_MG\oplus y\equiv_M G\oplus x^{(n+1)}$.

Indeed, for an ordinal $\xi<\ck$, if $y\not\leq_Mx^{(\zeta)}$ for all $\zeta<\xi$, there exists $G\in 2^\omega$ such that $G\geq_Mx$ and $G^{(\xi)}\equiv_MG\oplus y$.

\end{theorem}

\begin{proof} For this proof we will assume that the reader is reasonably familiar with the proof of the Shore-Slaman join theorem, for instance, found in  \cite{shore_slaman_defining_the_Turing_Jump}.

If $n=0$, we have $y\not\leq_Mx$.
Then, by Lemma \ref{lem:unif-nonunif}, there is $\hat{x}\in 2^\omega$ such that $x\leq_M\hat{x}$ and $y\not\leq_M\hat{x}$.
If $n>0$, choose a generic name $\hat{x}$ of $x\in\xx$ in Lemma \ref{prop:nameandjump}.
Therefore in any case, without loss of generality, we may assume that $x\in 2^\omega$.
We can also assume that $\yy=[0,1]^\omega$, since every recursively presented metric space $\yy$ is recursively embedded into Hilbert cube $[0,1]^\omega$, and any recursive embedding preserves the $M$-degrees.

We use the recursively bounded recursive tree $\mathbf{H}\subseteq\om^{<\om}$ of names (together with a representation $\rho_\yy:[\mathbf{H}]\to\yy=[0,1]^\om$) introduced in Example \ref{exa:rep-Hilbert-cube}. 
Given a set $\Phi\subseteq\omega\times 2\times\mathbf{H}$, we write $\Phi(\sigma)(n)=k$ if there is $\tau\preceq\sigma$ such that $(n,k,\tau)\in\Phi$.
Recall from Shore-Slaman \cite[Definition 2.4]{shore_slaman_defining_the_Turing_Jump} that a {\em use-monotone Turing functional (on $\mathbf{H}$)} is a set $\Phi$ of triples $(n,m,\sigma)\in\omega\times 2\times\mathbf{H}$ which defines a partial monotone function from $\mathbf{H}$ into $2^{<\omega}$ in the sense that $\Phi$ is single-valued, that is, $\Phi(\sigma)(n)=k$ and $\Phi(\sigma)(n)=l$ implies $k=l$, and that if $\Phi(\sigma)(n)$ is defined then $\Phi(\sigma)(m)$ is also defined for all $m<n$.
We write $\Dom(\Phi)$ for the set of all $\sigma\in\mathbf{H}$ such that $(n,k,\sigma)\in\Phi$ for some $n$ and $k$.
A Turing functional $\Phi$ can also be viewed as a partial continuous function from $[\mathbf{H}]$ into $2^\omega$.
However, we do not require that a use-monotone Turing functional $\Phi$ is consistent along all points.
Here, we say that $\Phi$ is {\em consistent along a point $z\in\yy$} provided for any $\sigma,\tau\in\mathbf{H}$ with $z\in B^\ast_\sigma\cap B^\ast_\tau$ and $n\in\omega$, if $\Phi(\sigma)(n)$ and $\Phi(\tau)(n)$ are defined, then $\Phi(\sigma)(n)=\Phi(\tau)(n)$ (where recall from Example \ref{exa:rep-Hilbert-cube} that $B^\ast_\sigma$ is the basic open ball coded by $\sigma\in\mathbf{H}$).
Instead, we will construct a use-monotone Turing functional $\Phi$ on names of a specific point $z$, and use the following property. Use-monotone Turing functionals need not be $\Sigma^0_1$, and in fact we will sometimes view use-monotone Turing functionals $\Phi$ as members of $2^\omega$. For instance:

\begin{claim}
Let $\Phi$ be a use-monotone Turing functional on $\mathbf{H}$.
Suppose that $\Phi$ is consistent along a point $z$,
and moreover, there is a name $\alpha_z$ of $z$ such that $\Phi(\alpha_z)\in 2^\omega$ is defined.
Then, we have $\Phi(\alpha_z)\leq_M\Phi\oplus z$.
\end{claim}

\begin{proof}
Given $n\in\omega$, search $\sigma\in\mathbf{H}$ and $m\in\{0,1\}$ such that $(n,m,\sigma)\in\Phi$ and $z\in B^\ast_\sigma$.
Since $\Phi(\alpha_z)$ is defined for a name $\alpha_z$ of $z$, there is at least one such $(\sigma,m)$, and a brute-force search (relative to $\Phi$ and any name for $z$) eventually finds such a pair.
Using different names for $z$ may yield different search results, but by consistency, we must have $m=\Phi(\alpha_z)(n)$.
\newcommand{\qedclaim}{\qedsymbol(claim)}
\renewcommand{\qed}{\hfill\qedclaim}
\end{proof}

Now we describe the idea behind our proof of the Shore-Slaman Join Theorem in the continuous degrees. The proof is non-uniform, and as in Lemma \ref{thmconeavoidance}, the proof splits into two cases depending on whether or not the oracle $y\oplus x^{(n+1)}$ is total.

If $y\oplus x^{(n+1)}$ is total, then we will adopt a continuous analogue of the classical construction of Shore and Slaman \cite{shore_slaman_defining_the_Turing_Jump}. This construction is computable in $y\oplus x^{(n+1)}$ via the canonical name for $y$, and we get $G^{(n+1)}\equiv_MG\oplus y\equiv_M y\oplus x^{(n+1)}$ as usual.

A key difficulty is that if $y\oplus x^{(n+1)}$ is not total, then such a $\left(y\oplus x^{(n+1)}\right)$-computable construction cannot work because different names for $y$ will produce different sets $G$ at the end, and we are forced to innovate in this case. We fix a name $\alpha_y$ of $y$ for the construction; any arbitrary choice of $\alpha_y$ will do. The first observation is that we do not (and in fact, cannot) require the construction to be recoverable from the oracle $y\oplus x^{(n+1)}$. Instead we will code the construction (and $\alpha_y$) into the object $G$ constructed, and then argue that the construction is recoverable from the oracle $G\oplus x^{(n+1)}$ instead. To do this we will need to modify the Kumabe-Slaman forcing conditions to include the parameter $\lambda\in\baire$ (which codes key facts about the construction) and the parameter $\ep\in\mathbb{Q}^+$ (which records the amount of ``measure" allowed  to be added to $\Dom(\Phi)$ by future conditions).
 
 We will construct a partial continuous function $\Phi_G:[\mathbf{H}]\partialf 2^\omega$ by finite extensions (using conditions similar to the Kumabe-Slaman forcing conditions).
The set $G$ will be defined as the join of $\Phi_G$, $x$ and $\lambda_G$, where recall that $x$ is assumed to be in $2^\omega$.
At the end we ensure that $\Phi_G$ is consistent along $y$, so that $G\oplus y\geq_M \Phi_G(\alpha_y)$ by the above claim. This latter object is of total degree and is in fact an element of $2^\omega$. 

We now indicate the required modifications to ensure that $\alpha_y\leq_T G\oplus x^{(n+1)}$, so that $G\oplus y\leq_MG\oplus x^{(n+1)}$ as required. The idea is to code the name $\alpha_y$ into $G$, in a way where $G\oplus x^{(n+1)}$ can recover the construction as well as the coding location of each value of $\alpha_y$.
The problem is that that knowing $G$ alone (or even together with $x^{(n+1)}$) is only enough to read off the axioms in $\Phi_G$. It cannot, for instance, compute the sequence $(\Phi_{p_i})_{i\in\omega}$ such that $\Phi_G=\bigcup_i\Phi_{p_i}$, since we cannot know where one forcing condition finishes and the next begins (we need $y$ as oracle to find this). Hence it is not immediately obvious why $G\oplus x^{(n+1)}$ can re-construct a construction in which $\alpha_y$ is coded. This is where the parameter $\lambda$ is needed (together with an innovative use of the non-totality assumption).

Our modified version of Kumabe-Slaman forcing is defined as follows:

\begin{definition}[Modified Kumabe-Slaman Forcing]
A {\em modified Kumabe-Slaman forcing condition} $p\in\mathbb{P}_{\rm KS}$ is a quadruple $(\Phi_p,\mathbf{X}_p,\lambda_p,\ep_p)$ consisting of the following:
\begin{enumerate}
\item $\Phi_p\subseteq\omega\times 2\times\mathbf{H}$ is a use-monotone finite Turing functional on $\mathbf{H}$,
\item $\mathbf{X}_p$ is a finite subset of $[\mathbf{H}]$, that is, a finite set of names,
\item $\lambda_p$ is a finite string, that is, $\lambda_p\in\omega^{<\omega}$,
\item $\varepsilon_p$ is a positive rational number less than or equal to $1$.
\end{enumerate}

For modified Kumabe-Slaman forcing conditions $p,q\in\mathbb{P}_{\rm KS}$, we say that {\em $q$ is stronger than $p$} (written as $q\leq p$) if
\begin{enumerate}
\item $\Phi_p\subseteq\Phi_q$, $\mathbf{X}_p\subseteq\mathbf{X}_q$, $\lambda_p\preceq\lambda_q$ and $\ep_q\leq \ep_p$,
\item the length of any $\sigma\in\Dom(\Phi_q)\setminus\Dom(\Phi_p)$ is greater than that of any $\tau\in\Dom(\Phi_p)$,
\item $\sum \left\{2^{-|\sigma|}\mid \sigma\in \Dom(\Phi_q)\setminus\Dom(\Phi_p)\right\} + \varepsilon_q \leq \varepsilon_p$,

    
\item no $\sigma\in\Dom(\Phi_q)\setminus\Dom(\Phi_p)$ meets $\mathbf{X}_p$.
 
\end{enumerate}
\end{definition}



If $q\leq p$ then we call $\sum \left\{2^{-|\sigma|}\mid \sigma\in \Dom(\Phi_q)\setminus\Dom(\Phi_p)\right\}$ the \emph{amount added by $\Dom(\Phi_q)$}. We say that \emph{$\sigma$ meets $\beta$} if $B^\ast_{\beta\rs|\sigma|}\cap B^\ast_\sigma\neq\emptyset$, and that \emph{$\sigma$ meets $\mathbf{X}$} if $\sigma$ meets some $\beta\in\mathbf{X}$.

It is easy to see that $(\mathbb{P}_{\rm KS},\leq)$ forms a partially ordered set.
We will construct a sequence $(p_n)_{n\in\omega}$, and then define $\Phi_G:=\bigcup_n\Phi_{p_n}$ and $\lambda_G:=\bigcup_n\lambda_{p_n}$.
The desired set will be obtained as $G=x\oplus\Phi_G\oplus\lambda_G$.
We use the symbol $\genname$ to denote such a generic element $x\oplus\Phi_G\oplus\lambda_G$.
We have to make sure that our modification of the definition of Kumabe-Slaman forcing does not increase the complexity of the forcing relation.
Hereafter, given a modified Kumabe-Slaman forcing condition $p=(\Phi_p,\mathbf{X}_p,\lambda_p,\ep_p)$, we write $p^{\mathbf{0}}$ for $(\Phi_p,\emptyset,\lambda_p,\ep_p)$.

\begin{definition}[see also Shore-Slaman {\cite[Definition 2.8]{shore_slaman_defining_the_Turing_Jump}}]
Let $p\in\mathbb{P}_{\rm KS}$ be a condition, and $\psi(\genname)$ be a sentence of the form $\forall m\theta(m,\genname)$.
Given $\vec{\tau}=(\tau_1,\dots,\tau_k)$ a sequence of strings in $\mathbf{H}$ all of the same length (i.e. $|\tau_i|=|\tau_j|$ for every $1\leq i,j\leq k$), we say that $\vec{\tau}$ is {\em essential to (force the sentence) $\neg\psi(\genname)$ over $p$ (without $\mathbf{X}_p$)} when the following condition holds.
For any condition $q\in\mathbb{P}_{\rm KS}$, if
\[q\leq p^\mathbf{0}\mbox{ and }(\exists m\in\omega)\;q\Vdash\neg\theta(m,\dot{\mathbf{r}}_{gen}),\]
then there are $\sigma\in\Dom(\Phi_q)\setminus\Dom(\Phi_p)$ and $j\leq k$ such that $\sigma$ meets $\tau_j$.
\end{definition}

Note that in the Hilbert cube one can computably decide whether or not two basic open balls (with respect to $\mathbf{H}$) intersect. 

Let $T(p,\psi,k)$ be the tree of length $k$ vectors $\vec{\tau}\in\mathbf{H}^k$ which are essential to $\neg\psi(\genname)$ over $p$.
As in Shore-Slaman \cite[Lemma 2.11]{shore_slaman_defining_the_Turing_Jump}, we can see that for any forcing condition $p\in\mathbb{P}$, any $\Pi^0_{n+1}$ sentence $\psi$, and $k\in\omega$, $T(p,\psi,k)$ is a $\Pi^0_{n+1}(x)$ (hence $\Pi^0_1(x^{(n)})$) subtree of $\left(\mathbf{H}^k\right)^{<\omega}$ uniformly in $\Phi_p$, $\lambda_p$, $\ep_p$, $\psi$, and $k$. Note that $T(p,\psi,k)$ is recursively branching (since $\mathbf{H}$ is) and so K\"{o}nig's lemma can be applied.

\medskip

\noindent
{\bf Construction.}
Now we start to describe our construction.
The construction begins with the empty condition $p_0$, that is, $\Phi_{p_0}=\emptyset$, $\mathbf{X}_{p_0}=\emptyset$, $\lambda_{p_0}=\langle\rangle$, and $\ep_{p_0}=1$.
At stage $s$ of our construction, inductively assume that we are given:
\begin{itemize}
\item the condition $p=(\Phi_p,\mathbf{X}_p,\lambda_p,\ep_p)$, 
\item the indices computing each member of $\mathbf{X}_p$ from $x^{(n+1)}$, and, 
\item that $\mathbf{X}_p$ does not include any name of the point $y\in\yy$.
\end{itemize}
Let $\psi(\genname)$ be the $s^{\text{th}}$ $\Pi^0_{n+1}$ sentence, so it is of the form $\forall m\theta(m,\genname)$ for some $\Sigma^0_n$ formula $\theta$.
Given $p\in\mathbb{P}_{\rm KS}$, we now describe our strategy to find the next condition $r\leq p$ which forces either $\psi$ or $\neg\psi$.
If $y\oplus x^{(n+1)}$ is total, the parameters $\lambda_p$ and $\ep_p$ will have no role.
We will supply more details about this after describing the non-total case.
So now we first suppose that $y\oplus x^{(n+1)}$ is not total. Fix any name $\alpha_y$ of $y$. The construction will be in two phases. In Phase 1 we force the next sentence without extending $\Phi_G$ along any name for $y$. This is necessary to ensure that $\Phi_G$ is consistent along $y$. In Phase 2 we extend $\Phi_G$ along $\alpha_y$ by coding the result of Phase 1 into $\Phi_G(\alpha_y)$ and $\lambda_G$.

\medskip

\noindent
\emph{{\bf Phase 1} \tu{(}Forcing the next sentence\tu{):}}
Suppose $\mathbf{X}_p=\{\beta_0,\dots,\beta_{k-1}\}$.
For each $i\in\omega$, we first define $q_i\leq p$ as follows:
\[q_i=(\Phi_p,\mathbf{X}_p,\lambda_p\fr i,2^{-i}\ep_p),\]
We check if there exists some $i$ such that $T(q_i,\psi,k+1)$ is infinite.

{\em Case 1 \tu{(}Forcing $\psi$\tu{)}.}
If $T(q_i,\psi,k+1)$ is infinite for some $i$, as in Shore-Slaman \cite[Corollary 2.12]{shore_slaman_defining_the_Turing_Jump}, we claim that $T(q_i,\psi,k+1)$ has an $x^{(n+1)}$-computable infinite path $\mathbf{X}$ which does not include a name of the point $y$.
By applying the Cone Avoiding Lemma \ref{thmconeavoidance} to our $\Pi^0_1(x^{(n)})$ set $P=[T(q_i,\psi,k+1)]$ in the effectively compact space $[\mathbf{H}]^{k+1}$, we have $\mathbf{Y}\in P$ such that $y\not\leq_M\mathbf{Y}\leq_My\oplus x^{(n+1)}$.
Suppose $\mathbf{Y}=\{\gamma_0,\dots,\gamma_{k}\}$.
The property $y\not\leq_M\mathbf{Y}$ implies that for every $i\leq k$, $\rho_\yy(\gamma_i)\not=y$ (where recall from Example \ref{exa:rep-Hilbert-cube} that $\rho_\yy$ is a representation of $\yy=[0,1]^\om$ with domain $\mathbf{H}$).
Therefore, some open ball separates $\rho_\yy(\gamma_i)$ and $y$, that is, there has to be $l_i$ such that $y\not\in B^\ast_{\gamma_i\rs l_i}$.
Put $l=\max_{i\leq k}l_i$ and we consider the subtree $\tilde{T}$ consisting of $(\tau_i)_{i\leq k}\in T(q_i,\psi,k+1)$ such that for each $i\leq k$, $\tau_i$ is comparable with $\gamma_i\rs l$.
Since $\tilde{T}$ is still a nonempty $\Pi^0_1(x^{(n)})$ subtree of $\mathbf{H}^{k+1}$, there is an $x^{(n+1)}$-computable infinite path $\mathbf{X}$ through $\tilde{T}$.
Obviously, $\mathbf{X}$ does not contain a name of $y$.

As in Shore-Slaman \cite[Lemma 2.10 (2)]{shore_slaman_defining_the_Turing_Jump}, by adding $\mathbf{X}$ to $\mathbf{X}_p$, we can force the sentence $\psi$ without changing $\Phi_p$, $\lambda_p\fr i$, and $2^{-i}\ep_p$.
Then, we take the next condition to be
\[q=(\Phi_p,\mathbf{X}_p\cup\mathbf{X},\lambda_p\fr i,2^{-i}\ep_p).\]

{\em Case 2 \tu{(}Forcing $\neg\psi$\tu{)}.}
Now suppose that for every $i$, $T(q_i,\psi,k+1)$ is finite.
Then, for each $i$ there is $l$ such that $(\beta_0\rs l,\dots,\beta_{k-1}\rs l,\alpha_y\rs l)\not\in T(q_i,\psi,k+1)$.
In particular, for every $i$ there exists $\widehat{r}_i\leq q_i^\mathbf{0}$ and some $m$ such that $\widehat{r}_i\Vdash\neg\theta(m,\genname)$ and such that ${\rm Dom}(\Phi_{\widehat{r}_i})\setminus{\rm Dom}(\Phi_p)$ does not add any axiom which meets $\mathbf{X}_p\cup\{\alpha_y\}$, and adds an amount which is no more than $2^{-i}\varepsilon_p$. We would like to take (any one of the) $\widehat{r}_i$ as our next condition because it enables us to force $\neg\psi$ without adding axioms meeting $\alpha_y$. 
Unfortunately, the search for $\widehat{r}_i$ requires knowing $\alpha_y$ and hence cannot be performed recursively in $x^{(n+1)}$ (in fact, not even recursively in $y\oplus x^{(n+1)}$ unless $y\oplus x^{(n+1)}$ is total). We must replace $\{\widehat{r}_i\}$ with conditions $\{r_i\}$ that are easier to find.

We now note that as in Shore-Slaman \cite[Lemmata 2.10 and 2.11]{shore_slaman_defining_the_Turing_Jump}, the following two conditions are equivalent for a given $q\in\mathbb{P}_{\rm KS}$:
\begin{enumerate}
\item There is a condition $r\leq q$ such that $r\Vdash\neg\theta(m,\genname)$ for some $m\in\omega$.
\item There is a condition $r=(\Phi_r,\emptyset,\lambda_r,\ep_r)\leq q^\mathbf{0}$ such that the tree $T(r,\neg\theta(m,\genname),l)$ is infinite for some $l\geq|\mathbf{X}_q|$ and $m\in\omega$.
\end{enumerate}

Now, for every $i$, instead of searching for $\widehat{r}_i$, we let $r_i=(\Phi_{r_i},\emptyset,\lambda_{r_i},\ep_{r_i})\leq q_i^\mathbf{0}$, $l$ and $m$ be the first triple found such that the tree $T(r_i,\neg\theta(m,\genname),l)$ is infinite, and such that ${\rm Dom}(\Phi_{{r}_i})\setminus{\rm Dom}(\Phi_p)$ does not add any axioms comparable with $\mathbf{X}_p$ (that is, we allow $\Phi_{{r}_i}$ to contain an axiom meeting $\alpha_y$), and adds an amount no more than $2^{-i}\varepsilon_p$.
This search is effective given $\Phi_p$, $\lambda_p$, $\ep_p$, the oracle $x^{(n+1)}$, and the indices computing each member of $\mathbf{X}_p$ from $x^{(n+1)}$. 
Note that one can effectively decide whether two basic open balls intersect.

\begin{claim}
If $y\oplus x^{(n+1)}$ is not total, then there is $i\in\omega$ such that $\Dom(\Phi_{r_i})\setminus\Dom(\Phi_p)$ does not meet $\alpha_y$, that is, for all $\sigma\in\Dom(\Phi_{r_i})\setminus\Dom(\Phi_p)$, $B^\ast_\sigma$ has no intersection with $B^\ast_{\alpha_y\rs{|\sigma|}}$.
\end{claim}

\begin{proof}
Otherwise, for every $i\in\omega$, there is some $\sigma\in\Dom(\Phi_{r_i})\setminus\Dom(\Phi_p)$ such that we have $B^\ast_\sigma\cap B^\ast_{\alpha_y\rs{|\sigma|}}\not=\emptyset$.
Let $B'_\sigma$ be the open ball with the same center as $B^\ast_\sigma$, but the radius of $B'_\sigma$ is $3\cdot 2^{-|\sigma|}$.
Then, $B^\ast_\sigma\cap B^\ast_{\alpha_y\rs{|\sigma|}}\not=\emptyset$ implies that $y\in B^\ast_{\alpha_y\rs{|\sigma|}}\subseteq B'_\sigma$ since for every $\tau$, the diameter of $B^\ast_\tau$ is less than $2^{-|\tau|}$ by our choice of such balls in the definition of $\mathbf{H}$.
In this way, we will be able to obtain a name for $y$.
That is,
\[y\in E_i:=\bigcup \left\{B'_\sigma \mid \sigma\in \Dom(\Phi_{r_i})\setminus\Dom(\Phi_p)\right\},\]

Note that the total sum of the radius of all such balls $B'_\sigma$ is at most $3\cdot 2^{-i}\varepsilon_p$ by the choice of $r_i$. Now we say that two basic balls $B'_\sigma$ and $B'_\tau $ in $E_i$ are in the same connected component if there exists a sequence of balls $B'_{\sigma_0}, B'_{\sigma_1}, \cdots , B'_{\sigma_N}$ all of which are in $E_i$, and such that $B'_{\sigma_0}=B'_\sigma$, $B'_{\sigma_N}=B'_\tau$, and for every $i< N$, $B'_{\sigma_i}\cap B'_{\sigma_i+1}\neq\emptyset$.

We enumerate all distinct connected components of $E_i$ as $U_{i,0}$, $U_{i,1},\dots$ in some fixed way. 
Now note that the maximum distance between any two points in the same component is at most $12\cdot 2^{-i}\varepsilon_p<2^{-i+4}$. Thus we let $C_{i,j}$ be a ball of radius $2^{-i+4}$ which contains $U_{i,j}$.

We define $z\in\omega^\omega$ by $z(i)$ to be the unique component such that $y\in U_{i,z(i)}$.
We shall show that $y\oplus x^{(n+1)}$ is $M$-equivalent to the total degree $z\oplus x^{(n+1)}$ for a contradiction.
First we see that $y\leq_M z\oplus x^{(n+1)}$ because we can recursively in $x^{(n+1)}$ (with finitely many fixed parameters)
compute $\Phi_{r_i}$ and then the sequence $\{C_{i,j}\}$ (here we do need $\emptyset'$ to tell apart the different components).
We then see that $\bigcap_i C_{i,z(i)}=\{y\}$ since the radius of $C_{i,z(i)}$ is $2^{-i+4}$.
On the other hand we have $z\leq_M y\oplus x^{(n+1)}$, because once again $x^{(n+1)}$ allows us to compute the components $\{U_{i,j}\}$, and for each $i$ and any name for $y$ we can compute the unique $z(i)$ such that $y\in U_{i,z(i)}$ (since the components are pairwise disjoint).
This means that $y\oplus x^{(n+1)}$ has total degree, contrary to the assumption.

Notice that the choice of $C_{i,j}$ and $U_{i,j}$ depends on the sequence $\{r_i\}$, where $r_i$ is  fixed to be the first axiom found under a $x^{(n+1)}$-effective search.\newcommand{\qedclaim}{\qedsymbol(claim)}
\renewcommand{\qed}{\hfill\qedclaim}
\end{proof}

Thus we conclude that there is some $i_0$ such that  $\Dom(\Phi_{r_{i_0}})\setminus\Dom(\Phi_p)$ does not meet $\alpha_y$.
By the property of $r_{i_0}$, the tree $T(r_{i_0},\neg\theta(m,\genname),l)$ is infinite for some $m\in\omega$ and $l\geq k$ and so, we may take the next forcing condition to be
\[q=\left(\Phi_{r_{i_0}},\mathbf{X}_p\cup\mathbf{X},\lambda_{r_{i_0}},\ep_{r_{i_0}}\right)\]
where $\mathbf{X}$ is an $l$-tuple of $x^{(n+1)}$-computable names containing no name of $y$, which is obtained by applying Lemma \ref{thmconeavoidance} to the tree $T(\Phi_{r_{i_0}},\neg\theta(m,\genname),l)$ as in Case 1. We can check easily that $q\leq q_i\leq p$.

\medskip

\noindent
{\bf Phase 2} {\em \tu{(}Coding $\alpha_y$, $\mathbf{X}$ and the result of Phase 1\tu{)}}.
Suppose Phase 1 produces the condition $q=(\Phi_q,\mathbf{X}_q,\lambda_q,\varepsilon_q)$.
For $\mathbf{X}_q=\{\beta_0,\dots,\beta_{m-1}\}$, we choose the least $b$ and the least $u\geq s$ such that 
$B^\ast_{\alpha_y\rs u}$ is disjoint from $B^\ast_{\beta_0\rs b},B^\ast_{\beta_1\rs b},\dots, B^\ast_{\beta_{m-1}\rs b}$.
Note that such $b,u$ exist since $\mathbf{X}_q$ does not contain a name for $y$.
Now consider an index $j_0$ which computes each element of $\mathbf{X}_q$ from $x^{(n+1)}$, and for each string $\sigma\in\omega^{<\omega}$, let $\sigma^+$ be the string defined by $\sigma^+(n)=\sigma(n)+1$ for every $n\in\omega$.
Then we define the next condition as follows:
\[r=\left(\Phi_q\cup\{(s,k,\alpha_y\rs u)\},\mathbf{X}_q,\lambda_q\fr j_0\fr (\alpha_y\rs u)^+\fr 0,\varepsilon_q-2^{-u}\right)\] where $k=0$ or $1$ depending on how the sentence was forced in Phase 1.
That is, this construction ensure that $\Phi_r(\alpha_y\rs u)(s)=0$ if and only if $r$ forces the $s^{\text{th}}$ $\Pi^0_{n+1}$ sentence $\psi$ to be true.
It is easy to check that $r\leq q$.

\medskip

\noindent
{\bf Verification}.
Eventually, our construction produces a sequence $(\Phi_{p_n},\mathbf{X}_{p_n},\lambda_{p_n},\ep_{p_n})_{n\in\omega}$.
We define $G$ as $x\oplus\Phi_G\oplus\lambda_G$, where $\Phi_G=\bigcup_n\Phi_{p_n}$ and $\lambda_{G}=\bigcup_n\lambda_{p_n}$.
As in Shore-Slaman \cite{shore_slaman_defining_the_Turing_Jump}, our construction (in Phase 2) ensures that $\Phi_G(\alpha_y)=G^{(n+1)}$ since $\Phi_G(\alpha_y)(s)=0$ if and only if $\psi(G)$ is true, where $\psi$ is the $s^{\text{th}}$ $\Pi^0_{n+1}$ formula with one parameter.
We first check the consistency of $\Phi_G$.

\begin{claim}
$\Phi_G$ is consistent along $y$.
\end{claim}

\noindent
\textit{Proof.}
A new computation is enumerated into $\Phi_G$ only when we construct $r_{i_0}$ from $p$ in Case 2 of Phase 1, and when we construct $r$ from $q$ in Phase 2.
By the property of $r_{i_0}$, $\Phi_{r_{i_0}}$ does not add any new computation along $y$.
It is also clear that our action in Phase 2 does not add any inconsistent computation.\qed(claim)

\medskip

By combining the above claim and our first claim, we obtain $G^{(n+1)}=\Phi_G(\alpha_y)\leq_MG\oplus y$.
We also have that $G\oplus x^{(n+1)}\leq_MG^{(n+1)}$ since $x\leq_MG$.
We next claim that the construction is recoverable in the oracle $G\oplus x^{(n+1)}$.

\begin{claim}
If $y\oplus x^{(n+1)}$ is not total, then we have $\alpha_y\leq_M G\oplus x^{(n+1)}$.
In particular, $y\leq_MG\oplus x^{(n+1)}$ holds.
\end{claim}

\begin{proof}
Assume that at stage $s$ of the construction we have computed the finite string representing $\Phi_p$, the index $j_p$ corresponding to elements of $\mathbf{X}_p$ (relative to $x^{(n+1)}$), the finite string $\lambda_p$, and the rational $\varepsilon_p$.
We first decode $i_0$ from $\lambda_G$, which is a unique number $i_0$ such that $\lambda_p\fr i_0\prec\lambda_G$.
We use $x^{(n+1)}$ to check if $T(p_{i_0},\psi,k+1)$ is infinite.
If infinite, we know that the output from Phase 1 is  $\left(\Phi_p,\mathbf{X}_p\cup\mathbf{X},\lambda_p\fr i_0,2^{-i_0}\varepsilon_p\right)$ for some $\mathbf{X}$, and so we can decode $j_0$ and $\alpha_y\rs u$ since an initial segment of $\lambda_G$ is now of the form $\lambda_p\fr i_0\fr j_0\fr (\alpha_y\rs u)^+\fr 0$ by our action in Phase 2.
By using $\alpha_y\rs u$ and $\Phi_G$, we can also decode the outcome $z$ in Phase 1.
These codes tell us the full information on the next forcing condition $r$.

If $x^{(n+1)}$ tells us that the tree $T(p,\psi,k)$ is finite, then we know that the  output from Phase 1 is  $\left(\Phi_{r_{i_1}},\mathbf{X}_p\cup\mathbf{X},\lambda_{r_{i_1}},\ep_{r_{i_1}}\right)$ for some $i_1\in\omega$. But by the construction we have that $r_{i_1}\leq q_{i_1}$ and since $\lambda_{q_{i_1}}=\lambda_p\fr i_1$, we conclude that $i_0=i_1$. Thus we can search for $r_{i_0}$ because we are given the oracle $x^{(n+1)}$ and the indices for members of $X_p$. The first $\Phi_{r_{i_0}}$ found must avoid $\alpha_y$ by our claim in Phase 1.
Having found $r_{i_0}$, we can proceed as above to recover $j_0$ and $\alpha\rs u$ and full information on the next forcing condition $r$.
\newcommand{\qedclaim}{\qedsymbol(claim)}
\renewcommand{\qed}{\hfill\qedclaim}
\end{proof}

Consequently, if $y\oplus x^{(n+1)}$ is not total, we obtain $G\oplus x^{(n+1)}\equiv_MG^{(n+1)}\equiv_MG\oplus y$ as desired. An analysis of the proof reveals that $\alpha_y$ is coded in segments into $\lambda_G$. In between segments of $\lambda_G$ where $\alpha_y$ is coded, $\lambda_G$ was determined by the outcome of forcing $\psi$.

Now we return to the remaining case where $y\oplus x^{(n+1)}$ is total.
In this case, we will not have the property $y\leq_MG\oplus x^{(n+1)}$.
Instead, we proceed via a $y\oplus x^{(n+1)}$-computable construction.
Recall that in the previous case where $y\oplus x^{(n+1)}$ is not total, we have used $\lambda_p$ to code our construction; however, if $y\oplus x^{(n+1)}$ is total, we will use $y\oplus x^{(n+1)}$ itself to recover our construction.
If $y\oplus x^{(n+1)}$ is total, then it computes a canonical name of $y$, that is, we can choose a name $\alpha_y\in\mathbf{H}$ of $y\in\yy$ such that $\alpha_y\leq_Ty\oplus x^{(n+1)}$.
In Case 1 of Phase 1, we can find $\mathbf{X}$ by using $\alpha_y\oplus x^{(n+1)}$ as in Shore-Slaman \cite[Corollary 2.12]{shore_slaman_defining_the_Turing_Jump}.
We also skip the construction of $r_i$ in Case 2 of Phase 1, and just use $\alpha_y\oplus x^{(n+1)}$ to find $\widehat{r}_0$, and take the next forcing condition to be $(\Phi_{\widehat{r}_0},\mathbf{X}_p\cup\mathbf{X},\lambda_{\widehat{r}_0},\ep_{\widehat{r}_0})$ as in the last paragraph in Case 2 of Phase 1.
We can find such a forcing condition by an $(\alpha_y\oplus x^{(n+1)})$-computable way. The parameters $u$ and $b$ in Phase 2 can also be found using oracle $\alpha_y\oplus x^{(n+1)}$.
Consequently, the entire construction is $(\alpha_y\oplus x^{(n+1)})$-computable.
Since the construction decides the truth of all $\Pi^0_{n+1}$ sentences about $G$, this implies that $G^{(n+1)}\leq_My\oplus x^{(n+1)}$.
Our construction again implies $\Phi_G(\alpha_y)=G^{(n+1)}$, and therefore $G^{(n+1)}\leq_MG\oplus y$ as before.
By combining these inequalities with $x\leq_MG$, we obtain that $G^{(n+1)}\leq_MG\oplus y\equiv_My\oplus x^{(n+1)}$.
Now, since $y\oplus x^{(n+1)}$ is total, as in Shore-Slaman \cite[Theorem 2.3]{shore_slaman_defining_the_Turing_Jump}, we can apply the Friedberg Jump Inversion Theorem to $y\oplus x^{(n+1)}\geq_TG^{(n+1)}$ to get $\tilde{G}\geq_TG$ such that $\tilde{G}^{(n+1)}\equiv_Ty\oplus x^{(n+1)}$.
One can see that  $\tilde{G}^{(n+1)}\equiv_M\tilde{G}\oplus y\equiv_My\oplus x^{(n+1)}$ as desired.
This concludes the proof of the theorem for all finite cases.
Here, note that if $y\oplus x^{(n+1)}$ is not total, then it is impossible to satisfy the last equality since $\tilde{G}^{(n+1)}$ must be total.

For transfinite cases, it only remains to calculate the complexity of our forcing relation, that is, the complexity of the tree $T(p,\psi,k)$ for a computable $\Pi^0_\xi$ formula $\psi$.
The calculation is straightforward and we omit it. This finishes the proof of Theorem \ref{thm:shoreslamanpolish}.
\end{proof}



\section{Borel transition in the codes and decomposability}

\label{section Borel transition in the codes}
\newcounter{remarks}

In this section we give the proof of Theorem \ref{theorem main}. First we prove Theorem \ref{theorem Borel transition in the codes} by applying results from effective descriptive set theory.
The first key lemma is obtained from the (effective) $\Pi^1_1$-uniformization theorem.

\begin{lemma}[{\cite[p. 367]{louveau_a_separation_theorem_for_sigma_sets}}]
\label{theorem Borel selection}
Suppose that \ca{X} and \ca{Y} are recursively presented metric spaces and that $P \subseteq \ca{X} \times \ca{Y}$ is $\Pi^1_1(\ep)$ for some $\ep \in \baire$, which satisfies the condition $(\forall x)(\exists y \in \del(\ep,x))P(x,y)$.
Then there exists a $\del(\ep)$-recursive function $f: \ca{X} \to \ca{Y}$ such that $P(x,f(x))$ for all $x \in \ca{X}$.
\end{lemma}

Another important tool is the following result of Louveau. Let us recall that a set $C$ \emph{separates $A$ from $B$} or that $A$ \emph{is separated from} $B$ by $C$ if $A \subseteq C$ and $C \cap B = \emptyset$.

\begin{lemma}[Louveau separation \cf \cite{louveau_a_separation_theorem_for_sigma_sets}]
\label{theorem louveau separation}
Suppose that \ca{X} is a recursively presented metric space and that $A, B$ are disjoint \sig \ subsets of \ca{X}. If $A$ is separated from $B$ by a $\boldp^0_\xi$ set, then $A$ is separated from $B$ by a $\Pi^0_\xi(\gamma)$ set for some $\gamma \in \del$.
\end{lemma}

We now prove our result on the Borel transition.

\begin{proof}[Proof of Theorem \ref{theorem Borel transition in the codes}]
We fix some $\ep \in \baire$ such that $\ca{X}$ and $\ca{Y}$ are $\ep$-recursively presented, $\ca{A}$ belongs to $\sig(\ep)$ and $f$ is $\del(\ep)$-recursive.
We define the sets $P, Q \subseteq \baire \times \ca{X}$ by
\begin{align*}
P(\alpha,x) \iff& x \in \ca{A} \ \& \ \univ^{\ca{Y}}_m(\alpha,f(x))\\
Q(\alpha,x) \iff& x \in \ca{A} \ \& \ \neg \univ^{\ca{Y}}_m(\alpha,f(x)).
\end{align*}
Since $\ca{Y}$ is $\ep$-recursively presented, the set $\univ^{\ca{Y}}_m$ is $\Sigma^0_m(\ep)$. It follows that the sets $P, Q$ are $\sig(\ep)$.
We now fix some $\alpha \in \baire$ for the discussion. It is clear that the sections $P_\alpha$, $Q_\alpha$ are disjoint $\sig(\ep,\alpha)$ subsets of $\om \times \ca{X}$. From our hypothesis about $f$ there exists a $\bolds^0_n$ set $R_{\alpha} \subseteq \ca{X}$ such that $P_{\alpha} = R_{\alpha} \cap \ca{A}$.

It is evident that the set $R_\alpha$ separates $P_\alpha$ from $Q_\alpha$.
It follows from Louveau's Theorem (Lemma \ref{theorem louveau separation}) relative to $\ep$ that $P_\alpha$ is separated from $Q_\alpha$ by some $\Sigma^0_n(\ep,\gamma)$ subset of $\ca{X}$, for some $\gamma \in \del(\ep,\alpha)$. The preceding $\Sigma^0_n(\ep,\gamma)$ set is the $\beta$-section of $\univ^{\om \times \ca{X}}_n$ for some $(\ep,\gamma)$-recursive $\beta \in \baire$. The latter $\beta$ is obviously in $\del(\ep,\alpha)$.

Overall we have that for all $\alpha \in \baire$ there exists $\beta \in \del(\ep,\alpha)$ such that $U(\alpha,\beta)$, where $U(\alpha,\beta)$ denotes that $\univ_{n,\beta}^{\ca{X}}$ separates $P_\alpha$ from $Q_\alpha$.
It is not hard to see that $U$ is a $\pii(\ep)$ subset of $\baire \times \baire$.
It follows from Lemma \ref{theorem Borel selection} that there exists a Borel-measurable function $u: \baire \to \baire$ such that $U(\alpha,u(\alpha))$ for all $\alpha$. This function $u$ satisfies the required properties.
\end{proof}

Using the Borel-uniform transition in the codes and the theory of continuous degrees one can give the analogue of \cite[Lemma 2.6]{kihara_decomposing_Borel_functions_using_the_Shore_Slaman_join_theorem} in Polish spaces.

\begin{lemma}\label{lem:turing-degree-analysis}
Let \ca{X}, \ca{Z} be recursively presented metric spaces and that \ca{A} is a $\Sigma^1_1$ subset of \ca{X}. Suppose that a function $f:\mathcal{A}\to\mathcal{Z}$ satisfies that condition $f^{-1} \bolds^0_{m+1} \subseteq \bolds^0_{n+1}$ holds Borel-uniformly.
Then, there are $z\in 2^\omega$ and a $z$-computable ordinal $\xi$ such that we have $(f(x)\oplus q)^{(m)}\leq_M(x\oplus q^{(\xi)})^{(n)}$ for all $q\geq_Tz$.
\end{lemma}

\begin{proof}
Fix a sufficiently powerful oracle $\ep$ such that $\mathcal{A}$ and $\mathcal{Z}$ are $\ep$-computably isomorphic to subspaces of $\ep$-recursively presented Polish spaces $\mathcal{X}$ and $\mathcal{Y}$.
Let $f:\mathcal{A}\to\mathcal{Z}$ be a function satisfying that $f^{-1}\bolds^0_{m+1}\subseteq\bolds^0_{n+1}$.
We can think of $f$ as a function from $\mathcal{A}$ into $\mathcal{Y}$ with the same property since every $\bolds^0_{n+1}$ set $S\subseteq\mathcal{Z}$ is of the form $\hat{S}\cap\mathcal{Z}$ for some $\bolds^0_{n+1}$ set $\hat{S}\subseteq\mathcal{Y}$.
By Theorem \ref{theorem Borel transition in the codes}, $f:\bolds^0_{m+1}\subseteq\bolds^0_{n+1}$ holds Borel uniformly; in particular, both $f:\bolds^0_{m}\subseteq\bolds^0_{n+1}$ and $f:\boldp^0_{m}\subseteq\bolds^0_{n+1}$ hold Borel uniformly.
Therefore, we have Borel measurable transitions $u,v$ such that for all $x\in\mathcal{A}$ and all $(i,p)\in \omega \times \baire$,
\begin{align*}
f(x)\in\univ^\mathcal{Y}_{m,\cn{i}{p}}
\iff&\;x\in\univ^\mathcal{X}_{n+1,u(\cn{i}{p})}\\
\iff&\;x\not\in\univ^\mathcal{X}_{n+1,v(\cn{i}{p}))}.
\end{align*}
Since $u$ and $v$ are Borel, there is $z\geq_T\ep$ such that $u$ and $v$ are $\Sigma^0_a(z)$-recursive for some $a\in\mathcal{O}^z$, where $\mathcal{O}^z$ is Kleene's $\mathcal{O}$ relative to $z$, and the lightface pointclass $\Sigma^0_a(z)$ is defined in a straightforward manner.
By our definition of the jump operation, and the above equivalences it follows that $J_\yy^{(m),p\oplus z}(f(x))\leq_T J_\xx^{(n),(p\oplus z)^{(\xi)}}(x)$ for all $p\in 2^\omega$, where $\xi$ is a countable ordinal coded by $a$.
By Lemma \ref{prop:jump-equivalence}, we can restate this inequality as $(f(x)\oplus p\oplus z)^{(m)}\leq_M(x\oplus(p\oplus z)^{(\xi)})^{(n)}$ for all $p\in 2^\omega$.
\end{proof}

To show Theorem \ref{theorem main}, we will need first to extend a result of \cite[Lemma 2.5]{kihara_decomposing_Borel_functions_using_the_Shore_Slaman_join_theorem} on canceling out Turing jumps. Since the inequality of the degrees that we obtain is slightly different, we will need to utilize not only the Shore-Slaman Join Theorem (in the continuous degrees; Theorem \ref{thm:shoreslamanpolish}) like in \cite{kihara_decomposing_Borel_functions_using_the_Shore_Slaman_join_theorem} but also -as we mentioned above- other degree-theoretic results, namely the non-uniformity feature of the enumeration reducibility (Lemma \ref{lem:unif-nonunif}), the almost totality of continuous degrees (Lemma \ref{lem:nonsplitting}), the Friedberg Jump Inversion Theorem (Lemma \ref{lem:jump-inversion}) and the jump inversion property of a generic name (Lemma \ref{prop:nameandjump}).

\begin{lemma}[The Cancellation Lemma]\label{lem:deg-Shore-Slaman}
Suppose that $\xx$ and $\yy$ are recursively presented metric spaces, and fix $x\in\xx$, $y\in\yy$, and $z\in 2^\omega$.
If $(y\oplus g)^{(m)}\leq_T(x\oplus g^{(\xi)})^{(n)}$ for all $g\geq_Tz$, then $y\leq_T(x\oplus z^{(\xi)})^{(n-m)}$.
\end{lemma}

\begin{proof}
Otherwise, $y\not\leq_M(x\oplus z^{(\xi)})^{(n-m)}$.
We claim that there is $p\geq_Mz$ such that $y\not\leq_Mp^{(\xi+n-m)}$ and $x\leq_Mp^{(\xi)}$.
To see this, by Lemma \ref{lem:nonsplitting}, since $z^{(\xi)}$ is total, either $x\oplus z^{(\xi)}$ is total or else $x>_Mz^{(\xi)}$ holds.
If $x\oplus z^{(\xi)}$ is total, by the $\xi^{\text{th}}$ Friedberg jump inversion theorem on $2^\omega$ (see Lemma \ref{lem:jump-inversion}), there is $p\geq_Tz$ such that $p^{(\xi)}\equiv_Tx\oplus z^{(\xi)}$.
This $p$ clearly satisfies the desired condition.
Now we assume that $x>_Mz^{(\xi)}$, which implies $y\not\leq_Mx^{(n-m)}$.
If $n=m$, then, by Lemma \ref{lem:unif-nonunif}, there is $\hat{x}\in\omega^\omega$ such that $x\leq_M\hat{x}$ and $y\not\leq_M\hat{x}$.
If $n>m$, then let $\hat{x}\in\omega^\omega$ be a generic $\mathcal{X}$-name of $x$ in Lemma \ref{prop:nameandjump}; therefore, $\hat{x}'\equiv_TJ_\xx(x)$.
By Lemma \ref{prop:jump-equivalence}, $x^{(n-m)}$ can be interpreted as $TJ^{(n-m-1)}\circ J_\xx(x)$ since $n-m\geq 1$, and therefore, we have $\hat{x}^{(n-m)}\equiv_T x^{(n-m)}$.
This implies $y\not\leq_M\hat{x}^{(n-m)}$.
In any cases, we have $\hat{x}\in\omega^\omega$, $\hat{x}\geq_Mx$ and $y\not\leq_M\hat{x}^{(n-m)}$.
Then, by the $\xi^{\text{th}}$ Friedberg jump inversion theorem on $2^\omega$ (Lemma \ref{lem:jump-inversion}), there is $p\geq_Tz$ such that $p^{(\xi)}\equiv_T\hat{x}$ since $\hat{x}\geq_Mx\geq_Mz^{(\xi)}$.
Then we have that $y\not\leq_Mx^{(n-m)}\equiv_Tp^{(\xi+n-m)}$ and $x\leq_M\hat{x}\equiv_Tp^{(\xi)}$.

Therefore, by the Shore-Slaman Join Theorem in the continuous degrees (Theorem \ref{thm:shoreslamanpolish}), there is $g\geq_Mp$ such that $g^{(\xi+n-m+1)}\leq_Mg\oplus y$.
This implies $g^{(\xi+n+1)}\equiv_T(g^{(\xi+n-m+1)})^{(m)}\leq_T(g\oplus y)^{(m)}$ by Lemma \ref{fact:jumpproperties}.
Note also that $g\geq_Mp$ implies that $g\geq_Tz$ and $g^{(\xi)}\geq_Mp^{(\xi)}\geq_M x$.
Therefore,
\[(y\oplus g)^{(m)}\geq_Tg^{(\xi+n+1)}>_Tg^{(\xi+n)}\geq_T(x\oplus g^{(\xi)})^{(n)},\]
which contradicts our assumption.
\end{proof}

We are finally ready to give the proof of our decomposition result.

\begin{proof}[Proof of Theorem \ref{theorem main}]
Without loss of generality we may assume that the underlying spaces are recursively presented and that $\ca{A}$ is $\Sigma^1_1$. We have from Theorem \ref{theorem Borel transition in the codes} that the condition $f^{-1} \bolds^0_{m+1} \subseteq \bolds^0_{n+1}$ holds Borel-uniformly in the codes.
Therefore, by Lemma \ref{lem:turing-degree-analysis}, there is $z\in 2^\omega$ such that $(f(x)\oplus q)^{(m)}\leq_M(x\oplus q^{(\xi)})^{(n)}$ for all $q\geq_Tz$ and all $x \in \mathcal{A}$. By using Lemma \ref{lem:deg-Shore-Slaman}, we obtain
\[
f(x)\leq_M(x\oplus z^{(\xi)})^{(n-m)}
\]
for all $x \in \mathcal{A}$. As in \cite[Lemma 2.7]{kihara_decomposing_Borel_functions_using_the_Shore_Slaman_join_theorem} the function $f$ is decomposed into the $\mathbf{\Sigma}^0_{n-m+1}$-measurable functions $x\mapsto\Phi_e((x\oplus z^{(\xi)})^{(n-m)})$, $e\in\omega$ on the domains
\[
\mathcal{B}_e : = \set{x \in \mathcal{A}}{f(x) = \Phi_e((x\oplus z^{(\xi)})^{(n-m)})},
\]
That is, \ $\ca{A} = \bigcup_{e} \ca{B}_e$ and each $f \rs{\ca{B}_e}$ is $\bolds^0_{n-m+1}$-measurable. This proves the first assertion of the statement.

It remains to estimate the complexity of $\mathcal{B}_e$.
If $m=1$, then the assertion is trivial, so we may assume that $m\geq 2$.
We consider the $\bolds^0_{n-m+1}$-measurable functions $g_e:\mathcal{A}\to\yy$ defined by $g_e(x)=\Phi_e((x\oplus z^{(\xi)})^{(n-m)})$.
Let $\Delta_\yy=\{(y,y):y\in\yy\}$ be the diagonal set, which is closed in $\yy^2$.
Given functions $g$ and $h$, we write $(g,h)$ be the function $x\mapsto(g(x),h(x))$.
Then, clearly $\mathcal{B}_e=(f,g_e)^{-1}[\Delta_\yy]$, which is $\boldp^0_{n}$, since $f$ is $\bolds^0_{n}$-measurable and ${\rm dom}(g_e)$ is $\boldp^0_{n-m+2}\subseteq\boldp^0_{n}$ by the assumption $m\geq 2$.
Moreover, if $f$ is $\bolds^0_{n}$-measurable and $m\geq 3$ then $\mathcal{B}_e\in\boldp^0_{n-1}$ since $n-m+2\leq n-1$.

Finally assume that the graph ${\rm Gr}(f)$ of $f$ is $\bolds^0_{m}$. We show that in this case the preceding sets $\ca{B}_e$ are $\bolds^0_n$. This will finish the proof since we can decompose further $\ca{B}_e = \bigcup_{j} \ca{B}_{e,j}$, where $\ca{B}_{e,j}$ is $\boldd^{0}_{n-1}$.
It is clear that $\mathcal{B}_e=({\rm id},g_e)^{-1}[{\rm Gr}(f)]$.
Since each $g_e$ is $\bolds^0_{n-m+1}$-measurable and ${\rm Gr}(f)$ is $\bolds^0_m$, it follows that the set $\mathcal{B}_e$ is a $\bolds^0_{n-m+1 + (m -1)} = \bolds^0_n$ subset of \ca{A}.
\end{proof}

\begin{remark}\normalfont
If $n=m$, the proof of Kihara's Theorem \ref{theorem Kihara continuous transition} only requires the Posner-Robinson Join Theorem (that is, the Shore-Slaman Join Theorem for $n=0$).
However, a noteworthy fact is that our proof of Theorem \ref{theorem main} requires the full strength of the Shore-Slaman Join Theorem even in the case of $n=m$.
\end{remark}

It is also interesting to mention that this strategy of decomposing a given Borel-measurable function provides an upper bound for the complexity of the decomposing sets. In fact it is easy to check that an arbitrary function $f: \ca{A} \to \ca{Y}$ is decomposable to $\bolds^0_{k+1}$-measurable functions if and only if there is some $w \in \cantor$ such that for all $x \in \mathcal{A}$ is holds $f(x) \leq_M (x \oplus w)^{(k)}$ (see also \cite[Lemma 2.7]{kihara_decomposing_Borel_functions_using_the_Shore_Slaman_join_theorem}).
So, if the latter condition holds, the decomposing sets can be defined through the sets $\ca{B}_{e}$ of the preceding proof (with $k = n-m$). The complexity of the latter sets can be easily estimated by the complexity of $f$ as above: if $f$ is $\bolds^0_\xi$-measurable then the each $\ca{B}_e$ is a $\boldp^0_{\max\{\xi,k+1\}}$ set. Overall we conclude to the following.

\begin{proposition}
Let \ca{X}, \ca{Y} be Polish spaces, $\ca{A} \subseteq \ca{X}$ and $f: \ca{A} \to \ca{Y}$ be $\bolds^0_{\xi}$-measurable. If $f \in \dec(\bolds^0_k)$ for some $k < \om$, then $f \in \dec(\bolds^0_k,\boldd^0_{\max\{\xi,k+1\}})$.
\end{proposition}

\section{The Transfinite Case}

In this section, we show that with the necessary modifications the preceding results can be extended to the Borel pointclasses of transfinite order.
We begin this section by proposing the full decomposability conjecture including transfinite cases in scope.
The conjecture has already been mentioned by Kihara \cite[Problems 2.13 and 2.14]{kihara_decomposing_Borel_functions_using_the_Shore_Slaman_join_theorem}; however, his calculation of ordinals in transfinite cases contains a minor error.
We restate the correct version here:
\newtheorem*{gendec}{The Full Decomposability Conjecture}
\begin{gendec}
Suppose that $\ca{A}$ is an analytic subset of a Polish space and that $\ca{Y}$ is separable metrizable. For any function $f:\ca{A}\to\ca{Y}$ and any countable ordinals $\eta\leq\xi<\omega_1$, the following assertions are equivalent:
\begin{enumerate}
\item $f^{-1}\bolds^0_{1+\eta}\subseteq\bolds^0_{1+\xi}$ holds.
\item $f^{-1}\bolds^0_{1+\eta}\subseteq\bolds^0_{1+\xi}$ holds continuous-uniformly.
\item There is a $\boldd^0_{1+\xi}$-cover $(\ca{A}_i)_{i\in\omega}$ of $\ca{A}$ such that for all $i\in\omega$ the restriction $f\rs{\ca{A}_i}$ is $\bolds^0_{1+\theta_i}$-measurable for some ordinal $\theta_i$ with $\theta_i+\eta\leq\xi$.
\end{enumerate}
\end{gendec}

Note that the implication from the assertion (2) to (1) is obvious, and the implication from the assertion (3) to (2) has been observed by Kihara \cite[Lemma 2.3]{kihara_decomposing_Borel_functions_using_the_Shore_Slaman_join_theorem}.
Hence, the problem is whether the assertion (1) implies (3).
The main result of this section will be the following transfinite version of Theorem \ref{theorem main}:

\begin{theorem}
\label{theorem main transfinite}
Suppose that $\eta \leq \xi < \om_1$, \ca{X}, \ca{Y} are Polish spaces, $\ca{A}$ is an analytic subset of \ca{X}, and $f: \ca{A} \to \ca{Y}$ satisfies $f^{-1} \bolds^0_{1+\eta}\subseteq \bolds^0_{1+\xi}$. Then there exists a $\boldd^0_{1+\xi+1}$-cover $(\ca{A}_i)_{i \in \om}$ of $\ca{A}$ such that for all $i \in \om$ the restriction $f \upharpoonright \ca{A}_i$ is $\bolds^0_{1+\theta_i}$-measurable for some ordinal $\theta_i$ with $\theta_i + \eta \leq \xi$.\smallskip

If moreover $f$ is $\bolds^0_{1+\zeta}$-measurable for some $\zeta < \xi$ then the preceding sets $\ca{A}_i$ can be chosen to be $\boldd^0_{1+\xi}$.
\end{theorem}

As before we prove the preceding theorem using Borel-uniformity functions. We recall the following coding of the pointclasses $\bolds^0_\xi$. The references about the following notions and results are to \cite{louveau_a_separation_theorem_for_sigma_sets} and \cite{yiannis_dst}.

Define the sets $\BC_\xi \subseteq \baire$, $\xi < \om_1$ recursively
\begin{align*}
\alpha \in \BC_0 \iff& \alpha(0) = 0,\\
\alpha \in \BC_\xi \iff& \alpha(0)=1 \ \& \ (\forall n)(\exists \zeta < \xi)[(\alpha^\ast)_n \in \BC_\zeta],
\end{align*}
where by $\alpha^\ast$ we mean the function $n \mapsto \alpha(n+1)$, so that $\alpha=\alpha(0)\fr\alpha^\ast$.
The set of \emph{Borel codes} is $\BC = \bigcup_{\xi < \om_1} \BC_\xi$,
For $\alpha \in \BC$ we define $|\alpha|$ as the least $\xi$ such that $\alpha \in \BC_\xi$.
Given a Polish space $\ca{X}$ with a fixed open basis $(B^\xx_e)_{e\in\omega}$ we define the representations $\pi^{\ca{X}}_\xi : \BC_\xi \to \bolds^0_\xi \upharpoonright \ca{X}$ by recursion,
\begin{align*}
\pi^\ca{X}_0(\alpha) =& \ \ca{X} \setminus B^\ca{X}_{\alpha^\ast(1)}\\
\pi^\ca{X}_\xi(\alpha) =& \ \bigcup_{n} \ca{X} \setminus \pi^\ca{X}_{|(\alpha^\ast)_n|}((\alpha^\ast)_n),
\end{align*}
whereas by $\bolds^0_0 \upharpoonright \ca{X}$ we mean the family $\set{\ca{X} \setminus B^\ca{X}_s}{s \in \om}$. An easy induction shows that for all $1 \leq \zeta \leq \xi$ we have that $\BC_\zeta \subseteq \BC_\xi$ and $\pi^\ca{X}_\xi \upharpoonright \BC_\zeta = \pi^\ca{X}_\zeta$.

Suppose that $\ca{X}$ is recursively presented. Recall that the set $P \subseteq \ca{X}$ belongs to the \emph{lightface} $\Sigma^0_\xi$ if there exists a \emph{recursive} $\alpha \in \BC_{\xi}$ such that $P = \pi^\ca{X}_\xi(\alpha)$. Similarly the set $P$ belongs to $\Sigma^0_\xi(\ep)$ if there exists some $\ep$-recursive $\alpha \in \BC_\xi$ such that $P = \pi^\ca{X}_\xi(\alpha)$.
One can verify that for finite $1 \leq \xi = n < \om$ this definition coincides with the definition of the corresponding arithmetical pointclass $\Sigma^0_n$ of Kleene, so no conflict arises from this notation.\footnote{We point out another substantial difference from Kihara's results \cite{kihara_decomposing_Borel_functions_using_the_Shore_Slaman_join_theorem}. In the transfinite case the latter employes universal systems, which are defined through indexes in Kleene's $\ca{O}$ in a uniform way. This ``recursion-theoretic" coding proves the transfinite version of the decomposition with a rather straightforward modification of the finite version. However here we use the coding of $\bolds^0_\xi$ sets that is employed in Louveau-separation. The proof of the decomposition under this coding needs substantially more work, \cf Lemmata \ref{lemma good parametrization for transfinite} and \ref{borel-code-computability} below.}

\begin{definition}\normalfont
\label{definition of bcode sets}
Given a Polish space \ca{X} we define the sets $\bcode^\ca{X}_\xi \subseteq \baire \times \ca{X}$ by
\[
\bcode^\ca{X}_\xi(\alpha,x) \iff \alpha \in \BC_\xi \ \& \ x \in \pi^\ca{X}_\xi(\alpha).
\]
\end{definition}
It is easy to verify that $\bcode^\ca{X}_\xi$ parametrizes $\bolds^0_\xi \upharpoonright \ca{X}$. In the transfinite case we will need to replace the system $(\univ^{\ca{X}}_n)_{\ca{X}}$ with $(\bcode^{\ca{X}}_\xi)_{\ca{X}}$.

\begin{theorem}
\label{theorem Borel transition in the codes transfinite}
Suppose that $\ca{X}$, \ca{Y} are Polish spaces, $\ca{A} \subseteq \ca{X}$ is analytic and that $f:\ca{A} \to \ca{Y}$ satisfies $f^{-1} \bolds^0_\eta \subseteq \bolds^0_\xi$ for some $\eta, \xi \geq 1$. Then condition $f^{-1} \bolds^0_\eta \subseteq \bolds^0_\xi$ holds Borel-uniformly in the codes with respect to $\bcode^{\ca{Y}}_\eta, \bcode^{\ca{X}}_\xi$.
\end{theorem}

\begin{proof}
We repeat the proof of Theorem \ref{theorem Borel transition in the codes} and we notice that the corresponding sets $P, Q$ defined by
\begin{align*}
P(\alpha,x) \iff& x \in \ca{A} \ \& \ \bcode_\eta^{\ca{Y}}(\alpha,f(x))\\
Q(\alpha,x) \iff& x \in \ca{A} \ \& \ \neg \bcode_\eta^{\ca{Y}}(\alpha,f(x))
\end{align*}
are analytic sets - and hence $\sig(\ep)$ for some $\ep \in \baire$.
To see the this, we remark that the sets $\BC_\zeta$ are Borel, and that for all Polish spaces $\ca{Z}$ and all $\zeta < \om_1 $, the set of all $(\alpha,z)\in\BC_\zeta\times\ca{Z}$ such that $z\in\pi^\ca{Z}_\zeta(\alpha)$ is Borel in $\baire\times\ca{Z}$.
Moreover we notice that the sections $P_{\alpha,i}$ are $\bolds^0_\eta$. The rest of the proof remains the same.
\end{proof}

In order to extend the decomposition in the transfinite case we need to show that this parametrization $\bcode^\ca{X}_\xi$ is good as well.

\begin{lemma}[The Good Parametrization Lemma for $\bcode^{\ca{X}}_\xi$]
\label{lemma good parametrization for transfinite}
The parametrization system $(\bcode^{\ca{Y}}_\xi)_{\ca{Y}}$ for $\bolds^0_\xi$ is good.
Moreover, the realizer $S^{\ca{X},\ca{Y}}$ can be chosen as a recursive function.
\end{lemma}

\begin{proof}
Let $S_1$ be a recursive function which realizes the map $\bolds^0_1\rs(\baire\times\ca{Y})\times\baire$; $(A,\alpha)\mapsto A_\alpha$.
In other words, $S_1(\ep,\alpha)$ is a code of the $\alpha^{\text{th}}$ section of the $\ep^{\text{th}}$ open set in $\baire\times\ca{Y}$ (such a recursive function $S_1$ exists since the universal system $(\univ_1^\ca{X})_\ca{X}$ is good).
Then, we define $g:\baire \times \baire \to \baire$ by $g(\ep,\alpha) = S_1(\ep^\ast,\alpha)$, where as above $\ep^\ast$ is the result of dropping the first bit from $\ep$.
Clearly $g$ is recursive.
We construct a recursive function $f: \baire \times \baire \to \baire$ which satisfies the following conditions:
\begin{align}
\label{equation A lemma for sigma-xi good transition} f(\ep,\alpha)(0) &= \ep(0)\\
\label{equation B lemma for sigma-xi good transition} \ep(0) = 0 \Longrightarrow& \ f(\ep,\alpha)^\ast = g(\ep,\alpha)\\
\label{equation C lemma for sigma-xi good transition} \ep(0) = 1 \Longrightarrow& \ (f(\ep,\alpha)^\ast)_n = f((\ep^\ast)_n,\alpha)
\end{align}
for all $n, \ep, \alpha$.
It follows that $\ep \in \BC_{\xi}$ if and only if $f(\ep,\alpha) \in \BC_{\xi}$ for all $\ep, \alpha$ and all $\xi < \om_1$.
We will define by induction on $t$ the value $f(\ep,\alpha)(t)$ for all $\ep, \alpha$ using Kleene's Second Recursion Theorem.
We define the partial function $\varphi: \om \times \om \times \baire \times \baire \rightharpoonup \om$ as follows
\[
\varphi(e,k,\ep,\alpha) =
\begin{cases}
\ep(0), & \ \textrm{if} \ k=0\\
g(\ep,\alpha)(m), & \ \textrm{if} \ k=m+1 \ \textrm{and} \ \ep(0) = 0\\
\rfn{e}(t,(\ep^\ast)_n,\alpha), & \ \textrm{if} \ k = 1 + \langle n,t \rangle \ \textrm{for some $n,t < k$ and} \ \ep(0) = 1\\
0, & \ \textrm{else}.
\end{cases}
\]
Here, when $(k,\ep)$ satisfies the third condition, the value $\varphi(e,k,\ep,\alpha)$ is defined if and only if $\rfn{e}(t,(\ep^\ast)_n,\alpha)$ is defined.
It is clear that $\varphi$ is partial recursive on its domain. From Kleene's Second Recursion Theorem (\cf \cite[7A.2 and 7A.4]{yiannis_dst}) there exists some $e^\ast \in \om$ such that whenever $\varphi(e^\ast,k,\ep,\alpha) \downarrow$ then $\rfn{e^\ast}(k,\ep,\alpha) \downarrow$ and
\[
\varphi(e^\ast,k,\ep,\alpha) = \rfn{e^\ast}(k,\ep,\alpha).
\]
One can verify that the function
\[
f: \baire \times \baire \to \baire: f(\ep,\alpha)(k) = \varphi(e^\ast,k,\ep,\alpha) = \rfn{e^\ast}(k,\ep,\alpha).
\]
satisfies all required properties. We omit the details.
Now, we claim that
\[
\bcode^{\baire \times \ca{Y}}_\xi(\ep,\alpha,y) \iff \bcode^{\ca{Y}}_\xi(f(\ep,\alpha),y)
\]
for all $\ep,\alpha, y$ and all $\xi < \om_1$, \ie we can take $\goodt^{\baire,\ca{Y}}$ to be the preceding function $f$.

This is proved by induction on $\xi$. The equivalence
\[
\bcode^{\baire \times \ca{Y}}_0(\ep,\alpha,y) \iff \bcode^{\ca{Y}}_0(f(\ep,\alpha),y)
\]
is proved using the properties (\ref{equation A lemma for sigma-xi good transition}) and (\ref{equation B lemma for sigma-xi good transition}) from above.
To complete the inductive step, for a code $\ep$ of a $\bolds^0_\xi$ set (i.e., $\ep\in\BC_\xi$), let $B_\ep$ be the $\bolds^0_\xi$ subset of $\baire\times\ca{Y}$ coded by $\ep$ (that is, $B_\ep$ is the $\ep^{\text{th}}$ section of $B_\xi^{\baire\times\ca{Y}}$).
By the definition of a Borel code, we have $B_\ep=\bigcup_{n}B^\complement_{(\ep^*)_n}$.
In particular, $B_{\ep,\alpha}=\bigcup_{n}B_{(\ep^*)_{n},\alpha}^\complement$.
By the induction hypothesis, $B_{(\ep^*)_{n},\alpha}$ is coded by $f((\ep^\ast)_n,\alpha)$.
By (\ref{equation C lemma for sigma-xi good transition}), $f((\ep^\ast)_n,\alpha)=(f(\ep,\alpha)^\ast)_n$.
Therefore, $f(\ep,\alpha)$ is a code of $B_{\ep,\alpha}$ by (\ref{equation A lemma for sigma-xi good transition}).
By the property of $f$, if $\ep\not\in\BC_\xi$, then $f(\ep,\alpha)\not\in\BC_\xi$ as well.
This shows the desired equivalence.
\end{proof}

With the help of this lemma one can extend a natural property of $\Sigma^0_\xi$ in all recursively presented metric spaces.

\begin{corollary}
\label{corollary sections of sigma xi}
Suppose that \ca{X} is a recursively presented metric space and that $P \subseteq \baire \times \ca{X}$ is $\Sigma^0_\xi(\ep)$. Then for all $\alpha$ the section $P_\alpha$ is $\Sigma^0_\xi(\ep,\alpha)$.
\end{corollary}

Before proving the transfinite version of Lemma \ref{lem:turing-degree-analysis}, we need to show that every $\Sigma^0_{1+\xi}$ subset of $\om$ is recursive in $\emptyset^{(\xi+1)}$. This might seem obvious but there is a subtlety here, which needs some explanation. First we recall that the transfinite Turing jump is usually defined via the $H$-sets along Kleene's system $\mathcal{O}$ of ordinal notations.
As it is well-known, for any ordinal $\xi < \ck$, if $a\in\mathcal{O}$ codes the ordinal $1+\xi$, then every $\Sigma^0_{a}$ subset of $\omega$ is clearly recursive in $\emptyset^{(\xi+1)}$ (where $\Sigma^0_a$ is defined in a straightforward manner, that is, in the sense of Mostowski-Kleene).
The problem here is that we do not know if every $\Sigma^0_{1+\xi}$ set in the above sense is $\Sigma^0_a$ for some $a \in \mathcal{O}$ with $|a|\leq 1+\xi$. We can nevertheless prove the following.

\begin{lemma}
\label{borel-code-computability}
Let $z\in\baire$ be a real, and $\xi$ be a $z$-computable ordinal.
Then, every $\Sigma^0_{1+\xi}(z)$ subset of $\omega$ is recursive in $z^{(\xi+1)}$.
\end{lemma}

\begin{proof}
It is clear if $\xi$ is finite.
If $\xi$ is infinite, then $1+\xi=\xi$.
We first extract a well-founded tree $T_\nu$ (and some auxiliary function $t_\nu:T_\nu\to\BC_\xi$) from a given code $\nu\in\BC_\xi$ as follows:
Put the empty string $\langle\rangle$ into $T_\nu$, and define $t_\nu(\langle\rangle)=\nu$.
Inductively, if we have already declared that $\sigma\in T_\nu$ with $\nu_\sigma:=t_\nu(\sigma)$, then in case that $\nu_\sigma(0)=1$ we declare that $\sigma\fr n\in T_\nu$ and define $t_\nu(\sigma\fr n):=(\nu_\sigma)^*_n$ for all $n\in\omega$.
Otherwise, we declare that $\sigma$ is a leaf of $T_\nu$, that is, $\sigma\fr n\not\in T_\nu$ for all $n\in\omega$.

We introduce a modified system $\tilde{\mathcal{O}}^\nu$ of ordinal notations with a function $||\cdot||_\nu:\tilde{\mathcal{O}}^\nu\to\omega_1^{{\rm CK},\nu}$ as follows:
\begin{enumerate}
\item Put $\mathbf{o}:=1\in\tilde{\mathcal{O}}^\nu$, and define $||\mathbf{o}||_\nu=0$.
\item If $\Phi_e^\nu(n)\in\tilde{\mathcal{O}}^\nu$ for all $n\in\omega$, then put $\mathbf{s}(e):=2^e\in\tilde{\mathcal{O}}^\nu$, and define $||\mathbf{s}(e)||_\nu=\sup_n(||\Phi_e^\nu(n)||_\nu+1)$.
\end{enumerate}

Then, the associated $\tilde{H}$-sets are defined as follows:
\[\tilde{H}^{\nu}_\mathbf{o}=\nu\mbox{, and }\tilde{H}^{\nu}_{\mathbf{s}(e)}=\bigoplus_{n\in\omega}(\tilde{H}^\nu_{\Phi_e^\nu(n)})'.\]

We inductively define a $\nu$-computable function $o_\nu:T_\nu\to\tilde{\mathcal{O}}^\nu$ such that $||o_\nu(\sigma)||_{\nu}=|t_\nu(\sigma)|$ as follows:
If $\sigma$ is a leaf of $T_\nu$, then define $o_\nu(\sigma)=\mathbf{o}$.
If $o_\nu(\sigma\fr n)$ has already been defined for all $n\in\om$, then we effectively obtain an index $e$ such that $\Phi_e^\nu(n)=o_\nu(\sigma\fr n)$, and then define $o_\nu(\sigma)=\mathbf{s}(e)$.
Finally, we define $a:=o_\nu(\langle\rangle)\in\tilde{\mathcal{O}}^\nu$.
Clearly, $||a||_{\nu}=|\nu|\leq \xi$.

It is straightforward to see that $\bcode^\om_{\xi,\nu}$ is $\Sigma^0_1$ in $\tilde{H}^{\nu}_a$; hence, Turing reducible to $(\tilde{H}^{\nu}_a)'$.
Thus, it suffices to show that the modified $H$-set $\tilde{H}^{\nu}_a$ is Turing reducible to the standard $H$-set ${H}^{\nu}_b$ for some (any) $b\in\mathcal{O}^\nu$ with $||a||_\nu=|b|_\nu$, where $|b|_\nu$ is the ordinal coded by $b$.
The proof is almost the same as that of the Spector uniqueness theorem (see Sacks \cite[Theorem II.4.5]{SacksBook}), so we only make a rough sketch.
We define a well order $<$ on $\tilde{\mathcal{O}}^\nu\times\mathcal{O}^\nu$ as $(a_1,b_1)<(a_2,b_2)$ if and only if $||b_1||_\nu<|b_2|_\nu$ or $||b_1||_\nu=|b_2|_\nu$ and $||a_1||_\nu<|a_2|_\nu$.
By induction on $<$, we show that there are $\nu$-recursive functions $h_0$ and $h_1$ such that
\begin{enumerate}
\item if $||a||_\nu<|b|_\nu$ then $(\tilde{H}^\nu_a)'\leq_TH^\nu_b$ via index $h_0(a,b)$,
\item if $||a||_\nu\leq|b|_\nu$ then $\tilde{H}^\nu_a\leq_TH^\nu_b$ via index $h_1(a,b)$.
\end{enumerate}

We first show the item (1), and assume that $||a||_\nu<|b|_\nu$.
If $b$ is of the form $2^c$ and $a=\mathbf{s}(e)$, then $||\Phi^\nu_e(n)||_\nu<|c|_\nu$ for all $n\in\omega$.
By induction hypothesis, for $a(n):=\Phi_e^\nu(n)$, we have $(\tilde{H}^\nu_{a(n)})'\leq_TH^\nu_c$ via index $h_0(a(n),c)$ for all $n\in\omega$.
Then, we can effectively find an index computing $(\bigoplus_n(\tilde{H}^\nu_{a(n)})')'$ via the oracle $H^\nu_b$.
If $b$ is of the form $3\cdot 5^e$, then we have $||a||_\nu+n<|b|_\nu$ for all $n\in\omega$ since $|b|_\nu$ is limit.
So, we compute $m\in\omega$ such that $||a||_\nu<|\Phi_e^\nu(m)|_\nu$ (we can do this as in \cite[Theorem II.4.5]{SacksBook}).
Then, by induction, use the index $h_0(a,\Phi_e^\nu(m))$ to find a desired index.
To show the item (2), assume that $||a||_\nu\leq|b|_\nu$ and $a=\mathbf{s}(e)$.
Then, since $||\Phi_e^\nu(n)||_\nu<|b|_\nu$, for $a(n):=\Phi_e^\nu(n)$, we have $(\tilde{H}^\nu_{a(n)})'\leq_TH^\nu_b$ via index $h_0(a(n),b)$ for all $n\in\omega$.
Therefore, we can effectively find an index computing $\bigoplus_n(\tilde{H}^\nu_{a(n)})'$ via the oracle $H^\nu_b$.
Formally, we need to use the Kleene recursion theorem; see Sacks \cite[Theorem II.4.5]{SacksBook} for the detail.
\end{proof}

Now we are ready to prove the following transfinite analogue of Lemma \ref{lem:turing-degree-analysis}.

\begin{lemma}\label{lem:turing-degree-analysis-transfinite}
Let $\eta$ and $\xi$ be countable ordinals.
Suppose that a function $f:\mathcal{A}\to\mathcal{Y}$ from an analytic subset $\mathcal{A}$ of a Polish space \ca{X} to a Polish space $\mathcal{Y}$ satisfies that the condition $f^{-1} \bolds^0_{1+\eta} \subseteq \bolds^0_{1+\xi}$ holds Borel-uniformly in the codes with respect to $\bcode^{\om \times \ca{Y}}_{1+\eta}, \bcode^{\om \times \ca{X}}_{1+\xi}$.
Then, there are $z\in 2^\omega$ and a $z$-computable ordinal $\theta$ such that we have $(f(x)\oplus q)^{(\eta+1)}\leq_M(x\oplus q^{(\theta)})^{(\xi+1)}$ for all $q\geq_Tz$.
\end{lemma}

\begin{proof}
Assume that $\eta$ and $\xi$ are infinite, and fix $\ep\in\baire$ such that $\eta$ and $\xi$ are $\ep$-computable.
Without loss of generality, we can assume that $\ca{A}$ is a subset of $\baire$, and $\ca{Z}=\baire$, since any uncountable Polish space is $\boldd^0_\omega$-isomorphic to $\baire$ and since $\eta, \xi \geq \om$.
The first idea is to add $1$ in the exponent of the Turing-jump power when the latter is infinite.
To be more specific, fix a notation $a\in\mathcal{O}^\ep$ for the ordinal $\eta+1$.
Then, the $a^{\text{th}}$ jump $y^{(a)}$ (formally, this is defined by using an $H$-set) is $\Sigma^0_{1+\eta}(\ep)$ uniformly in $y$; for instance, $\emptyset^{(\omega)}$ is $\Delta^0_\omega$, and therefore, $\emptyset^{(\omega+1)}$ is $\Sigma^0_\omega$.
Moreover, there is a code $\beta\in\BC_{1+\eta}$ such that
\[\bcode^{\baire\times\omega\times\baire}_{1+\eta,\beta}=\{(p,i,y)\in\baire\times\omega\times\baire:i\in (y\oplus p\oplus\ep)^{(a)}\}.\]

Using the Good Parametrization Lemma \ref{lemma good parametrization for transfinite}, given $p\in\baire$, one can $\ep$-computably find a code $\alpha(p)\in\BC_{1+\eta}$ (that is, $\alpha$ is $\ep$-computable) such that $\bcode^{\omega\times\baire}_{1+\eta,\alpha(p)}=\bcode^{\baire\times\omega\times\baire}_{1+\eta,\beta,p}$

Let $f:\mathcal{A}\to\baire$ be a function satisfying that $f^{-1}\bolds^0_{1+\eta}\subseteq\bolds^0_{1+\xi}$. It is easy to verify that the function $(i,x) \mapsto (i,f(x))$ satisfies the latter condition as well. From Lemma \ref{theorem Borel transition in the codes}, we have a Borel measurable transition $u$ such that for all $x\in\mathcal{A}$, $i\in\omega$ and $p\in\baire$,
\[\bcode^{\omega\times\baire}_{1+\eta}(\alpha(p),i,f(x))\;\iff\;\bcode^{\omega\times\baire}_{1+\xi}(u(\alpha(p)),i,x).\]

Since $u$ is Borel, there is $z\geq_T\ep$ such that $u$ is $\Sigma^0_b(z)$-recursive for some $b\in\mathcal{O}^z$.
Then $u(\alpha(p))\leq_T(p\oplus z)^{(\theta)}$ for some $z$-computable ordinal $\theta$ since $\alpha(p)\leq_T p\oplus\ep$.
By the Good Parametrization Lemma \ref{lemma good parametrization for transfinite}, the $x^{\text{th}}$ section $B^{\omega\times\baire}_{1+\xi,u(\alpha(p)),x}$ is $\Sigma^0_{1+\xi}(x\oplus (p\oplus z)^{(\theta)})$.
Consequently, by Lemma \ref{borel-code-computability}, for any $q\geq_Tz$ we have:
\[(f(x)\oplus q)^{(a)}\leq_T(x\oplus q^{(\theta)})^{(\xi+1)}.\]

Then, we have the desired inequality for Turing degrees by the Spector uniqueness theorem.
If $\eta$ is finite, combine this argument with the argument in Section \ref{section Borel transition in the codes}.
\end{proof}

The proof of the following transfinite analogue of the Cancellation Lemma \ref{lem:deg-Shore-Slaman} is straightforward.

\begin{lemma}[The Transfinite Cancellation Lemma]\label{lem:deg-Shore-Slaman-transfinite}
Suppose that $\xx$ and $\yy$ are computable metric spaces, and fix $x\in\xx$, $y\in\yy$, $z\in 2^\omega$, and $\gamma,\eta,\xi<\ckr{z}$.
If $(y\oplus g)^{(\eta)}\leq_T(x\oplus g^{(\gamma)})^{(\xi)}$ for all $g\geq_Tz$, then there exists $\theta$ with $\theta+\eta\leq\xi$ such that $y\leq_T(x\oplus z^{(\gamma)})^{(\theta)}$.\qed
\end{lemma}

Now we give the proof of our main result in this section.

\begin{proof}[Proof of Theorem \ref{theorem main transfinite}]
Suppose that $f:\mathcal{A}\to\mathcal{Y}$ satisfies $f^{-1} \bolds^0_{1+\eta} \subseteq \bolds^0_{1+\xi}$. Without loss of generality we may assume that the given spaces are recursively presented and that $\ca{A}$ is $\Sigma^1_1$. As in the proof of Theorem \ref{theorem main}, using Theorem \ref{theorem Borel transition in the codes transfinite} and Lemma \ref{lem:turing-degree-analysis-transfinite}, we have there is some $z \in \cantor$ and a $z$-computable ordinal $\gamma$ such that for all for all $q\geq_Tz$ and all $x \in \ca{A}$ it holds $(f(x)\oplus q)^{(\eta+1)}\leq_M(x\oplus q^{(\gamma)})^{(\xi+1)}$.
By the Transfinite Cancellation Lemma \ref{lem:deg-Shore-Slaman-transfinite}, we have for all $x \in \ca{A}$ that $f(x)\leq_T(x\oplus z^{(\gamma)})^{(\theta_x)}$ for some $\theta_x$ with $\theta_x+\eta\leq\xi$.

For all $\theta$ with $\theta + \eta \leq \xi$ we define the function $g_\theta:\ca{X}\to\baire$ as $g_\theta(x)=(x\oplus z^{(\gamma)})^{(\theta)}$.
Moreover for all $e \in \om$ and $\theta$ with $\theta + \eta \leq \xi$ we put
\[
\ca{B}_{e,\theta} = \set{x \in \ca{A}}{f(x)=\Phi_e(g_\theta(x))}.
\]
It follows that $\ca{A} = \bigcup_{e,\theta} \ca{B}_{e,\theta}$. Since each function $g_{e,\theta} : = \Phi_{e} \circ g_\theta$ is $\bolds^0_{1+\theta}$-measurable ($\theta < \xi$), it is easy to see that the sets $\ca{B}_{e,\theta}$ are $\boldp^0_{1+\xi}$ subsets of $\ca{A}$, and that each restriction $f \upharpoonright \ca{B}_{e,\theta}$ is $\bolds^0_{1+\theta }$-measurable.

If now the function $f$ is $\bolds^0_{1+\zeta}$-measurable for some countable $\zeta$ then $\ca{B}_{e,\theta}$ is a $\boldp^0_{\max\{1+\theta+1,1+\zeta\}}$ subset of \ca{A}, and so we have $1+\theta + 1 < 1+\theta + \eta \leq 1+\xi$. It follows that $\ca{B}_{e,\theta}$ is a $\boldd^0_{1+\xi}$ subset of \ca{A} when $2 \leq \eta$ and $\zeta < \xi$.
\end{proof}

\section{Computable and Continuous Cases}

\label{section computable and continuous cases}

The problem on the decomposability of Borel functions is also of great interest from the computability-theoretic perspective (see \cite{kihara_decomposing_Borel_functions_using_the_Shore_Slaman_join_theorem,pauly_debrecht_nondeterministic_computation_and_the_jayne_rogers_theorem}). It is therefore worth extracting the effective content from our proof of Theorem \ref{theorem main transfinite}.
Hereafter, for a function $f:\ca{X}\to\ca{Y}$ we write $f\in\dec((\Gamma_\lambda)_{\lambda\in\Lambda},\Delta)$ if there is a $\Delta$ cover $(\ca{X}_i)_{i\in\omega}$ of $\ca{X}$ such that for all $i\in\omega$ the restriction $f\rs{\ca{X}_i}$ is $\Gamma_\lambda$-measurable for some $\lambda\in\Lambda$.
In this terminology, Theorem \ref{theorem main transfinite} can be restated as follows:
\[f^{-1} \bolds^0_{1+\eta}\subseteq \bolds^0_{1+\xi}\;\Longrightarrow\;f\in\dec((\bolds^{0}_{1+\theta})_{\theta+\eta\leq\xi},\boldd^{0}_{1+\xi+1}).\]

Note that there is no a priori effectivization of condition $f^{-1} \bolds^0_{1+\eta}\subseteq \bolds^0_{1+\xi}$.
However, condition $f^{-1} \bolds^0_{1+\eta}\subseteq \bolds^0_{1+\xi}$ always implies that $f$ is $\bolds^0_{1+\xi}$-measurable (that is, condition $f^{-1} \bolds^0_{1}\subseteq \bolds^0_{1+\xi}$ holds continuous-uniformly), and the latter notion has some effective counterpart.
For a pointclass $\tboldsymbol{\Gamma}$, we define $\Gamma^\hyp$ as the collection of all hyperarithmetically-coded (i.e., $\Delta^1_1$-coded) sets in $\tboldsymbol{\Gamma}$.
A function $f$ is said to be $\Sigma^{0,\hyp}_{1+\xi}$-measurable if $f^{-1} \bolds^0_{1}\subseteq \bolds^0_{1+\xi}$ holds $\Delta^1_1$-uniformly (or equivalently, it holds $\Sigma^{0,\hyp}_1$-uniformly).
It is not hard to check that our proof of Theorem \ref{theorem main transfinite} indeed implies the following effective version:

\begin{theorem}
Suppose that $\eta \leq \xi < \ck$, $\ca{A}$ is a $\Sigma^1_1$ subset of a computable Polish space, $\ca{Z}$ is a computable metric space, and $f: \ca{A} \to \ca{Z}$ is a $\Sigma^{0,\hyp}_{1+\xi}$-measurable function.
Then, we have the following implication:
\[f^{-1} \bolds^0_{1+\eta}\subseteq \bolds^0_{1+\xi}\;\Longrightarrow\;f\in\dec((\Sigma^{0,\hyp}_{1+\theta})_{\theta+\eta\leq\xi},\Delta^{0,\hyp}_{1+\xi+1}).\]

If moreover $f$ is $\Sigma^{0,\hyp}_{1+\zeta}$-measurable for some $\zeta < \xi$ then:
\[f^{-1} \bolds^0_{1+\eta}\subseteq \bolds^0_{1+\xi}\;\iff\;f\in\dec((\Sigma^{0,\hyp}_{1+\theta})_{\theta+\eta\leq\xi},\Delta^{0,\hyp}_{1+\xi}).\]
\end{theorem}

We next give some generalizations of Kihara's Theorem \ref{theorem Kihara continuous transition}.
For a function $f:\ca{X}\to\ca{Y}$ we write $f\in\dec^\star((\Gamma_\lambda)_{\lambda\in\Lambda},\Delta)$ if there is a uniformly-$\Delta$ cover $(\ca{X}_i)_{i\in\omega}$ of $\ca{X}$ such that the restriction $f\rs{\ca{X}_i}$ is $\Gamma_\lambda$-measurable for some $\lambda\in\Lambda$ uniformly in $i\in\omega$ (see Kihara \cite[Definition 1.2]{kihara_decomposing_Borel_functions_using_the_Shore_Slaman_join_theorem}).
By combining the argument in Kihara \cite{kihara_decomposing_Borel_functions_using_the_Shore_Slaman_join_theorem} with the Shore-Slaman Join Theorem \ref{thm:shoreslamanpolish} on Polish spaces, one can show the following:

\begin{theorem}
\label{thm:computable}
Suppose that $\eta \leq \xi < \om_1^{\rm CK}$, $\ca{X}$ is a computable Polish space and $\ca{Z}$ is a computable metric space.
For any function $f: \ca{X} \to \ca{Z}$, we have the implications (1) $\Rightarrow$ (2) $\Rightarrow$ (3):
\begin{enumerate}
\item $f \in \dec^\star((\Sigma^0_{1+\theta})_{\theta+\eta\leq\xi},\Delta^0_{1+\xi})$.
\item $f^{-1} \bolds^0_{1+\eta} \subseteq \bolds^0_{1+\xi}$ holds recursive-uniformly.
\item $f \in \dec^\star((\Sigma^0_{1+\theta})_{\theta+\eta\leq\xi},\Delta^0_{1+\xi+1})$.
\end{enumerate}

Moreover, if $\xi<\eta\cdot 2$, then we have the equivalence (1) $\Leftrightarrow$ (2).
\qed
\end{theorem}

By relativizing Theorem \ref{thm:computable}, we have the following continuous-uniform version:

\begin{theorem}
Suppose that $\eta \leq \xi < \om_1$, $\ca{X}$ is a Polish space and $\ca{Z}$ is a separable metrizable space.
For any function $f: \ca{X} \to \ca{Z}$, we have the implications (1) $\Rightarrow$ (2) $\Rightarrow$ (3):
\begin{enumerate}
\item $f \in \dec((\bolds^0_{1+\theta})_{\theta+\eta\leq\xi},\boldd^0_{1+\xi})$.
\item $f^{-1} \bolds^0_{1+\eta} \subseteq \bolds^0_{1+\xi}$ holds continuous-uniformly.
\item $f \in \dec((\bolds^0_{1+\theta})_{\theta+\eta\leq\xi},\boldd^0_{1+\xi+1})$.
\end{enumerate}

Moreover, if $\xi<\eta\cdot 2$, then we have the equivalence (1) $\Leftrightarrow$ (2).
\qed
\end{theorem}

\section{The Borel Martin Conjecture}

The {\em Martin Conjecture} (for Turing invariant functions), also known as the {\em fifth Victoria-Delfino Problem} (see Kechris-Moschovakis \cite{KecMos}), is probably one of the most famous open problems in Turing degree theory.
We have an impression that there is distant resemblance between the Decomposability Conjecture and the Martin Conjecture.
For instance, the Shore-Slaman Join Theorem has been applied to give a partial answer to the Martin conjecture for $(\leq_T,\leq_T)$-preserving (possibly ideal-valued) Borel functions $2^\omega\to 2^\omega$.
In Section \ref{section generalized Turing degree theory}, we have developed the generalized degree theory, and obtained the Shore-Slaman Join Theorem in Polish spaces.
This enables us to generalize the Martin Conjecture to functions on Polish spaces, and also provides a powerful tool to analyze the generalized Martin conjecture.
We first state an abstract version of the Martin conjecture.

Let $(P,\leq_P)$ and $(Q,\leq_Q)$ be preordered sets.
We say that a function $f:P\to Q$ is {\em $(\equiv_P,\equiv_Q)$-invariant} if $p\equiv_Pq$ implies $f(p)\equiv_Qf(q)$, and that $f$ is {\em $(\leq_P,\leq_Q)$-preserving} if $p\leq_Pq$ implies $f(p)\leq_Qf(q)$.
For functions $f,g:P\to Q$, we say that {\em $f$ is Martin-below $g$} (written as $f\leq_\mathbf{m}g$) if there is $c\in P$ such that for every $p\geq_Pc$, $f(p)\leq_Qg(p)$ holds.
If $(P,\leq_P)$ is directed, then the relation $\leq_\mathbf{m}$ again forms a preorder.


Now we assume that $(P,\leq_P)$ is a suborder of $(Q,\leq_Q)$.
We say that a function $f:P\to Q$ is {\em increasing on a cone} if there is $c\in P$ such that for every $p\geq_Pc$, $p\leq_Q f(p)$ holds, and that $f$ is {\em constant on a cone} if there are $c\in P$ and $d\in Q$ such that for every $p\geq_Pc$, $f(p)\equiv_Qd$ holds.
If a $(\leq_Q,\leq_Q)$-preserving function $\mathcal{J}:Q\to Q$ is given, then for a function $f:P\to Q$, we define $\mathcal{J}(f):P\to Q$ by $\mathcal{J}(f)(p)=\mathcal{J}(f(p))$ for every $p\in P$.

\begin{conjecture}[The Martin Conjecture for $(\equiv_P,\equiv_Q,\mathcal{J})$]
Assume ${\sf ZF}+{\sf AD}+{\sf DC}$.
Then
\begin{enumerate}
\item[I.] $f:P\to Q$ is $(\equiv_P,\equiv_Q)$-invariant and not increasing on a cone, then $f$ is constant on a cone.
\item[II.] $\leq_\mathbf{m}$ prewellorders the set of $(\equiv_P,\equiv_Q)$-invariant functions which are increasing on a cone.
Moreover, if $f$ has $\leq_\mathbf{m}$-rank $\alpha$, then $\mathcal{J}(f)$ has $\leq_\mathbf{m}$-rank $\alpha+1$.
\end{enumerate}
\end{conjecture}

The (original) Martin Conjecture is, of course, the Martin Conjecture for $(\equiv_T,\equiv_T,TJ)$, where recall that $TJ$ denotes the Turing jump.
In this section, we only consider the Martin conjecture for $(\leq_P,\leq_Q)$-preserving functions rather than $(\equiv_P,\equiv_Q)$-invariant functions.
Then we call the modified conjecture the {\em Martin Conjecture for $(\leq_P,\leq_Q,\mathcal{J})$}.
Clearly, this is weaker than the original conjecture.

We further generalize the Martin Conjecture by considering the {\em ideal completion} $Q^\ast$ of $Q$.
The Martin Conjecture for the ideal completion of the Turing reducibility $(2^\omega,\leq_T)$ is first studied by Slaman \cite{Slaman05}.
Here we define $Q^\ast$ as the set of all ideals $I$ in $Q$, that is, $p\leq_Qq\in I$ implies $q\in I$, and $p,q\in I$ implies the existence of $r\in I$ such that $p,q\leq_Qr$.
The set  $Q^\ast$ is ordered by subset inclusion, that is, $I_0\leq_{Q}^\ast I_1$ if and only if $I_0\subseteq I_1$.
As usual, we have the canonical embedding $p\mapsto p^\ast$ from $Q$ into $Q^\ast$ by defining $p^\ast$ as the principal ideal generated by $p$, that is, $p^\ast=\{q\in Q:q\leq_Qp\}$.
So, for instance, $g:P\to Q^\ast$ is {\em increasing on a cone} if and only if $p^\ast\leq_Q^\ast g(p)$, which is also equivalent to that $p\in g(p)$.
Given $f:P\to Q$, we also consider $f^\ast:P\to Q^\ast$ defined by $f^\ast(p)=f(p)^\ast$ for every $p\in P$.
For $\mathcal{J}:Q\to Q$, the function $\mathcal{J}^\ast:Q\to Q^\ast$ can also be viewed as the partial function $\mathcal{J}^\ast:Q^\ast\partialf Q^\ast$ whose domain is the set of all principal ideals in $Q$.

\begin{proposition}
If the Martin Conjecture for $(\leq_P,\leq_{Q}^\ast,{^\ast}\!\mathcal{J})$ is true for some total extension ${^\ast}\!\mathcal{J}$ of $\mathcal{J}^\ast$, then so is the Martin Conjecture for $(\leq_P,\leq_{Q},\mathcal{J})$.
\end{proposition}

\begin{proof}
It is easy to check that the canonical embedding preserves all properties mentioned above.
For instance, if $f$ is $(\leq_P,\leq_Q)$-preserving, then $f^\ast$ is $(\leq_P,\leq_{Q}^\ast)$-preserving, and moreover, for any functions $f,g:P\to Q$, if $f\leq_\mathbf{m}g$ holds, then $f^\ast\leq_\mathbf{m}g^\ast$ holds as well.
This concludes the proof.
\end{proof}

Note that a function $f:P\to Q^\ast$ is a {\em closure operator} in the sense of Slaman \cite{Slaman05} if and only if $f:P\to Q^\ast$ is $(\leq_P,\leq_{Q}^\ast)$-preserving and increasing on a cone (up to Martin equivalence).
Slaman-Steel \cite{SlaSte88} and Slaman \cite{Slaman05} gave partial results on the {\em Borel Martin Conjecture}, that is, the Martin conjecture for $(\leq_P,\leq_Q)$-preserving Borel functions.
In this section, we only focus on the Borel Martin Conjecture, and thus, we now restrict our attention to functions between Borel spaces.

We assume that $\mathcal{X}$ and $\mathcal{Y}$ are Borel spaces endowed with preorders $\leq_\xx$ and $\leq_\yy$, respectively, and moreover, assume that $(\mathcal{X},\leq_\xx)$ is a substructure of $(\mathcal{Y},\leq_\yy)$.
We consider the ideal completion $(\mathcal{Y}^\ast,\leq_\yy^\ast)$ of the preordered Borel space $(\mathcal{Y},\leq_\yy)$.
We say that a function $f:\xx\to\yy^\ast$ is {\em Borel} if $\{(x,y)\in\xx\times\yy:y\in f(x)\}$ is Borel in the product Borel space $\xx\times\yy$.

Slaman-Steel \cite{SlaSte88} showed that every $(\leq_T,\leq_T)$-preserving Borel function $f:2^\omega\to 2^\omega$ which is increasing on a cone is the $\alpha^{\text{th}}$ Turing jump for some $\alpha<\omega_1$ up to Martin equivalence.
Slaman \cite{Slaman05} generalized this result by showing that for every $(\leq_T,\leq_T^\ast)$-preserving Borel function $f:2^\omega\to(2^\omega)^\ast$ which is increasing on a cone, one of the following holds:
\begin{enumerate}
\item There is a countable ordinal $\alpha$ such that $f$ is Martin equivalent to the following map:
\[x\mapsto\{y\in 2^\omega:y\mbox{ is recursive in }x^{(\beta)}\mbox{ for some $\beta<\alpha$}\}.\]
\item $f$ is Martin equivalent to $x\mapsto 2^\omega$.
\end{enumerate}

This result solves the Borel Martin Conjecture II affirmatively for $(\leq_T,\leq_T^\ast,{}^\ast TJ)$; hence, for $(\leq_T,\leq_T,TJ)$.
Slaman \cite{Slaman05} utilized the Shore-Slaman Join Theorem \cite{shore_slaman_defining_the_Turing_Jump} (on the Turing degrees) to remove the possibility of intermediate degrees.
It is reasonable to expect that the Shore-Slaman Join Theorem \ref{thm:shoreslamanpolish} on the continuous degrees provides a generalization of Slaman's Theorem on the Borel Martin Conjecture for continuous degrees.
This expectation is partially correct.
We will give an affirmative solution to the Borel Martin Conjecture II for the $(\leq_T,\leq_M^\ast)$-preserving functions from Cantor space to the ideals in Hilbert cube.
Surprisingly, non-total continuous degrees completely disappear in this situation.
This suggests that there is {\em no} degree-invariant hyperarithmetical method of constructing a non-total continuous degree from a total oracle.

\begin{theorem}\label{thm:positive-Martin}
The Martin order $\leq_\mathbf{m}$ prewellorders the set of $(\leq_T,\leq_M^\ast)$-preserving Borel functions $f:2^\omega\to([0,1]^\omega)^\ast$ which are increasing on a cone.

Indeed,  for a $(\leq_T,\leq_M^\ast)$-preserving Borel function $f:2^\omega\to([0,1]^\omega)^\ast$ which is increasing on a cone, one of the following holds.
\begin{enumerate}
\item There is a countable ordinal $\alpha$ such that $f$ is Martin equivalent to the following map:
\[x\mapsto\{y\in 2^\omega:y\mbox{ is recursive in }x^{(\beta)}\mbox{ for some $\beta<\alpha$}\}.\]
\item $f$ is Martin equivalent to $x\mapsto 2^\omega$.
\end{enumerate}
\end{theorem}

\begin{proof}
We follow the argument in Slaman \cite[Theorem 3.1]{Slaman05}.
By our assumption that $f$ is Borel, there is an oracle $\ep$ such that the relation $y\in f(x)$ is $\Delta^1_1(\ep)$.

First assume that there is $c\geq_T\ep$ such that $f(x)$ does not include ${\rm HYP}(x)$ for all $x\geq_Tc$, where ${\rm HYP}(x)$ is the set of all $y\in [0,1]^\omega$ hyperarithmetical in $x$, that is, those $y\leq_Mx^{(\alpha)}$ for some $\alpha<\omega_1^x$.
Then, since $f(x)$ is an ideal, for every $x\geq_Tc$, there is $\alpha<\omega_1^x$ such that $x^{(\alpha)}\not\in f(x)$.
By Martin's lemma, there is a pointed perfect tree $T\subseteq 2^\omega$ such that there exists a countable ordinal $\alpha_0<\omega_1$ such that for all $x\in[T]$, we have $x^{(\alpha_0)}\not\in f(x)$ and $x^{(\beta)}\in f(x)$ for all $\beta<\alpha_0$.
Since $f$ is $(\equiv_T,\equiv_M^\ast)$-invariant, there is $d\geq_Tc$ (that is, $d$ is a base of the pointed tree $T$) such that the above property holds for all $x\geq_Td$.

As in Slaman \cite{Slaman05}, we claim that for all $x\geq_T d$, if $y\in f(x)$ then $y\leq_Mx^{(\beta)}$ holds for some $\beta<\alpha_0$.
Otherwise, there are $x\geq_Td$ and $y\in f(x)$ such that $y\not\leq_Mx^{(\beta)}$ for all $\beta<\alpha_0$.
By the Shore-Slaman Join Theorem \ref{thm:shoreslamanpolish} for the continuous degrees, there is $G\in 2^\omega$ such that $G\geq_Tx$ and $G^{(\alpha_0)}\leq_MG\oplus y$.
Since $f$ is $(\leq_T,\leq_M^\ast)$-preserving, $x\leq_TG$ implies $y\in f(x)\subseteq f(G)$.
Since $f$ is increasing on a cone, $G\in f(G)$.
Hence, $G^{(\alpha_0)}\leq_MG\oplus y\in f(G)$ since $f(G)$ forms an ideal in the $M$-degrees.
However, we must have $G^{(\alpha_0)}\not\in f(G)$, which is a contradiction.
Consequently, if $x\geq_Td$, then $y\in f(x)$ if and only if $y\leq_Mx^{(\beta)}$ for some $\beta<\alpha_0$.

Next we assume that for any $c\geq_T\ep$, there is $x\geq_Tc$ such that $f(x)$ includes ${\rm HYP}(x)$.
Note that (a copy of) $2^\omega$ is $\Pi^0_1$ in $[0,1]^\omega$, and therefore $f(x)\cap 2^\omega$ is still $\Delta^1_1(x)$ while ${\rm HYP}(x)\cap 2^\omega$ is proper $\Pi^1_1(x)$.
Therefore, there is $y\in f(x)\cap 2^\omega$ such that $y$ is not hyperarithmetical in $x$.
Since $x$ and $y$ are total, as in Slaman \cite{Slaman05}, by the Woodin Join Theorem, there is $g\in 2^\omega$ such that $x\leq_Mg$ and $\mathcal{O}^g\equiv_Tg\oplus y$, where $\mathcal{O}^g$ is a complete $\Pi^1_1(g)$ subset of $\omega$.
Then $x\leq_Tg$ implies $y\in f(x)\subseteq f(g)$ since $f$ is $(\leq_T,\leq_M^\ast)$-preserving, and we also have $g\in f(g)$ since $f$ is increasing on a cone.
Then, $\mathcal{O}^g\equiv_T g\oplus y\in f(g)$ since $f(g)$ forms an ideal.
We now consider the set $[0,1]^\omega\setminus f(g)=\{p\in [0,1]^\omega:p\not\in f(g)\}$.
This set is $\Delta^1_1(g)$, so if it is nonempty, then by the Kleene Basis Theorem (see \cite{SacksBook}) it has an $\mathcal{O}^g$-recursive element since recall that $[0,1]^\omega$ is $\Delta^1_1$-isomorphic to Baire space $\baire$.
However, $f(g)$ includes all $\mathcal{O}^g$-recursive elements in $[0,1]^\omega$ since $f(g)$ forms an ideal.
Thus, $[0,1]^\omega\setminus f(g)$ must be empty.
Consequently, we obtain that for any $c\geq_T\ep$, there is $g\geq_Tc$ such that $f(g)=[0,1]^\omega$ as desired.
\end{proof}

\begin{corollary}\label{cor:Martin-positive}
Every $(\leq_T,\leq_M)$-preserving Borel function $f:2^\omega\to[0,1]^\omega$ which is increasing on a cone is the $\alpha^{\text{th}}$ Turing jump for some $\alpha<\omega_1$ up to Martin equivalence.
\end{corollary}

We notice that the Borel determinacy (Martin's pointed tree lemma for Borel sets) plays a key role in the proof of the above theorem.
Given a preorder $(P,\leq_P)$, we say that {\em the (Borel) $\leq_P$-determinacy holds} if for any (Borel) partition $P=P_0\cup P_1$, whenever for each $i<2$, $P_i$ is $\equiv_P$-invariant (in the sense that $q\equiv_Pp\in P_i$ implies $q\in P_i$), there is $i<2$ such that $P_i$ contains a $\leq_P$-cone, that is, there is $c\in P$ such that $p\in P_i$ for all $p\geq_P c$ (see also Martin \cite{Martin68}).
It is well-known that the axiom of (Borel) determinacy implies the (Borel) $\leq_T$-determinacy.
Of course, this is not generally true for a preorder $\leq_P$.

We assume that a recursively presented metric space $\xx$ is preordered by $\leq_\xx:=\leq_M\rs \xx$, the restriction of the $M$-reducibility to the space $\xx$.
We say that a space $\xx$ is {\em rich} if there is an effective embedding of Cantor space $2^\omega$ into $\xx$, so that $\xx$ has a $\Pi^0_1$ copy of $2^\omega$, and $\leq_\xx$ includes $\leq_T$.
This assumption is not very restrictive because it is not hard to see that if $\xx$ is uncountable Polish space, there is an oracle $\ep\in\baire$ such that $\leq_\xx^\ep$ is rich, where for a substructure $\leq_\xx$ of $\leq_M$, we define $x\leq_\xx^\ep y$ as $x\oplus\ep\leq_My\oplus\ep$.

\begin{proposition}
Let $\xx$ be a rich Polish space.
Then, the Borel $\leq_{\xx}$-determinacy holds if and only if $\xx$ has transfinite inductive dimension.
\end{proposition}

\begin{proof}
Kihara-Pauly \cite{KihPau} showed that a Polish space $\xx$ is transfinite dimensional if and only if the order $\leq_\xx^\ep$ is a substructure of $\leq_T^\ep$ relative to an oracle $\ep\in\baire$.
In particular, there is no non-total degree $x\in\xx$ above $\ep$.
Hence, if $\xx$ is transfinite dimensional, the $\leq_T$-determinacy implies the $\leq_{\xx}$-determinacy.
Conversely, if $\xx$ is not transfinite dimensional, then for any $\ep\in\baire$, there is $x_\ep\in\xx$ such that $x_\ep\oplus \ep$ is not total.
By Miller's Lemma \ref{lem:nonsplitting}, this implies $x_\ep\geq_M\ep$.
It is clear that any point $x\in\xx$ is $\leq_M$-bounded by a total degree.
Since $\xx$ is rich, $2^\omega$ is effectively embedded into $\xx$.
Then, let $C$ be the $\equiv_\xx$-saturation of such an embedded image of $2^\omega$, that is, the set of all points in $\xx$ having total degrees.
Clearly, $C$ is $\equiv_\xx$-invariant.
By the above argument, both $C$ and $\xx\setminus C$ are cofinal in $\leq_\xx$.
Therefore, the $\leq_\xx$-determinacy fails.
Note that this argument indeed shows that the Borel $\leq_\xx$-determinacy fails since $C$ is Borel.
\end{proof}

The role of the $\leq_P$-determinacy is to remove the difference between ``{\em cofinally many}'' and ``{\em on a cone}''.
Under the failure of the $\leq_P$-determinacy, the situation dramatically changes.
Recall from Theorem \ref{thm:positive-Martin} that the Borel Martin Conjecture II for $(\leq_T,\leq_M^\ast)$ is solved affirmatively.
Surprisingly, however, Slaman's Theorem fails if we allow continuous degrees in our domain space, that is, the {\em Martin Conjecture II is false} for the ideal completion of the continuous degrees.

\begin{theorem}\label{thm:failure-Martin}
The Martin order $\leq_\mathbf{m}$ is not a linear order on the set of $(\leq_M,\leq_M^\ast)$-preserving Borel functions $f:[0,1]^\omega\to([0,1]^\omega)^\ast$ which are increasing on a cone.
\end{theorem}

\begin{proof}
For each countable ordinal $\alpha<\omega_1$, we define $I^\alpha(x)$ for a point $x\in[0,1]^\omega$ with $\alpha<\omega_1^x$ as follows:
\[I^\alpha(x)=\{z\in [0,1]^\omega:(\exists y\in 2^\omega)\;y\leq_Mx \mbox{ and } z\leq_My^{(\alpha)}\oplus x\}.\]

We first claim that $I^\alpha(x)$ forms an ideal.
Given $z_0,z_1\in I^\alpha(x)$, there is $y_0,y_1\in 2^{\omega}$ such that $y_0,y_1\leq_Mx$ and $z_i\leq_M y_i^{(\alpha)}\oplus x$ for each $i<2$.
Then we have $y_0\oplus y_1\leq_Mx$ and $z_0\oplus z_1\leq_M y_0^{(\alpha)}\oplus y_1^{(\alpha)}\oplus x\leq_M(y_0\oplus y_1)^{(\alpha)}\oplus x$.
Therefore, $z_0\oplus z_1\in I^\alpha(x)$.
It is also easy to see that $I^\alpha$ is $(\leq_M,\leq_M^\ast)$-preserving.
We also see that $I^\alpha$ is increasing on a cone since $x^\ast=\{y\in[0,1]^\omega:y\leq_Mx\}\subseteq I^\alpha(x)$.

We next check that $I^\alpha:{[0,1]^\omega}\to ([0,1]^\omega)^\ast$ is Borel.
Recall from Example \ref{exa:rep-Hilbert-cube} that we have defined a tree $\mathbf{H}$ of names of points in Hilbert cube.
Then, $y\leq_Mx$ is equivalent to the existence of an index $e\in\omega$ such that for any $n\in\omega$ and any sufficiently large $l\in\omega$, if $\sigma\in\mathbf{H}$ is a string of length $l$ and $x\in B_\sigma$, then $\Phi_e^\sigma(n)$ is defined.
This condition is arithmetical.
We can also replace the existence of $y\leq_Mx$ with the existence of an index $e\in\omega$ computing $y$ from $x$.
Consequently, $I^\alpha$ is $\Delta^1_1$ relative to an oracle computing $\alpha$.
Therefore, $I^\alpha$ is Borel.

We claim that (the principal ideal generated by) the first jump $J^\ast:x\mapsto\{y\in[0,1]^\omega:y\leq_MJ_{[0,1]^\omega}(x)\}$ is not Martin-below $I^\alpha$ for any $\alpha<\omega_1$, and $I^\alpha$ is not Martin-below (the principal ideal generated by) the $\beta^{\text{th}}$ jump $J^{(\beta)\ast}$ for all $\beta<\alpha$.
The latter is obvious because $J^{(\beta)}(x)^\ast\subsetneq J^{(\alpha)}(x)^\ast=I^\alpha(x)$ for any total $x\in[0,1]^\omega$, and total degrees are cofinal in the continuous degrees.
To see $J^\ast\not\leq_\mathbf{m}I^\alpha$, recall from Miller \cite[Theorem 9.3]{miller2} that every Scott ideal is realized as the lower Turing cone of a point in $[0,1]^\omega$.
Now, for any $c\in 2^\omega$ with $\alpha<\omega_1^c$, we consider the Turing ideal $I_c=\{y\in 2^\omega:(\exists n\in\omega)\;y\leq_Tc^{(\alpha\cdot n)}\}$.
Clearly, $I_c$ forms a Scott ideal, and indeed, $I_c$ is closed under the $\alpha^{\text{th}}$ Turing jump, that is, for every $y\in I_c$, $y^{(\alpha)}\in I_c$.
By the above mentioned result by Miller \cite[Theorem 9.3]{miller2}, there is $x_c\in[0,1]^\omega$ such that for any $y\in 2^\omega$, $y\leq_Mx_c$ if and only if $y\in I_c$.
So, $c\leq_Mx_c$, and if $y\leq_Mx_c$ then $y^{(\alpha)}\leq_Mx_c$; hence $y^{(\alpha)}\oplus x_c\leq_Mx_c$.
This implies that $I^\alpha(x_c)=x_c^\ast\subsetneq J^\ast(x_c)$.
This concludes $J^\ast\not\leq_\mathbf{m}I^\alpha$.
In particular, $J^\ast$ and $I^2$ are incomparable with respect to the Martin order $\leq_\mathbf{m}$.
\end{proof}

A consequence of Slaman's result is that the relation $\leq_\mathbf{m}$ is a linear order on the set of $(\leq_{T},\leq_{T}^\ast)$-preserving Borel functions $f: \cantor\to\cantor$, which are increasing on a cone. This implies the same assertion for $(\leq_{T},\leq_{T})$. The preceding result refutes the $(\leq_{M},\leq_{M}^\ast)$ case, but still this has no implications to the non-$\ast$ case.

\begin{myquestion}
Does the Martin order $\leq_\mathbf{m}$ give a linear order on the set of $(\leq_{M},\leq_{M})$-preserving Borel functions $f:[0,1]^\omega\to[0,1]^\omega$ which are increasing on a cone?
\end{myquestion}

Steel \cite{Steel82} has shown that the Martin conjecture is true for {\em uniformly $(\equiv_T,\equiv_T)$-invariant} functions $f:2^\omega\to 2^\omega$.
It is natural to ask whether Steel's theorem holds in other Polish spaces.
Unfortunately, in proper infinite dimensional spaces, Steel's game-theoretic argument is useless to decide successor $\leq_\mathbf{m}$-ranks even if we restrict our attention to {\em uniformly $(\leq_T,\leq_M)$-preserving} Borel functions $f:2^\omega\to [0,1]^\omega$ which are increasing on a cone.
Nevertheless, Corollary \ref{cor:Martin-positive} tells us the complete description of all $\leq_\mathbf{m}$-ranks of such functions even in the non-uniform case, which exhibits the strength of our positive result.
However, of course, our current results have no implication to non-Borel functions.

\begin{myquestion}
Is it true that all successor $\leq_\mathbf{m}$-ranks of uniformly $(\leq_T,\leq_M)$-preserving functions $f:2^\omega\to[0,1]^\omega$ are given by the Turing jump $TJ$?
\end{myquestion}

\section{Continuous transition in the codes}

\label{section continuous transition in the codes}

Kihara \cite[Problem 2.13]{kihara_decomposing_Borel_functions_using_the_Shore_Slaman_join_theorem} asked whether the condition $f^{-1}:\bolds^0_m\subseteq\bolds^0_n$ always involves a continuous witness.
The motivation behind this question is the following consequence of Kihara's technique \cite[Lemma 2.9]{kihara_decomposing_Borel_functions_using_the_Shore_Slaman_join_theorem}. Let $m\leq n<2m-1$ and assume that $f$ satisfies the following two conditions:
\begin{enumerate}
\item[(a)] $f$ is decomposable into countably many $\bolds^0_{n-m+1}$-measurable functions,
\item[(b)] and $f^{-1}:\bolds^0_m\subseteq\bolds^0_n$ holds continuous-uniformly in the codes with respect to {\em some} good universal systems,
\end{enumerate}
Then the witnessing sets in (a) above can be chosen to be $\boldd^0_n$, \ie we have the best possible decomposition.

It is a well-known fact in descriptive set theory that Borel-measurable functions are turned into continuous ones without affecting the Borel-structure of the underlying spaces. It is perhaps tempting to assume that a Borel-transition in the codes is essentially a continuous one under a modified universal system, and therefore the Decomposability Conjecture can be solved at least in the cases $m \leq n \leq 2m-1$ using the techniques of Kihara \cite[Lemma 2.9]{kihara_decomposing_Borel_functions_using_the_Shore_Slaman_join_theorem}. Here we show that this assertion is partially correct.

\begin{theorem}
\label{corollary from continuous transition}
Suppose that \ca{X} and \ca{Y} are Polish spaces and that $f: \ca{X} \to \ca{Y}$ satisfies $f^{-1}\bolds^0_m \subseteq \bolds^0_n$ for some $m, n \geq 1$. Then there exists a universal system $(J_m^\ca{Z})_{\ca{Z}}$ for $\bolds^0_m$ such that condition $f^{-1}\bolds^0_m \subseteq \bolds^0_n$ holds continuous-uniformly in the codes with respect to $(J_m^{\ca{Y}},\univ_n^{\ca{X}})$.
\end{theorem}

To show Theorem \ref{corollary from continuous transition}, we need the following lemma:

\begin{lemma}
\label{theorem continuous transition}
Suppose that \ca{X} and \ca{Y} are Polish spaces and that $f: \ca{X} \to \ca{Y}$ satisfies $f^{-1}\bolds^0_m \subseteq \bolds^0_n$ for some $m, n \geq 1$. Then there exists a continuous surjection $g_1: \baire \surj \baire$ and a $\bolds^0_n$ set $H \subseteq \baire \times \ca{X}$ such that
\[
\univ_m^{\ca{Y}}(g_1(\ep),f(x)) \iff H(\ep,x)
\]
for all $(\ep,x) \in \baire \times \om \times \ca{X}$.
\end{lemma}

\begin{proof}
From Theorem \ref{theorem Borel transition in the codes} there exists a Borel-measurable function $\tau: \baire \to \baire$ such that
\[
\univ^{\ca{Y}}_m(\beta,f(x)) \iff \univ^{\ca{X}}_n(\tau(\beta),x)
\]
for all $(\beta,x) \in \baire \times \ca{X}$.
We consider basic neighborhoods $N_s$ of the Baire space and we define $Q \subseteq \baire \times \om$ by
\[
Q(\beta,s) \iff \tau(\beta) \in N_s.
\]
Clearly $Q$ is Borel. From (the relativized version of) 4A.7 in \cite{yiannis_dst} the sets $Q$ and $\neg Q$ are the injective recursive images of closed subsets of the Baire space. Since the graph of a continuous function is a closed set there exist closed sets $C_0$, $C_1 \subseteq \baire \times \om \times \baire$ such that
\begin{align*}
Q(\beta,s) \iff& (\exists \alpha)C_0(\beta,s,\alpha)\iff(\exists ! \alpha)C_0(\beta,s,\alpha),\\
\neg Q(\beta,s) \iff& (\exists \alpha)C_1(\beta,s,\alpha)\iff(\exists ! \alpha)C_1(\beta,s,\alpha),
\end{align*}
for all $\beta$ and $s$, where $\exists !$ stands for ``there exists a unique".

Now we define $F \subseteq \baire \times \baire \times \cantor$ by
\begin{align*}
F(\beta,\gamma,\delta) \iff& (\forall s)\big \{ \ [\delta(s) = 0 \ \& \ C_0(\beta,s,(\gamma)_s)] \ \vee \ [\delta(s) = 1 \ \& \ C_1(\beta,s,(\gamma)_s)] \ \big \}\\
                                        & \ \& \ \gamma(t) = 0 \ \textrm{for all $t$ which do are not of the form $\langle s,i \rangle$}.
\end{align*}
It is clear that $F$ is a closed set and that for all $\beta \in \baire$ there exists exactly one pair $(\gamma,\delta) \in \baire \times \cantor$ such that $F(\beta,\gamma,\delta)$.
It is also clear that $F$ is a Polish space and so there exists a continuous surjection
\[
g: \baire \surj F.
\]
We define the functions
\[
h: F \to \baire: h(\beta,\gamma,\delta) = \tau(\beta), \quad \textrm{and} \quad \pi: \baire \to \baire: \pi = h \circ g,
\]
so that the following diagram commutes
\begin{figure}[h!]
\begin{picture}(300,25)(0,5)
\put(100,25){\scalebox{1}{$\baire \ \ \stackrel{g}{\longrightarrow \kern -.8em \rightarrow} \ \ F \ \ \stackrel{h}{\longrightarrow} \ \ \baire$.}}
\qbezier(110,20)(145,8)(182,20)
\put(184,21.5){\vector(4,3){1}}
\put(145,8){${}_\pi$}
\end{picture}
\end{figure}

We claim that the function $h$ is continuous. From this it follows that the function $\pi$ is continuous as well. To see the former we notice that
\begin{align*}
h(\beta,\gamma,\delta) \in N_s \iff& \ \ \ \delta(s) = 0 \ \& \ C_0(\beta,s,(\gamma)_s)\\
                               \iff& \neg [ \delta(s) = 1 \ \& \ C_1(\beta,s,(\gamma)_s) ]
\end{align*}
for all $(\beta,\gamma,\delta) \in F$ and all $s \in \om$. The preceding equivalences show that the set $h^{-1}[N_s]$ is in fact clopen for all $s$.\footnote{It is a well-known result of classical descriptive set theory that for every Borel measurable function $f: \ca{X} \to \ca{Y}$ (\ca{X}, \ca{Y} Polish) there exists a zero-dimensional Polish topology on \ca{X} which refines the original topology, has the same Borel sets and the function $f$ is continuous with respect to this new topology, \cf \cite{kechris_classical_dst} Section 13.A. Here we gave in essence an effective proof of this fact for the case of the function $\tau$. The new topology on $\baire$ is the one induced by the bijection \mbox{$p: \beta \in \baire \mapsto (\beta,\gamma,\delta) \in F$}. This topology refines the original one because the function \mbox{$(\beta,\gamma,\delta) \mapsto \beta$} is continuous, and the Borel structure remains the same because the function $p$ is Borel measurable. Moreover the function $\tau$ is continuous under this new topology since $\tau = h \circ p$ and $h$ is continuous. For more information on turning Borel sets into clopen in an effective way the reader can refer to \cite{gregoriades_turning_borel_sets_into_clopen_sets_effectively}.}

We define the function
\[
g_1 = \pr_1 \circ g: \baire \to \baire,
\]
where $\pr_1$ is the projection $(\beta,\gamma,\delta) \in F \mapsto \beta \in \baire$. It is clear that $g_1$ is continuous, surjective and that $\pi = \tau \circ g_1$. Thus the preceding diagram is refined to the following (commutative) diagram

\begin{figure}[h!]
\begin{picture}(300,55)(0,-5)
\put(70,25){\scalebox{1}{$\baire \ \ \stackrel{g}{\longrightarrow \kern -.8em \rightarrow} \ \ F \ \ \stackrel{\hspace*{-1mm}\pr_1}{\scalebox{1}{\bij}} \ \ \baire \ \ \stackrel{\tau}{\longrightarrow} \ \ \baire$}.}
\qbezier(80,20)(115,8)(152,19)
\put(154,20.5){\vector(4,3){1}}
\put(130,10){${}_{g_1}$}
\qbezier(80,17)(135,-13)(194,17.5)
\put(196,19){\vector(4,3){1}}
\put(150,-3){$_\pi$}
\qbezier(118,37)(158,52)(198,38)
\put(199,37){\vector(4,-3){1}}
\put(158,49){${}_h$}
\end{picture}
\end{figure}

Now we define $H \subseteq \baire \times \ca{X}$ as follows
\[
H(\ep,x) \iff \univ^{\ca{Y}}_n(\pi(\ep),x).
\]
Since $\pi$ is continuous and $\univ^{\om \times \ca{X}}_n$ is a $\bolds^0_n$ set it follows that $H$ is a $\bolds^0_n$ set as well. Finally we compute
\begin{align*}
\univ^{\ca{Y}}_m((g_1(\ep),f(x))    \iff& \univ^{\om \times \ca{X}}_n(\tau(g_1(\ep)),x)\\
                             \iff& H(\ep,x)
\end{align*}
for all $(\ep,x) \in \baire \times \ca{X}$, and the proof is complete.
\end{proof}

\begin{proof}[Proof of Theorem \ref{corollary from continuous transition}]
From Lemma \ref{theorem continuous transition} there exists a continuous surjection $g_1: \baire \surj \baire$ and a $\bolds^0_n$ set $H \subseteq \baire \times \om \times \ca{X}$ such that
\[
\univ_m^{\om \times \ca{Y}}(g_1(\ep),f(x)) \iff H(\ep,x)
\]
for all $(\ep,x) \in \baire \times \ca{X}$.
We define
\[
J_m^\ca{Z}(\ep,z) \iff \univ_m^\ca{Z}(g_1(\ep),z)
\]
for all $(\ep,z) \in \baire \times \ca{Z}$ and all Polish spaces \ca{Z}. Since the function $g_1$ is continuous, we have that $J_m^\ca{Z}$ is a $\bolds^0_m$ set, and, since $g_1$ is surjective, it follows that $J_m^\ca{Z}$ parametrizes $\bolds^0_m \upharpoonright \ca{Z}$. Thus $(J_m^\ca{Z})_{\ca{Z}}$ is a universal system for $\bolds^0_m$.
Since the set $H \subseteq \baire \times \om \times \ca{X}$ is $\bolds^0_n$ there exists some $\beta \in \baire$ such that $H$ is the $\beta$-section of $\univ_n^{\ca{Z}}$, where $\ca{Z} = \baire \times \ca{X}$.
In particular, the $\ep$-section of $H$ is the $(\beta,\ep)$-section of $\univ_n^{\ca{Z}}$.
Therefore,
\[
J_m^{\ca{Y}}(\ep,f(x)) \iff H(\ep,x) \iff \univ_n^{\ca{Z}}(\beta,\ep,x) \iff \univ_n^{\ca{X}}(\pi(\ep),x)
\]
for all $(\ep,x) \in \ca{Z}$, where $\pi(\ep)$ is a code of the $(\beta,\ep)$-section of $\univ_n^{\ca{Z}}$.
Note that such a $\pi$ can be chosen as a continuous function.
\end{proof}

\begin{remark}\normalfont
\label{remark about continuous transition}
The universal system $(J_m^\ca{Z})_{\ca{Z}}$ of the preceding theorem depends on the choice of $f$. Moreover it does not seem to be a good universal system, unless the uniformity function in Theorem \ref{theorem Borel transition in the codes} is in fact continuous. The reason we claim this lies in the following equivalences. It is clear from the definition that
\begin{align*}
J_m^{\ca{X} \times \ca{Z}}(\ep,x,z) \iff& \univ_m^{\ca{X} \times \ca{Z}}(g_1(\ep),x,z)\\
                                    \iff& \univ_m^{\ca{Z}}(S(g_1(\ep),x),z)
\end{align*}
for some continuous function $S$. In order to infer that $(J^\ca{Z}_m)_{\ca{Z}}$ is a good universal system we would need a continuous function $\varphi: \baire \to \baire$ such that $g_1 \circ \varphi = {\rm id }$, \ie $g_1$ should have a \emph{continuous inverse}. In this case we would define the function
\[
S_{J,_m}^{\ca{X},\ca{Z}} \equiv S_J: \baire \times \ca{X} \to \baire: S_J(\ep,x) = \varphi(S(g_1(\ep),x))
\]
and we would have
\begin{align*}
J_m^\ca{\ca{X} \times \ca{Z}}(\ep,x,z) \iff& \univ_m^{\ca{Z}}(S(g_1(\ep),x),z)\\
                                         \iff& \univ_m^{\ca{Z}}(g_1(\varphi(S(g_1(\ep),x))),z)\\
                                         \iff& J_m^{\ca{Z}}(\varphi(S(g_1(\ep),x))),z)\\
                                         \iff& J_m^{\ca{Z}}(S_J(\ep,x),z).
\end{align*}
However the function $g_1$ does not seem to admit a continuous inverse, for otherwise by going back to the proof of Theorem \ref{theorem continuous transition} the bijection $(\beta,\gamma,\delta) \in F \mapsto \beta$ would also have a continuous inverse. This would imply that the uniformity function in Theorem \ref{theorem Borel transition in the codes} is continuous.

This remark says essentially that the idea of turning Borel sets into clopen does not seem to help in the Decomposability Conjecture even for the cases $2 \leq n \leq 2m-1$.
\end{remark}

\section{A plan for the full solution}\label{sec:last}

We conclude this article by proposing a recursion-theoretic strategy for solving the Decomposability Conjecture. As we point out at the end of Section \ref{section Borel transition in the codes} it is easy to verify that a function $f: \ca{A} \to \ca{Y}$ is decomposable to $\bolds^0_{k+1}$-measurable functions if and only if there is some $w \in \cantor$ such that for all $x \in \mathcal{A}$ is holds $f(x) \leq_M (x \oplus w)^{(k)}$. The latter condition gives us an upper bound for the complexity of the decomposing sets, but this is not necessarily the best possible.\footnote{In the case $m=n=2$ we consider the Lebesgue function $L: \cantor \to [0,1]$ from the Solecki Dichotomy (for Baire-$1$ functions on closed domains) \cite{Solecki98}. It is easy to arrange the definition of $L$ so that $L(\alpha)\leq_M \alpha$ for all $\alpha \in \cantor$. On the other hand $L$ is not decomposable to continuous functions on $\boldd^0_2$ domains.}

We now give the analoguous remark taking into consideration the complexity of the decomposing sets. Given $x,y$ with $x \leq_M y$ let us say that $e$ \emph{realizes} that $x \leq_M y$ if $x = \Phi_e(y)$.

\begin{proposition}
\label{proposition decomposability realizer}
Let \ca{X}, $\ca{Y}$ be recursively presented metric spaces, \ca{A} be any subset of \ca{X} and $f: \ca{A} \to \ca{Y}$ be a function.
Then, the following are equivalent for $n\geq k+2$:
\begin{enumerate}
\item The function $f$ is decomposable to $\bolds^0_{k+1}$-measurable functions on $\boldd^0_n$ sets.
\item There is $z \in \cantor$ such that $f(x) \leq_M (x \oplus z)\tpower{k}$ holds for all $x \in \ca{A}$, and moreover, the preceding condition is realized by a $\bolds^0_n$-measurable function $u: \ca{A} \to \om$.
\end{enumerate}
\end{proposition}

\begin{proof}
If $f(x) \leq_M (x \oplus z)\tpower{k}$ for all $x \in \ca{A}$ and the preceding condition is realized by a $\bolds^0_n$-measurable function $u: \ca{A} \to \om$, we take $\ca{B}_e : = u^{-1}[\{e\}]$ where $e \in u[\ca{A}]$. This settles the right-to-left-hand direction. (Notice that this holds also when $n < k+2$.)

Assume now that $f$ is decomposable to $\bolds^0_{k+1}$-measurable functions on a partition $(X_i)_{i \in \om}$ of $\boldd^0_n$ subsets of $\ca{A}$. Let $f_i$ be the restriction of $f$ on $X_i$. Since each $f_i$ is $\bolds^0_{k+1}$-measurable there is some $w \in \cantor$ such that for all $i \in \om$ and all $x \in X_i$ we have that $f_i(x) \leq_M (x \oplus w)\tpower{k}$. Define
\[
B_{i,e}:= \set{x \in X_i}{f_i(x) = \Phi_{e}^{(x \oplus w)\tpower{k}}}.
\]
Clearly $X_i = \bigcup_{e} B_{i,e}$. Moreover each $B_{i,e}$ is a $\boldp^0_{k+1}$ subset of $X_i$. Since $k+1 < n$ it follows that the $B_{i,e}$'s are $\boldd^0_n$ subsets of $X_i$.

We now consider the differences $C_{i,0} = B_{i,0}$, $C_{i,e+1} = B_{i,e+1} \setminus \bigcup_{k \leq e} B_{i,e}$ so that the for all $i$, the sets $(C_{i,e})_{e}$ are pairwise disjoint $\boldd^0_{n}$ subsets of $X_i$ and $X_i = \bigcup_{e} C_{i,e}$. Since $(X_i)_{i \in \om}$ is a $\boldd^0_n$ partition of $\ca{A}$ it follows that $(C_{i,e})_{i,e}$ is a $\boldd^0_n$ partition of $\ca{A}$ as well.

We then define $u: \ca{A} \to \om: u(x) = e$, where $x \in C_{i,e}$. Clearly $u$ is a $\bolds^0_n$-measurable realizing function.
\end{proof}

Since the existence of a $\bolds^0_n$-measurable realizing function is necessary to obtain the decomposability to $\bolds^0_{n-m+1}$-measurable functions on $\boldd^0_n$ sets ($2 \leq m \leq n$), it is natural to ask if this follows from our proof. To do this we essentially need to ask a question about the Shore-Slaman Join Theorem.

For simplicity we focus on zero-dimensional spaces.
We say that {\em $u:2^\omega\times 2^\omega\partialf\omega$ realizes the contrapositive form of the Shore-Slaman Join Theorem at $\xi+n-m$} if for all $\tilde{x},\tilde{y}\in 2^\omega$ it satisfies the following:
\[
\text{if for all $G \tgeq \tilde{x}$ we have that $G \tpower{\xi+n-m+1} \not \tleq \tilde{y} \oplus G$} \hspace*{10mm} (\ast)
\]
then $u(\tilde{x},\tilde{y})$ is defined and realizes that
\[
\text{$\tilde{y} \tleq \tilde{x} \tpower{\xi+n-m}$.}
\]

\begin{claim}\normalfont
Let $f: \ca{A} \subseteq \cantor \to \cantor$ be a function, where $\ca{A}$ is $\Sigma^1_1$, with the property $f^{-1}\bolds^0_m \subseteq \bolds^0_n$ for some $2 \leq m \leq n$.
Then, there are an oracle $z$, an ordinal $\xi$, and continuous functions $u_1,u_2,u_3$ such that for all functions $u_4$ which realize the contrapositive form of the Shore-Slaman Join Theorem, the natural $u_3(u_4(u_1(x),f(x)),u_2(x))$ realizes that $f(x)\leq_T((x\oplus z)^{(\xi)})^{(n-m)}$ for all $x\in\ca{A}$.
\end{claim}

\noindent
{\it Proof.}
From Lemma \ref{lem:turing-degree-analysis} there is some $z \in \cantor$ and an ordinal $\xi < \ckr{z}$ such that for all $x \in \ca{A}$ we have that $(f(x) \oplus g)\tpower{m} \tleq (x \oplus g\tpower{\xi})\tpower{n}$ for all $g \tgeq z$. From the Friedberg Inversion for every $x \in \ca{A}$ there is some $D_x \subseteq \om$ such that $D_x\tpower{\xi} \teq x \oplus z\tpower{\xi}$. In particular we have that $D_x\tpower{\xi+n-m} \tleq (x \oplus z\tpower{\xi})\tpower{n-m}$.

We utilize the following results from recursion theory:
\begin{enumerate}
\item[(a)] There are continuous functions $u_1: \cantor \to \cantor$ and $u_2: \cantor \to \om$ such that for all $x \in \cantor$ we have that $u_1(x)\tpower{\xi} \teq x \oplus z\tpower{\xi}$, $u_1(x) \tgeq z$, and $u_2(x)$ realizes that $u_1(x)\tpower{\xi+n-m} \tleq (x \oplus z\tpower{\xi})\tpower{n-m}$. This follows from the proof of the Friedberg Inversion Theorem.

\item[(b)] There is a recursive partial function $u_3: \om^2 \rightharpoonup \om$ such that for all $A, B, C \subseteq \om$ if $e_1$ realizes that $A \tleq B$ and $e_2$ realizes that $B \tleq C$ then $u_3(e_1,e_2)$ realizes that $A \tleq C$.
\end{enumerate}

Going back to the proof of the Cancellation Lemma we see that the condition $(f(x) \oplus g)\tpower{m} \tleq (x \oplus g\tpower{\xi})\tpower{n}$ for all $g \tgeq z$, is used to show that the condition $(\ast)$ above holds for $\tilde{y} = f(x)$ and $\tilde{x} = u_1(x)$.
Therefore $u_4(u_1(x),f(x))$ realizes that $f(x) \tleq u_1(x) \tpower{\xi+n-m}$. Recall that $u_2(x)$ realizes that $u_1(x) \tpower{\xi+n-m} \tleq (x \oplus z\tpower{\xi})\tpower{n-m}$. It follows that $u_3(u_4(u_1(x),f(x)),u_2(x))$ realizes that $f(x) \tleq (x \oplus z\tpower{\xi})\tpower{n-m}$.
\qed(Claim)

\medskip

Let $f$ and $u_i$, $i=1,2,3,4$ be as in the preceding claim. Using the continuity of $u_1,u_2, u_3$, it follows that if the function $x \mapsto u_4(f(x),u_1(x))$ is $\bolds^0_n$-measurable then $f$ is decomposed to $\bolds^0_{n-m+1}$-measurable functions on $\boldd^0_n$ sets, \ie the function $f$ would satisfy the conclusion of the Decomposability Conjecture. Since $f$ is $\bolds^0_n$-measurable it is enough to have that $u_4$ is continuous. The latter though does not seem to us possible, however we consider it likely that $u_4(\tilde{x},\tilde{y})$ can be chosen to be recursive in $\tilde{y} \oplus \tilde{x}\tpower{\xi+n-m+1}$. We will actually ask for a small variant of this.

\begin{myquestion}\normalfont
\label{question uniform Shore-Slaman}
Let $k \in \om$. Does there exist a \emph{continuous} partial function $\tau$ such that the function
\[u_\tau: \Dom(\tau) \to \om: (\tilde{x},\tilde{y}) \mapsto \tau(\tilde{x}\tpower{\xi+k+1},\tilde{y})\]
realizes the contrapositive form of the Shore-Slaman Join Theorem at $\xi+k$?
\end{myquestion}

\begin{claim}\normalfont
If Question \ref{question uniform Shore-Slaman} has an affirmative answer then the Decomposability Conjecture is true (for functions in zero-dimensional spaces).
\end{claim}

\noindent
{\it Proof.}
Let $\ca{A}$ be $\Sigma^1_1$ and $f: \ca{A} \to \cantor$ be such that $f^{-1}\bolds^0_m \subseteq \bolds^0_n$. Find then some $z \in \cantor$ and $\xi < \ckr{z}$ such that $(f(x) \oplus g)\tpower{m} \tleq (x \oplus g\tpower{\xi})\tpower{n}$ for all $g \tgeq z$ and all $x \in \ca{A}$. We consider continuous functions $u_1$, $u_2$ and $u_3$ as above and a function $\tau$ which answers Question \ref{question uniform Shore-Slaman} affirmatively with $k = n-m$. We let $u_4$ be the function $u_\tau$.

Then according to the preceding, the pair $(f(x),u_1(x))$ satisfies condition $(\ast)$ above for all $x \in \ca{A}$. We claim that the function $x \in \ca{A} \mapsto u_4(f(x),u_1(x))$ is $\bolds^0_n$-measurable. As explained above this is enough to ensure that the function $f$ is decomposed to $\bolds^0_{n-m+1}$-measurable functions on $\boldd^0_n$ sets.

First we fix a recursive partial function $u_5: \om \rightharpoonup \om$ such that for all $A, B \subseteq \om$ and all $e \in \om$ if $e$ realizes that $A \tleq B$ then $u_5(e)$ realizes that $A ' \tleq B'$. Let some $x \in \ca{A}$. Since $u_2(x)$ realizes that $u_1(x)\tpower{\xi+k} \tleq (x \oplus z\tpower{\xi})\tpower{k}$ we have that $u_5(u_2(x))$ realizes that $u_1(x)\tpower{\xi+k+1} \tleq (x \oplus z\tpower{\xi})\tpower{k+1}$, \ie
\[
u_1(x)\tpower{\xi+k+1} = \Phi_{u_5(u_2(x))}((x \oplus z\tpower{\xi})\tpower{k+1}).
\]
Therefore
\[
u_4(f(x),u_1(x)) = \tau(f(x),u_1(x)\tpower{\xi+k+1}) = \tau(f(x),\Phi_{u_5(u_2(x))}((x \oplus z\tpower{\xi})\tpower{k+1})).
\]
For all $e \in \om$ the (perhaps partial) function $x \mapsto \Phi_e(x \oplus z\tpower{\xi})\tpower{k+1}$ is $\bolds^0_{k+2}$-measurable. Since $k+2 = n-m+2 \leq n$ it follows that the preceding function is also $\bolds^0_n$-measurable. Using the continuity of the functions $\tau$, $u_5 \circ u_2$, and the fact that $f$ is $\bolds^0_n$-measurable it follows that the function $x \mapsto u_4(f(x),u_1(x))$ is $\bolds^0_n$-measurable as well. This finishes our argument.
\qed(Claim)

\medskip

The analogous to Question \ref{question uniform Shore-Slaman} can be asked about the continuous degrees. If the answer is affirmative then with similar arguments one would be able to prove that the Decomposability Conjecture is correct.


\subsection*{Acknowledgments.} The first named author was partially supported by the E.U. Project no: 294962 COMPUTAL, and the second named author was partially supported by a Grant-in-Aid for JSPS fellows. The third author was partially supported by the grants MOE-RG26/13 and MOE2015-T2-2-055. We would like to thank Andrew Marks, Yiannis Moschovakis, Luca Motto Ros and Arno Pauly for valuable discussions. The first named author would also like to thank Ulrich Kohlenbach for his continuing substantial support.

\bibliographystyle{plain}
\bibliography{gkbiblio}

\end{document}